\documentclass[12pt]{amsart}
\usepackage{indentfirst}
\usepackage{bm, mathrsfs, graphics,float,amssymb,amsmath,setspace,graphicx,amsthm,epstopdf,subfigure, enumerate, color, bbm}
\usepackage{thmtools, thm-restate}
\usepackage[hidelinks]{hyperref}
\usepackage[margin = 0.8in]{geometry}
\def\b{\boldsymbol}

\def\R{\mathbb{R}}

\def\p{\partial}
\def\T{\mathbb{T}}

\newcommand{\loc}{\text{loc}}

\newcommand{\pt}{\partial}
\newcommand{\eps}{\varepsilon}

\newcommand{\sij}{{ij}}
\newcommand{\ud}{\,\mathrm{d}}

\numberwithin{equation}{section}

\newtheorem{lemma}{Lemma}[section]
\newtheorem{theorem}{Theorem}
\newtheorem{proposition}{Proposition}
\newtheorem{definition}{Definition}

\newtheorem{remark}{Remark}

\begin{document}
\title[Existence and rigidity of vectorial dislocation model]{
Existence and rigidity of the vectorial Peierls-Nabarro model for dislocations in high dimensions}

\author[Y. Gao]{Yuan Gao}
\address{Department of Mathematics, Duke University, Durham, NC and Department of Mathematics, Purdue University, West Lafayette, IN}
\email{gao662@purdue.edu}

\author[J.-G. Liu]{Jian-Guo Liu}
\address{Department of Mathematics and Department of
  Physics, Duke University, Durham, NC}
\email{jliu@math.duke.edu}

\author[Z. Liu]{Zibu Liu}
\address{Department of Mathematics, Duke University, Durham, NC}
\email{zibu.liu@duke.edu}



\maketitle
\begin{abstract}
We focus on the existence and rigidity problems of the vectorial Peierls-Nabarro (PN) model for dislocations.  Under the assumption that the misfit potential on the slip plane only depends on the shear displacement along the Burgers vector, a reduced non-local scalar Ginzburg-Landau equation with an anisotropic positive (if Poisson ratio belongs to $(-1/2,1/3)$) singular kernel is derived on the slip plane. We first prove that minimizers of the  PN energy for this reduced scalar problem exist. Starting from $H^{1/2}$ regularity, we prove that these minimizers are smooth 1D profiles only depending on the shear direction, monotonically and uniformly converge to two stable states at far fields in the direction of the Burgers vector. Then a De Giorgi-type conjecture of single-variable symmetry for both minimizers and layer solutions is established. As a direct corollary, minimizers and layer solutions are unique up to translations. The proof of this De Giorgi-type conjecture relies on a delicate spectral analysis which is especially powerful for nonlocal pseudo-differential operators with strong maximal principle. All these results hold in any dimension since we work on the domain periodic in the transverse directions of the slip plane. The physical interpretation of this rigidity result is that the equilibrium dislocation on the slip plane only admits shear displacements and is a strictly monotonic 1D profile provided exclusive dependence of the misfit potential on the shear displacement. 
\end{abstract}
\section{Introduction and main results}\label{sec:introduction}

In materials science, the Peierls-Nabarro (PN) model with Poisson ratio $\nu\in[-1,1/2]$ plays a fundamental role in describing dislocations or line defects in materials \cite{anderson2017theory,lu2005peierls}. Understanding this model provides insights on designing new materials with robust performance \cite{jiang2020stochastic,ghoussoub1998conjecture,blass2015dynamics,fonseca2020motion}. However, the existence and rigidity problem regarding the vector-field PN model has not been explored.

The PN model is a nonlinear model that studies the core structure of the dislocation by  incorporating the atomistic effect in the dislocation core into the continuum elastic model.
In the PN model in three dimensions, two half-spaces separated by the slip plane of a dislocation are assumed to be linear elastic continua. Here the slip plane is assumed to be a fixed plane $\Gamma=\{(x,y,z):\ y=0\}$, where the horizontal  displacement discontinuity (known as disregistry) happens.   Without loss of generality, we assume that the  Burgers vector is $\b b=(b,0,0)$ where $b>0$. 
The magnitude of the Burgers vector represents the typical length to observe a heavily distorted region in the dislocation core. Hence it is natural to rescale all the quantities including spatial variable $x,y,z$,  the displacement vector $\b u=(u_1, u_2, u_3)$ and $\b b$ with respect to the magnitude of the Burgers vector. After rescaling,  we regard all these quantities (with same notations) as dimensionless quantities and $b=1$.

In this paper, the shear direction is referred to as the direction of the Burgers vector, i.e. the $x$ direction; the vertical direction of the slip plane is referred to as the $y$ direction and the transverse direction in the slip plane is referred to as the $z$ direction.

The PN model is a minimization problem for the total energy $E$ which is given by
\begin{equation}\label{E2.2}
E(\b u):= E_\mathrm{els}(\b u)+E_\mathrm{mis}(\b u).
\end{equation}
Here $\b u=(u_1, u_2, u_3)$ is the displacement vector. 
\eqref{E2.2} incorporates not only the elastic energy in the bulk but also the atomistic effect in the dislocation core. The elastic energy  in the two half-spaces is defined as
\begin{align}
E_{\mathrm{els}}(\b u) = \int_{\mathbb{R}^3\backslash \Gamma} \frac12  \sigma :\varepsilon \ud x\ud y \ud z,
\end{align}
where 
$\sigma : \varepsilon=\sum_{i,j=1}^3 \sigma_{ij}\varepsilon_{ij}$. Here $\varepsilon$  and $\sigma$ are the strain tensor and the stress tensor respectively, defined as
\begin{equation}\label{strainandconstitutive}
\eps_\sij=\frac{1}{2}(\pt_j u_i+\pt_i u_j),\ \sigma_\sij=2G \eps_\sij+\frac{2\nu G}{1-2\nu} \sum_{k=1}^3\eps_{kk}\delta_{ij}, \quad i,j=1,2,3.
\end{equation}
Here $\nu\in[-1,1/2]$ is the Poisson ratio and $G$  is the shear modulus.

On the slip plane, we denote the upper limit and lower limit of the displacement as
\begin{align}
u_i^+(x,z)=u_i(x,0^+,z),\ u_i^-(x,z)=u_i(x,0^-,z),\ i=1,2,3.
\end{align} 
Moreover, we assume that $u_i,\ i=1,2,3$ are subject to the following boundary conditions at the slip plane:  
\begin{equation}\label{BC}
\begin{aligned}
u_1^+(x, z)= - u_1^-(x, z),\,\,\,
u_2^+(x, z)= u_2^-(x, z),\,\,\,
u_3^+(x, z)= - u_3^-(x, z).
\end{aligned}
\end{equation}
We call \eqref{BC} the symmetric assumption. Characterizing the nonlinear atomistic interactions, the misfit energy $E_\mathrm{mis}(\b u)$ is defined as the integral of the misfit potential $\gamma:\R^2\to\R$  on the slip plane:
\begin{equation}\label{Emis}
{E_\mathrm{mis}}(\b u):=\int_{\Gamma} \gamma(u_1^+-u_1^-, u_3^+-u_3^-) \ud x \ud z =\int_{\Gamma}\gamma (2u_1^+, 2u_3^+) \ud x \ud z.
\end{equation}
The last equality is due to the symmetric assumption \eqref{BC}. Notice, $\b u$ is already dimensionless quantity so $E_\mathrm{mis}$ is well-defined. For brevity, we will omit factor 2 in \eqref{Emis} before $u_1^+$ and $u_3^+$ which makes no difference in the conclusions.

In this paper, to characterize the key property imposed by the Burgers vector, i.e., the direction of the dislocation and existence of two stable states, we assume that $\gamma$ depends only on the shear displacement in the Burgers direction, i.e.,
\begin{align}\label{assum:1d}
\gamma(u_1,u_3)=\gamma(u_1),
\end{align}
where the misfit potential is a double-well type potential, i.e. $\gamma:\R\to\R$ is a $C^2$ function satisfying
 \begin{equation}\label{condition:potential}
\gamma(x)>0\ \mathrm{if}\ {-1<x<1}, \quad \gamma(\pm 1)=0,\ \ \gamma''(\pm 1)>0.
\end{equation}
We remark that this assumption on the misfit potential also includes some other typical periodic potentials which satisfy $\gamma(x+2)=\gamma(x)$ and represent the periodic lattice structure of crystalline materials; see $\gamma_0$ in an explicit example \eqref{eq:straightsolution}.

Because the magnitude of the rescaled Burgers vector is order $1$, for convenience, we take $u_1=\pm 1$ as bi-states at far fields, i.e.,
\begin{equation}\label{assm:farendlimit}
\lim\limits_{x_1\to\pm\infty}u_1(\b x)=\pm 1,
\end{equation}
  which equivalent to the magnitude of the dislocation is assumed to be $b=4$.

{In this bi-states case \eqref{assm:farendlimit}, the total energy $E$ in \eqref{E2.2} in the whole space is always infinite. Therefore, we consider the global minimizer in the following perturbed sense. However we remark that for a dislocation loop, i.e., a disregistry with compact support instead of the bi-states far field condition, energy $E$ in \eqref{E2.2} is finite.}

\begin{definition}\label{minimizer} A function $\b u: \mathbb{R}^3 \to \mathbb{R}^3$ {satisfying \eqref{assm:farendlimit}} is called a global minimizer of $E$ defined in \eqref{E2.2} if it satisfies 
	\begin{equation}
	E(\b u + \b \varphi)-E(\b u)\geq 0
	\end{equation}
	for any perturbation $ \bm\varphi=(\varphi_1,\varphi_2,\varphi_3)  \in C^\infty(\mathbb{R}^3\backslash \Gamma; \mathbb{R}^3)$ supported in some ball $B(R)\subset\R^3$ satisfying \eqref{BC}, i.e.
	\begin{equation}\label{bcphi}
	\varphi_1^+(x, z)= - \varphi_1^-(x, z),\ 
	\varphi_2^+(x, z)=  \varphi_2^-(x, z),\ 
	\varphi_3^+(x, z)= - \varphi_3^-(x, z).
	\end{equation} 
\end{definition}

The problem of existence and rigidity for the PN model interests us the most: 
\begin{enumerate}[(i)]
	\item Does the minimizer of total energy \eqref{E2.2} {in the sense of Definition \ref{minimizer}} exist? 
	\item Do minimizers {in (i)} and layer solutions {(see Definition \ref{def:layersolution})} have 1D symmetry on the slip plane, i.e. are only depending on the shear direction, but independent with the transverse direction? 
\end{enumerate}

The answers to these two questions are both positive. To provide explicit and complete answers to these two questions, we consider the resulting Euler-Lagrange equation satisfied by the minimizer, which is a Lam\'e system with nonlinear boundary conditions on the slip plane (see \eqref{eq:EL}). Because we assume \eqref{assum:1d}, i.e., the misfit potential $\gamma$ depends only on the shear displacement {$u_1$}, this Euler-Lagrange equation is reduced to a nonlocal semi-linear scalar equation on the slip plane with an elliptic pseudo-differential operator of order $1$ (see \eqref{eq:reduced1d}). In particular, when $\nu\in(-1/2,1/3)$, the pseudo-differential operator can be described in the singular kernel formulation; see Assumptions \eqref{A}-\eqref{D}.

After these simplifications and reformulation, we only need to focus on the existence and rigidity of this reduced scalar nonlocal equation (see \eqref{eq:generalized1d}). This equation is the Euler-Lagrange equation of a reduced energy function $F$ on the slip plane (see \eqref{def:energyfunctional}). We will first prove that minimizers of {energy functional $F$ (see \eqref{def:energyfunctional})} {in set \eqref{def:funcsetA}} exist by constructing a minimizing sequence {in which each function} is an $H^{1/2}$ perturbation of a given 1D profile; see Theorem \ref{thm:existence}. Although starting from this weak regularity, we finally prove that these minimizers are smooth 1D profiles that monotonically and uniformly converge to stable states of the misfit functional $\gamma$ in the shear direction, {i.e., they converge to $\pm 1$ as $x\to\pm\infty$}. 

After proving Theorem \ref{thm:existence}, we also establish a rigidity result of De Giorgi-type conjecture on 1D symmetry for all minimizers {in set \eqref{def:funcsetA}}, and more generally for all layer solutions {(see Definition \ref{def:layersolution})}. As a corollary, the uniqueness of {these} minimizers, as well as layer solutions, is also demonstrated; see Theorem \ref{thm:degiorgi} and Theorem \ref{thm:uniqueness}. The existence and rigidity results are also stated for the original vectorial PN model \eqref{eq:EL} in Theorem \ref{thm:PN}.  

Our results on both existence and rigidity hold in any dimension $d\geq 1$ due to the periodic assumption: we are interested in solutions that are periodic in $d-1$ transverse directions. This dimension-independent rigidity is also observed in other equations if the domain is armed with periodicity \cite{ignat2020giorgi}.

In terms of materials science, our results provide a compatible physical interpretation. For Poisson ratio $\nu\in(-1/2,1/3)$, if the misfit potential $\gamma$ depends only on the shear displacement, then the equilibrium dislocation profile only admits shear displacements on the slip plane. Furthermore, this uniquely (up to translations) determined shear displacement is a strictly monotonic 1D profile connecting two stable states. In view of this rigidity result, the vectorial PN model \eqref{eq:EL} in three dimensions is reduced to a two-dimensional problem which {was} thoroughly investigated in our previous work \cite{gao2019mathematical}.

In the remaining parts of the introduction, we will introduce the vectorial PN model and its reduced scalar equation {(see \eqref{eq:reduced1d})} in Section \ref{sec:PN}. From the reduced scalar equation, in Section \ref{sec:setup}, we introduce the nonlocal high-dimensional equation {(see \eqref{eq:generalized1d})} which contains the PN model as a special case. Finally, in this general context, we will present our main results and strategies in Section \ref{sec:mainresult} to provide a rigorous and complete answer to the main questions {(i) and (ii)}.

\subsection{The vectorial Peierls-Nabarro model and its reduced scalar equation}\label{sec:PN}
Denote the unit torus $\R/\mathbb{Z}$ as $\T$.
Instead of minimizing the total energy \eqref{E2.2} on $\R^3$, we consider the model on $\R^2\times\T$. Correspondingly, the slip plane $\Gamma$ is replaced by $\Gamma'$ {which is }defined as
\begin{align}\label{def:slipplane}
\Gamma'=\{(x,y,z)\in\R^2\times\T:\ y=0\}.
\end{align} 
A standard calculation of the first variation of the total energy \eqref{E2.2} derives the following Euler--Lagrange equation satisfied by minimizers of \eqref{E2.2} {in the sense of Definition \ref{minimizer}}. The proof of this lemma can be found in Appendix \ref{sec:proof} or our previous work \cite{gao2019mathematical}.
\begin{lemma}\label{Lem2.2}
Assume that
 $\b{u}\in C^2(\mathbb{R}^2\times\T\backslash \Gamma') $ is a  minimizer of the total energy $E$ in the sense of Definition \ref{minimizer} satisfying the boundary conditions \eqref{BC}. Then $\b u$ satisfies the Euler--Lagrange equation
\begin{align}\label{eq:EL}
\begin{cases}
\Delta\b{u}+\dfrac{1}{1-2\nu}\nabla(\nabla\cdot\b{u})=0,\ &\mathrm{in\ }\mathbb{R}^2\times\T\setminus\Gamma', \\
\sigma^{+}_{12}+\sigma^{-}_{12}=\dfrac{\p \gamma}{\p u_1}(u_1^+,u_3^+),\ &\mathrm{on\ }\Gamma',\\
\sigma^{+}_{22}=\sigma^{-}_{22},\ &\mathrm{on\ }\Gamma',\\
\sigma^{+}_{32}+\sigma^{-}_{32}=\dfrac{\p \gamma}{\p u_3}(u_1^+,u_3^+),\ &\mathrm{on\ }\Gamma'.\\
\end{cases}
\end{align}
\end{lemma}
\begin{remark}
We can view \eqref{eq:EL}, especially the second and the fourth equations as an incorporation of the linear response theory. Moreover, they are coupled equations of $u_1$ and $u_3$. Notice that taking trace in \eqref{strainandconstitutive}, $\sigma_{12}=G(\p_1u_2+\p_2u_1)$ on $\Gamma'$. (i) Regard the elastic bulks $\mathbb{R}^3\backslash \Gamma'$ as an environment and the slip plane $\Gamma$ as an open system. (ii) Given a Dirichlet  disregistry boundary condition $u^+_1, u^+_3$, by solving $\nabla\cdot\sigma=0$ in the environment, one can obtain the trace $\sigma_{12}, \sigma_{32}$ on $\Gamma'$. (iii) We call this operator $(u^+_1, u^+_3)\mapsto (\sigma_{12}^\pm, \sigma_{32}^\pm)$ the Dirichlet to Neumann map; also known as a nonlocal linear response operator. As a consequence, this enables us to consider a nonlocal semi-linear elliptic system on $\Gamma'$; see Section  \ref{sec:preliminary} on the kernel representation of   this Dirichlet to Neumann map.
\end{remark}
For the special case {in \eqref{assum:1d}, where $\gamma$ only depends on the shear displacement $u_1^+$}, we can simplify and decouple the system \eqref{eq:EL} into two independent equations and finally drive a reduced scalar equation of $u_1^+$, {i.e. \eqref{eq:reduced1d}}. In details, one can  employ the Dirichlet to Neumann map and the elastic extension introduced in \cite{gao2019mathematical} to reduce the problem from $\mathbb{R}^2\times\T$ onto $\Gamma'$, i.e., to  {equations} of $(u_1^+,u_3^+)$ on the split plane $\Gamma'$. Second, if further employing {\eqref{assum:1d}}, one can derive a linear representation formula between $u_1^+$ and $u_3^+$ on the Fourier side, i.e.,
\begin{align}\label{eq:Fouriersymbol}
	\hat{u}_3^+(\b k)=-\dfrac{\nu k_1k_2}{(1-\nu)k_1^2+k_2^2}\hat{u}_1^+(\b k).
\end{align}
Here $\hat{u}_i^+(k_1,k_2),i=1,3$ are the Fourier transform of $u_i^+, i=1,3$ with frequency vector $\b{k}=(k_1,k_2), k_1\in\R, k_2\in 2\pi\mathbb{Z}$. Substituting {\eqref{eq:Fouriersymbol}} into system \eqref{eq:EL}, an independent equation of $u_1^+$ is derived, which contains a pseudo-differential operator $\mathcal{L}$ defined on $H^1(\R\times\T)$: 
\begin{equation}\label{eq:reduced1d}
	\begin{aligned}
	&\mathcal{L}u_1^+(x,z)+\dfrac{\gamma'(u_1^+(x,z))}{2G}=0,\quad \widehat{\mathcal{L}u_1}(\b k)=\dfrac{|k|^3\hat{u}_1(\b k)}{(1-\nu)k_1^2+k_2^2}.
	\end{aligned}
\end{equation}

The derivation of \eqref{eq:reduced1d} is standard and can be found in Section \ref{sec:preliminary}. Therefore, as long as we can solve the non-local semi-linear equation of $u_1^+$, {i.e. the first equation} in \eqref{eq:reduced1d}, we can also find $u_3^+$ by {\eqref{eq:Fouriersymbol}}, and then derive the solution of the original system \eqref{eq:EL}. For brevity, we will omit the superscript '$+$' in the following sections. We call \eqref{eq:reduced1d} the reduced scalar equation. 

To solve \eqref{eq:reduced1d}, a meaningful observation is that we can write down an explicit solution to it for certain double-well potential $\gamma$'s. Highly compatible with dislocations in Halite, the cosine potential $\gamma_0=\dfrac{1}{\pi^2}(\cos(\pi u)+1)$ in the PN model implements an explicit solution \cite{anderson2017theory,frenkel1926theorie} to \eqref{eq:reduced1d} and \eqref{eq:Fouriersymbol}:
\begin{align}\label{eq:straightsolution}
u_1(x,z)=\dfrac{2}{\pi}\arctan\left(\dfrac{(1-\nu)x}{2G}\right), \quad u_3(x,z)=0.
\end{align}
In particular, $u_1$ in \eqref{eq:straightsolution} is a layer solution (see Definition \ref{def:layersolution}) since it is strictly monotonic in $x$ direction and satisfies assumption {\eqref{assm:farendlimit}}.
In fact, \eqref{eq:straightsolution} is a good candidate for minimizers of total energy \eqref{E2.2} {in the sense of Definition \ref{minimizer}}. We will prove that this solution is the unique minimizer up to translations which concludes the question on existence and rigidity. {We remark here that this is just a concrete example of our general result: for general double-well type potentials, we prove that minimizers of the total energy $F$ (see \eqref{def:energyfunctional}) in function set \eqref{def:funcsetA} exist and they are layer solutions (see Definition \ref{def:layersolution}). Moreover, they are unique up to translations.} 

\subsection{Unsolved problems on the vectorial PN model} \label{sec:unsolved}
The existence and rigidity of the vectorial PN model are important in understanding dislocations. Previous literature on the vectorial PN model mainly focused on numerical simulations \cite{xiang2008generalized,zimmerman2000generalized,lu2000generalized,schoeck1999peierls} and physical experiments \cite{zimmerman2000generalized,stobbs1971weak}, while only few rigorous mathematical results \cite{gao2019mathematical} were derived. In this section, we aim at mathematically formulating those important but unsolved problems into a framework and embedding our result into this macroscopic framework.

Consider the Euler-Lagrange equation of the vectorial model. We first observed that the second and fourth equations of \eqref{eq:EL} can be rewritten as
\begin{equation} \label{eq:coupled}
\mathcal{A}
	\begin{pmatrix}
	u_1^+\\ u_3^+
	\end{pmatrix}=\dfrac{1}{2G}
	\begin{pmatrix}
	\dfrac{\p \gamma}{\p u_1^+}\\
	\dfrac{\p \gamma}{\p u_3^+}
	\end{pmatrix}, \quad  \hat{\mathcal{A}}=\begin{pmatrix}
\dfrac{k_2^2}{\|k\|}+\dfrac{1}{1-\nu}\cdot\dfrac{k_1^2}{\|k\|} & \dfrac{\nu}{1-\nu}\cdot\dfrac{k_1k_2}{\|k\|}\\
\dfrac{\nu}{1-\nu}\cdot\dfrac{k_1k_2}{\|k\|} & \dfrac{k_1^2}{\|k\|}+\dfrac{1}{1-\nu}\cdot\dfrac{k_2^2}{\|k\|}
\end{pmatrix},
\end{equation}
where $\mathcal{A}$ is a pseudo-differential operator with Fourier symbol $\hat{\mathcal{A}}$. Equation \eqref{eq:coupled} is a nonlocal reduced elliptic equations on {slip plane} $\Gamma$, {which is an open system}. 
Meanwhile, the misfit potential here may depend on both $u_1$ and $u_3$: $\gamma=\gamma(u_1,u_3)$. As far as we know, neither existence nor rigidity of \eqref{eq:coupled} was studied by previous literatures.

In fact, if the misfit potential is carefully selected \cite{zimmerman2000generalized,lu2000generalized,schoeck1999peierls,anderson2017theory}, solutions of \eqref{eq:coupled} determine the underlying structure of dislocations in crystals such as Cu ($\nu=0.36$) and Al ($\nu=0.33$), no matter straight ones or curved ones. Considering symmetry of the crystal lattice, authors of \cite{xiang2008generalized} adopted a truncated Fourier expansion for the generalized stacking fault energy (see eq.(14) in \cite{xiang2008generalized}) as the misfit energy $\gamma(u_1, u_3)$, with different coefficients for Cu and Al.

Numerical simulations in \cite{xiang2008generalized} indicated that straight edge dislocations in both Cu (figure 4 in \cite{xiang2008generalized}) and Al (figure 5 in \cite{xiang2008generalized}) possess the following structure: the displacement in the shear direction (i.e. $u_1$) is a layer solution (see Definition \ref{def:layersolutionwhole}) and the displacement in the transverse direction (i.e. $u_3$) is a solitary wave. As claimed by authors in \cite{xiang2008generalized}, this numerical result agrees with data from experiments on real materials \cite{stobbs1971weak}.

To understand this consistency between the numerical result and the experimental data, we consider \eqref{eq:coupled} where the Poisson ratio $\nu =0$, which is exactly the case for cork. This special case is less obscure since equations for $u_1$ and $u_3$ are decoupled now. Furthermore, we assume that the misfit potential consists of two parts which depend merely on $u_1$ and $u_3$ respectively:
	\begin{align} \label{condition:1+3}
	\gamma(u_1,u_3)=\gamma_1(u_1)+\gamma_3(u_3).
	\end{align}
	Suppose that $\gamma_1$ is a double-well potential (see \eqref{condition:potential})
	and $\gamma_3$ is the nonlinear potential in the Benjamin-Ono equation \cite{benjamin1967internal}, i.e.
	\begin{align} \label{condition:BJO}
	\gamma_3(u_3)=\dfrac{u_3^2}{2}-\dfrac{u_3^3}{3}.
	\end{align}
	If $\b{u}$ is a minimizer of \eqref{E2.2} with boundary conditions 
	\begin{align}\label{lim_u3}
	\lim\limits_{x\to\pm\infty}u_1(x,y,z)=\pm 1,\ \lim\limits_{x\to\pm\infty}u_3(x,y,z)=0,
	\end{align}
	then separately, $u_1$ and $u_3$ satisfy
	\begin{align} \label{eq:help}
	2G\cdot(-\Delta)^{1/2}u_i^++\gamma_i'(u_i^+)=0,\ i=1,3.
	\end{align}
	For $i=1$, because $\gamma_1$ is a double-well potential, \cite{cabre2005layer} proved that \eqref{eq:help} admits layer solutions $u_1$ (unique up to translations) in the sense of Definition \ref{def:layersolutionwhole}. For $i=3$, \eqref{eq:help} is the  traveling wave form of the Benjamin-Ono equation which admits
	\begin{align}
	u_3^+(x,z) = \dfrac{4G}{4G^2+x^2}
	\end{align} 
	as a solitary solution \cite{benjamin1967internal}. These solutions also satisfy the boundary condition \eqref{lim_u3}. These special solutions partially explain the structure of minimizers observed in \cite{xiang2008generalized,stobbs1971weak}, i.e., it is a layer solution in the shear direction while it is a solitary wave in the transverse direction.

As far as we know, a complete answer on the rigidity of minimizers is still unknown even  for the case $\nu=0$. More explicitly, does the De Giorgi conjecture hold in this case? Does \eqref{eq:coupled} (or \eqref{eq:help}) admit any other solution? Existing evidence indicates negative results. If $(-\Delta)^{1/2}$ is replaced by $-\Delta$ in \eqref{eq:help} and we take $i=3$, the author of \cite{dancer2001new} proved that there exist solutions being a soliton in the $x$ direction while being periodic (but non-constant) in the $z$ direction. Thus, the one-dimensional symmetry (or the De Giorgi conjecture) fails in this case. This evidence strongly indicates the existence of high-dimensional solutions of $u_3$ in \eqref{eq:help}. High dimensionality physically indicates the existence of curved dislocations but the construction of a counterexample for the De Giorgi conjecture in the nonlocal case is still open.

For general cases where the Poisson ratio $\nu$ is non-zero, neither existence nor rigidity result is proved to our best knowledge. In particular, no matter $u_3$ has a one-dimensional profile or not, no conclusion can be drawn on the rigidity of $u_1$.

Another question regarding the De Giorgi conjecture is also of great interest: for what misfit potential $\gamma$, there exist one-dimensional solutions for \eqref{eq:coupled}? For what misfit potential $\gamma$, the De Giorgi conjecture holds, i.e. all solutions of  \eqref{eq:coupled} are one dimensional? No previous study has ever considered these problems as far as we know.
	
In summary, the rigidity and existence of the vectorial PN model is an important problem that is central to studies of dislocations, both straight and curved dislocations. Our contribution to this macroscopic framework is that, under the assumption $\gamma=\gamma(u_1)$, i.e. $\gamma$ only depends on $u_1$, a complete answer to existence and rigidity is justified even for high dimensions. See the following sections in the introduction.

\subsection{The nonlocal scalar equation in high dimensions}\label{sec:setup}
We remind our audience here that we will focus on the case where $\gamma$ only depends on $u_1$ and $\gamma$ is a double-well {type} potential {(see \eqref{condition:potential})} in the following sections.

We extend the discussion to any dimension $d\geq 1$ and clarify the set up. Denote $\Omega_d:=\R\times\T^{d-1}$, consider the high-dimensional reduced scalar equation in $\Omega_d$:
\begin{align}\label{eq:generalized1d}
\mathcal{L}u(\b w)+\gamma'(u(\b w))=0,\ \b w\in\Omega_d.
\end{align}
The potential function $\gamma\in C^{2}(\R)$ is a double-well potential that satisfies \eqref{condition:potential}. The linear operator $\mathcal{L}$ is a convolution-type singular integral operator \cite{stein1970singular} which is defined as
\begin{align}\label{def:convolutionL}
(\mathcal{L}u)(\b{w}):=\mathrm{P.V.}\int_{\Omega_d}(u(\b{w})-u(\b{w}'))K(\b{w}-\b{w}')\mathrm{d}\b{w}'
\end{align} 
whose convolution  kernel $K(\b{w})$ can be written as
\begin{align}\label{eq:kernel}
K(x,\b{y})=\sum_{\b{j}\in\mathbb{Z}^{d-1}}H(x,\b{y}+\b{j})
\end{align}
where $\b{w}=(x,\b{y}),\  x\in\R,\ \b{y}\in\T^{d-1}$.

We impose several assumptions on the operator $\mathcal{L}$ and its kernel $H$. To clarify these assumptions, we first  introduce the Fourier transform on $\Omega_d$ and the Sobolev spaces $H^s(\Omega_d)$. 

Denote $\Omega_d'=\R\times(2\pi\mathbb{Z})^{d-1}$. The Fourier transform on $\Omega_d$ is understood as a composition in two directions: the Fourier transform on $\R$ in the $x$ direction and the Fourier series expansion on $\T^{d-1}$ in the $\b{y}$ direction. Denote $\b{\nu}=(\xi,\b{k})$ where $\xi\in\R$ and $\b{k}\in(2\pi\mathbb{Z})^{d-1}$, then the Fourier transform of $u(\b{w})$, denoted as $\hat{u}(\b{\nu})$, is defined as 
\begin{align*}
\hat{u}(\b{\nu})=\int_{\Omega_d}e^{-2\pi i(x\xi+\b{y}\cdot\b{k})}u(x,\b{y})\mathrm{d}x\mathrm{d}\b{y}.
\end{align*} 
Thus the Fourier transform on $\Omega_d$ maps functions
defined on $\Omega_d$ into functions defined on $\Omega_d'$. For any $s>0$, we define Sobolev spaces $H^s(\Omega_d)$ in the classical way on the Fourier side:
\begin{align}\label{def:Sobolevspaces}
	H^s(\Omega_d):=\{u\in L^2(\Omega_d):\ |\b{\nu}|^s\hat{u}(\b{\nu})\in L^2(\Omega_d')\}.
\end{align}
$H^s(\Omega_d),\ s>0$ are Hilbert spaces with inner product
\begin{align}
	\langle u,v\rangle_{H^s(\Omega_d)}:=\langle \hat{u},\hat{v}\rangle_{L^2(\Omega_d')}+\langle |\b{\nu}|^{s}\hat{u},|\b{\nu}|^{s}\hat{v}\rangle_{L^2(\Omega_d')}.
\end{align}
Denote the norm induced by this inner product as $\|\cdot\|_{H^s(\Omega_d)}$. We also define the homogeneous norm $\|\cdot\|_{\dot{H}^s(\Omega_d)}$ as
\begin{align}\label{def:homogenuousnorm}
	\|u\|_{\dot{H}^s(\Omega_d)}:=\||\b{\nu}|^s\hat{u}\|_{L^2(\Omega_d')}.
\end{align}

Now we are ready to impose assumptions of $\mathcal{L}$ {in \eqref{def:convolutionL}} and introduce several important properties of it. We assume that:
\begin{enumerate}[(A)]\label{assumption}
	\item\label{A} (symbol of order 1) The Fourier symbol of $\mathcal{L}$ is positive with same order as $|\b{\nu}|$, i.e. for any $\b{\nu}\in\Omega_d'=\R\times(2\pi\mathbb{Z})^{d-1}$, there exist positive constants $c$ and $C$ such that
	\begin{align}\label{condition:L}
	\widehat{\mathcal{L}u}(\b{\nu})=\sigma_{\mathcal{L}}(\b{\nu})\hat{u}(\b{\nu}),\ c|\b{\nu}|\leq\sigma_{\mathcal{L}}(\b{\nu})\leq C|\b{\nu}|
	\end{align}
	\item\label{B} (positivity and continuity) $H(\b{z}):\mathbb{R}^d\to\R$ is positive and continuous on $\R^d\setminus\{\b{0}\}$.
	\item\label{C} (homogeneity) For any $\b{z}\neq \b{0}$ and $a>0$,
	\begin{align}\label{condition:G1}
	\ H(a\b{z})&=a^{-d-1}H(\b{z}).
	\end{align}
	\item\label{D} (symmetry) For any $\b{z}\in\R^d$, $H(\b{z})=H(-\b{z})$.
\end{enumerate}

The assumptions we impose on $\mathcal{L}$ and its kernel {$K$} include two important cases. First, in dimension $d=2$, the non-local operator in equation \eqref{eq:reduced1d} derived from the PN model is included if the Poisson ratio $\nu\in(-1/2,1/3)$. In this case, the operator $\mathcal{L}$ has Fourier symbol $|\b k|^3/((1-\nu)k_1^2+k_2^2)$ (see \eqref{eq:reduced1d}) which is of the same order as $|k|$, so Assumption \eqref{A} is satisfied. Moreover, the authors of \cite{dong2020existence} proved that {$K$} satisfies assumption \eqref{B}, \eqref{C} and \eqref{D} if and only if $\nu\in(-1/2,1/3)$. So equation \eqref{eq:reduced1d} is included in this context if $\nu\in(-1/2,1/3)$. Second, in arbitrary dimensions, {if we take $\nu=0$}, then $\mathcal{L}=(-\Delta)^{1/2}$ defined on $\Omega_d$ is also included. In this case, the Fourier symbol is exactly $|k|$ and according to \cite{kwasnicki2017ten}, there exists a constant $C_d>0$ such that 
\begin{align}
H(z)=\dfrac{C_d}{|\b{z}|^{d+1}}.
\end{align}
So all assumptions are satisfied. 

We remark here that we adopt two different but equivalent definitions for $\mathcal{L}$: one is as a Fourier multiplier and the other is as a singular convolution. The result that these two definitions for fractional Laplacian $(-\Delta)^{\alpha}, \alpha\in(0,1)$ are equivalent is thoroughly investigated in \cite{kwasnicki2017ten}. For equation \eqref{eq:reduced1d}, the equivalence of these two definitions is also well-studied in \cite{dong2020existence}. So in later sections, we will switch between these two definitions for the sake of convenience.

\subsection{Main results and strategies}\label{sec:mainresult}
Before presenting the main results, we introduce the fractional Allen-Cahn equation \cite{allen1972ground}: 
    \begin{align}\label{eq:halflaplacian}
	(-\Delta)^{1/2}u(\b x)+\gamma'(u(\b x))=0,\ \b x\in\R^d.
	\end{align}
Here the double-well {type} potential $\gamma\in C^2(\R)$ {satisfying \eqref{condition:potential}} is exactly the misfit potential in the PN model.
Taking $\nu=0$ in \eqref{eq:reduced1d}, we see that \eqref{eq:halflaplacian} is a special case of \eqref{eq:generalized1d} and \eqref{eq:reduced1d}. \eqref{eq:halflaplacian} has already been thoroughly investigated in the literatures \cite{cabre2005layer,palatucci2013local,sire2009fractional,savin2018rigidity,savin2019rigidity,figalli2020stable}. In particular, the well-posedness result of \eqref{eq:halflaplacian} is completely developed. A long standing conjecture named after De Giorgi \cite{De1978Convergence,ghoussoub1998conjecture} {(which originally discussed the local case, i.e., one replaces $(-\Delta)^{1/2}$ by $-\Delta$ in \eqref{eq:halflaplacian}, but then generalized to the non-local case \eqref{eq:halflaplacian})} is proved for dimension $d\leq 8$. The De Giorgi conjecture claims that any layer solution (see below) to \eqref{eq:halflaplacian} is a simple 1D profile for dimensions $d\leq 8$. {In the classical Allen-Cahn equation,} this conjecture is optimal in the sense that a counterexample in dimension $d=9$ is constructed \cite{del2011giorgi}. 
	The layer solution in the De Giorgi conjecture is defined as:
	\begin{definition}\label{def:layersolutionwhole}
		$u:\R^d\to\R$ is a layer solution to equation \eqref{eq:halflaplacian} if
		\begin{align}\label{condition:layerwhole}
		\dfrac{\p u(\b x)}{\p x_1}>0,\ \lim\limits_{x_1\to\pm\infty}u(\b x)=\pm 1.
		\end{align}
\end{definition} 
The layer solution is also of main interest in the PN model since it models a dislocation {profile} that monotonically converges to two stable states {at far field} in the shear direction.

Now we are ready to articulate our existence and rigidity result on the PN model and \eqref{eq:generalized1d}. Although the PN model and \eqref{eq:generalized1d} share the common double-well nonlinearity and non-localness with the fractional Landau-Ginzburg equation (see \eqref{eq:halflaplacian}), the main difference between our setting and previous work on \eqref{eq:halflaplacian} is that we work on a partially periodic domain $\Omega_d$ and Hilbert spaces $H^s(\Omega_d)$ while previous work focused on the whole domain $\R^d$ and Banach spaces $C^{k,\alpha}(\R^d)$. This discrepancy in the setting urges us to develop more appropriate methods while referring to some valuable techniques introduced in previous work.

For the existence problem, although we know that \eqref{eq:straightsolution} is a solution to \eqref{eq:reduced1d} and \eqref{eq:Fouriersymbol}, we are still not aware of whether a minimizer of the total energy \eqref{E2.2} {in function set \eqref{def:funcsetA}} exists. In \cite{cabre2005layer}, authors worked on H\"older spaces $C^{k,\alpha}(\R^d),\ k=0,1,2,$ and derived some Schauder's estimates of the weak solution to \eqref{eq:halflaplacian}. Based on these estimates, they proved the existence of the classical solutions to \eqref{eq:halflaplacian} for $d=1$ by considering the harmonic extension of \eqref{eq:halflaplacian} on the upper half-plane. In \cite{palatucci2013local}, the authors adopted the direct method in calculus of variations minimizing the total energy on a subset of $L^1_{\mathrm{loc}}(\R)$, i.e.
\begin{align}\label{eq:savin}
	\mathcal{X}:=\{f\in L^1_{\mathrm{loc}}(\R):\  \lim\limits_{x\to\pm\infty}f(x)=\pm 1\}.
\end{align}
They proved the existence and uniqueness of the minimizer in one dimension and {the existence result is generalized to any dimension $d$}. 

For the high-dimensional equation \eqref{eq:generalized1d}, we will follow the idea of \cite{palatucci2013local} by using the direct method in the calculus of variations to prove that the minimizer of a functional $F(u)$ exists. However, instead of requiring the far field assumption \eqref{assm:farendlimit}, we only consider $H^{1/2}$ perturbation of a given 1D profile $\eta$ who satisfies  \eqref{assm:farendlimit}:
\begin{align}\label{def:funcsetA}
\mathcal{A}:=\{u\in H^{1/2}_{\mathrm{loc}}(\Omega_d): u-\eta\in H^{1/2}(\Omega_d)\}.
\end{align}
Here $\eta(x,\b{y})$ is a smooth 1D profile, i.e. $\eta(x,\b{y})=\eta(x)$ for any $(x,\b{y})\in\Omega_d$, satisfying 
\begin{equation}\label{eq:1dprofile}
\begin{aligned}
\eta(x)&\in C^{\infty}(\R), \quad
\eta(x) = \left\{  \begin{array}{cc}
1 \quad \mathrm{if}\quad x\in[1,+\infty),\\
-1 \quad \mathrm{if}\quad x\in(-\infty,-1].
\end{array}  \right.
\end{aligned}
\end{equation}
We will abuse the notation $\eta$ to represent the profile defined on either $\Omega_{d}$ or $\R$. We remark here that the weak $H^{1/2}$ regularity does not ensure any far field limit behavior in any dimension, even in dimension $d=1$. 

The functional $F$ that we aim to minimize is the perturbed version of the total energy \eqref{E2.2}:
\begin{equation}\label{def:energyfunctional}
\begin{aligned}
F(u)&:=\dfrac{1}{2}\int_{\Omega_d}\int_{\Omega_d}{|(u(x,\b{y})-u(x',\b{y}')|^2}K(x-x',\b{y}-\b{y}')\\
&-{|(\eta(x,\b{y})-\eta(x',\b{y}')|^2}K(x-x',\b{y}-\b{y}')\mathrm{d}x\mathrm{d}\b{y}\mathrm{d}x'\mathrm{d}\b{y'}+\int_{\Omega_d}\gamma(u(x,\b y))\mathrm{d}x\mathrm{d}\b{y}.
\end{aligned}
\end{equation}
Here $\gamma\in C^2(\R)$ is the double-well {type} potential considered in \eqref{eq:generalized1d} that satisfies \eqref{condition:potential}. 

Starting from functions only with $H^{1/2}$ regularity and even not necessarily satisfying the far field limit condition, we construct a minimizer with $H^2$ regularity (in fact smooth) that also satisfies the desired rigidity result that we aim to prove: it is a layer solution with 1D symmetry.   
\begin{theorem}(Existence of the minimizer)\label{thm:existence}
	Suppose that $\gamma\in C^2(\R)$ is a double-well {type} potential satisfying condition \eqref{condition:potential}. Consider set $\mathcal{A}$ defined in \eqref{def:funcsetA} and energy functional $F$ defined in \eqref{def:energyfunctional}. Then: 
	\begin{enumerate}[(i)]
		\item (existence) There exists $u^*\in\mathcal{A}$ such that $F(u^*)=\min\limits_{u\in\mathcal{A}}F(u)$. In particular, $u^*$ is a weak solution to \eqref{eq:generalized1d}, i.e.
		\begin{align*}
		\mathcal{L}u^*+\gamma'(u^*)=0.
		\end{align*}
		Here $\mathcal{L}$ is defined as in \eqref{def:convolutionL}.  
		\item (regularity) $u^*$ in $(i)$ satisfies $u^*-\eta\in H^2(\Omega_d)$. In particular, $u^*$ solves equation \eqref{eq:generalized1d} in $L^2$ sense and satisfies the far end limit condition uniformly in $\b{y}$: 
		\begin{align*}
		\lim\limits_{x\to\pm\infty}u^*(x,\b{y})=\pm 1
		\end{align*}
		\item (monotonicity) $u^*$ in $(i)$ satisfies $\dfrac{\p u^*(x,\b{y})}{\p x}>0$ for any $(x,\b{y})\in\Omega_d$, i.e. $u^*(x,\b{y})$ is strictly increasing in $x$ direction.
		\item (symmetry) $u^*$ in $(i)$ satisfies $ \nabla_{\b{y}}u^*(x,\b{y})=\b{0}$ for any $(x,\b{y})\in\Omega_d$, i.e. $u^*(x,\b{y})=u^*(x)$ is a 1D profile. 
	\end{enumerate}
\end{theorem}
The critical technique in the proof is the energy decreasing rearrangement method in \cite{palatucci2013local}. This method relies on the rearrangement inequality (see Lemma \ref{lmm:rearrangeinequality}) whose proof is quite elementary. However, driven by this basic inequality, the energy decreasing method is powerful in proving monotonicity and 1D symmetry. We will introduce this method in Section \ref{sec:rearrangement}.

For the rigidity problem, it worth mentioning the De Giorgi conjecture on the fractional Ginzburg-Landau equation \eqref{eq:halflaplacian}. It claims that at least for $d\leq 8$, layer solutions to \eqref{eq:halflaplacian} are in fact just 1D profiles. Here a layer solution is defined in Definition \ref{def:layersolution}. This conjecture was proved for dimension $d=2$ in \cite{cabre2005layer} and finally completely proved by Savin in his series of work \cite{savin2018rigidity,savin2019rigidity,palatucci2013local}. We also mention  the asymptotic analysis for the sharp interface limit of the fractional  diffusion-reaction equation with isotropic/anisotropic nonlocal kernel of order $\frac{1}{r^{d+2s}}, s\in(0,1)$ in   \cite{ABS_1998, Garroni_Muller_2005, Imbert_Souganidis_2009, Millot_Sire_Wang_2019} . 

The common approach to prove the De Giorgi conjecture is to develop a Liouville-type theorem and then {apply} the theorem on ratios of partial derivatives in different directions $u_{x_i}/u_{x_1}, i=1,2,...,d-1$. Then one can conclude that there exist constants $c_i, i=1,2,...,d-1$ such that
\begin{align*}
	u_{x_i}=c_iu_{x_1}, \quad  i=1,2,..,d-1.
\end{align*}
So $u$ in fact only depends on $x_1$ variable and hence is a 1D profile. For \eqref{eq:generalized1d}, we can prove the following theorem which is true for any dimension $d\geq 1$, not only for dimension $d\leq 8$:

\begin{theorem}\label{thm:degiorgi}(De Giorgi Conjecture)
	For any dimension $d\geq 1$, suppose that $u:\Omega_d\to\R$ satisfies $u-\eta\in H^1(\Omega_d)$ and is a layer solution (defined in Definition \ref{def:layersolution}) to equation \eqref{eq:generalized1d}, i.e.
	\begin{align*}
	\mathcal{L}u+\gamma'(u)=0.
	\end{align*}
	Here $\gamma\in C^{\infty}(\R)$ is a double-well {type} potential satisfies \eqref{condition:potential} and $\mathcal{L}$ is defined in \eqref{def:convolutionL} satisfying Assumptions \eqref{A}-\eqref{D}.
	Then $u(x,\b{y})$ only depends on $x$ variable, i.e. there exists $\phi:\R\to\R$ such that $u(x,\b{y})=\phi(x)$.
\end{theorem}
{We emphasize here that the main reason of Theorem \ref{thm:degiorgi} being true in any dimension $d\geq 1$ instead of only dimensions less than eight is that we fully employed the compactness of the torus $\T^{d-1}$. The compactness ensures convergence of a sequence which is a key step in our proof (see the proof of Theorem \ref{thm:degiorgi}). As we explained in Section \ref{sec:unsolved}, without periodicity, the De Giorgi conjecture may fail in the classical case \cite{dancer2001new}. Therefore, domain $\Omega^d=\R\times\T^{d-1}$ is critical to our result which is also physically meaningful since it incorporates the periodicity in materials.}

Instead of using Liouville type theorems, we {prove} Theorem \ref{thm:degiorgi} by analyzing the spectrum of a linear operator. This method sufficiently respects the maximal property of operator $\mathcal{L}$ (see Lemma \ref{lmm:max}) which is realized by the positivity assumption (Assumption \eqref{B}). Utilized in our previous work \cite{gao2019long}, this spectral analysis method is straightforward and appropriate in our setting since the working spaces are selected as Hilbert spaces $H^s(\Omega_d)$ instead of Banach spaces $C^{k,\alpha}(\R^d)$ in \cite{cabre2005layer}. Under this setting, we can use the perturbation theory of self-adjoint operators on Hilbert spaces \cite{kato2013perturbation}.

Specifically speaking, suppose that $u$ is a given solution to \eqref{eq:generalized1d} that satisfies conditions in Theorem \ref{thm:degiorgi}. Differentiating on both sides of equation \eqref{eq:generalized1d}, we see that $u_x$ and $u_{y_i}, i=1,2,...,d-1$ are solutions to the following non-local linear elliptic equation of $\phi$ on $\Omega_d$ which is given by:
\begin{align*}
\left[\mathcal{L}+\gamma''(u)\right]\phi=0.
\end{align*}    
Equivalently, $u_x$ and $u_{y_i},i=1,2,...,d-1$ are eigenfunctions of eigenvalue 0 for the linear operator linearized along profile $u$: 
\begin{align}\label{def:operatorL}
L:H^1(\Omega_d)\subset L^2(\Omega_d)\to L^2(\Omega_d),\ L\phi=\mathcal{L}\phi+\gamma''(u)\phi.
\end{align} 
Therefore, as long as we can prove that 0 is a simple eigenvalue of $L$, i.e. the eigenspace of 0 is only 1 dimension, then we prove that $u_x$ and $u_{y_i},i=1,2,...,d-1$ are linearly dependent, which indicates 1D symmetry. This is the main idea and approach we will utilize to prove Theorem \ref{thm:degiorgi}.

As a direct corollary of Theorem \ref{thm:degiorgi}, we can prove that both layer solutions and minimizers of $F$ on $\mathcal{A}$ are unique up to {translations}. Define $\mathcal{A}_\ell$ and $\mathcal{A}_m$ as
\begin{equation}\label{def:minimizersandlayersolutions}
\begin{aligned}
\mathcal{A}_\ell&:=\{u\in\dot{H}^1(\Omega_d): u-\eta\in H^1(\Omega_d),\ u\ \mathrm{is\ a \ layer\ solution\ to}\ \eqref{eq:generalized1d}\},\\
\mathcal{A}_m&:=\{u\in\mathcal{A}: F(u)=\min\limits_{v\in\mathcal{A}} F(v)\},
\end{aligned}
\end{equation}
i.e. $\mathcal{A}_\ell$ is the set of layer solutions to \eqref{eq:generalized1d} with $H^1$ regularity and $\mathcal{A}_m$ is the set of minimizers of $F$ on set $\mathcal{A}$. Then we can prove the following theorem:
\begin{theorem}\label{thm:uniqueness}
(uniqueness of minimizers and layer solutions) For any dimension $d\geq 1$, suppose that $\gamma\in C^{\infty}(\R)$ is a double-well {type} potential satisfying \eqref{condition:potential}. Consider functional energy $F$ in \eqref{def:energyfunctional}, set $\mathcal{A}$ in \eqref{def:funcsetA}, set $\mathcal{A}_\ell$ and set $\mathcal{A}_m$ in \eqref{def:minimizersandlayersolutions}.
Then 
\begin{align}\label{eq:uniqueness}
\mathcal{A}_m=\mathcal{A}_\ell=\{u:\ u(x,\b{y})=u^*(x+x_0)\ \mathrm{for\ some}\ x_0\in\R\}.
\end{align}
Here $u^*(x)$ is the unique solution to equation 
\begin{align}\label{eq:essential1d}
c_{\mathcal{L}}(-\p_{xx})^{1/2}u^*+\gamma'(u^*)=0,\ u^*(0)=0.
\end{align}
Here $c_{\mathcal{L}}$ is the constant in Lemma \ref{lmm:propertyofL2}.
\end{theorem}

Theorem \ref{thm:uniqueness} provides a compatible physical interpretation of the PN model in three dimensions (with periodicity in the transverse direction): if we assume exclusive dependence of the misfit potential on the shear displacement, then the equilibrium dislocation on the slip plane only admits shear displacements. Furthermore, this uniquely (up to translations) determined shear displacement is a strictly monotonic 1D profile connecting two stable states. This reduces the vectorial PN model to the two-dimensional PN model which was  investigated in our previous work \cite{gao2019mathematical}. In summary, we have the following theorem:
\begin{theorem}\label{thm:PN}
	Suppose that $\gamma\in C^{\infty}(\R)$ is a double-well {type} potential satisfying \eqref{condition:potential}. Consider the functional energy $E$ in \eqref{E2.2} integrating on $\R^2\times\T$, i.e.
	\begin{align}\label{energy}
		\tilde{E}(\b u)=\int_{\R^2\times\T\setminus\Gamma'}\dfrac{1}{2}\sigma:\varepsilon\mathrm{d}x\mathrm{d}y\mathrm{d}z+\int_{\Gamma'}\gamma(u_1^+)\mathrm{d}x\mathrm{d}z.
	\end{align}
	Here $\b u=(u_1,u_2,u_3)$ is the displacement vector, $\sigma$ and $\varepsilon$ are the strain tensor and the stress tensor respectively given by \eqref{strainandconstitutive} and $\Gamma'$ is the slip plane defined in \eqref{def:slipplane}. Assume $\nu\in(-1/2,1/3)$. Suppose that $\b u$ is a global minimizer of $\tilde{E}$ {as in Definition \ref{minimizer}}, then: 
	\begin{enumerate}[(i)]
		\item (regularity) The displacement vector $\b u$ is smooth in $\R^2\times\T\setminus{\Gamma'}$.
		\item (rigidity) The displacement in transverse direction is 0, i.e. $u_3=0$ in $\R^2\times\T$; $u_1^+$ is the unique (up to translation in $x$ direction) 1D profile independent with $z$ variable, strictly monotonic in $x$ direction satisfying
		\begin{align}
			\lim\limits_{x\to\pm\infty}u^+_1(x)=\pm 1.
		\end{align}
		\item (Fourier representation) $u_1$ and $u_2$ only depend on $x$ and $y$ in $\R^2\times\T$. On the Fourier side, $u_1$ and $u_2$ can be uniquely represented by $u_1^{\pm}(x)$:
		\begin{align}
			\hat{u}_1^{\pm}(\xi,y)&=\hat{u}_1^{\pm}(\xi)\left(1-\dfrac{|\xi y|}{2-2\nu}\right)e^{-|\xi y|}\\
			\hat{u}_2^{\pm}(\xi,y)&=-\dfrac{\hat{u}_1^{\pm}(\xi)}{2-2\nu}\left((1-2\nu)\dfrac{i\xi}{|\xi|}+i\xi|y|\right)e^{-|\xi y|}.
		\end{align} 
		\item (Dirichlet to Neumann map) On $\Gamma'$, the stress tensor can be expressed as
		\begin{align}
			\sigma^+_{12}(x)&=\sigma^-_{12}(x)=-\dfrac{G}{(1-\nu)\pi}\mathrm{P.V.}\int_{\R}\dfrac{(u^{+}_1)'(s)}{x-s}\mathrm{d}s\\
			\sigma^+_{22}(x)&=\sigma^-_{22}=0.
		\end{align}
		\item The stress tensor is divergence free, i.e.
		\begin{align}
			\nabla\cdot\sigma=\b 0,\ \mathrm{holds}\ \mathrm{ in}\ D'(\R^2\times\T).
		\end{align}
		This also holds point-wisely in $\R^2\times\T\setminus\Gamma'$.
	\end{enumerate}
\end{theorem}

These four theorems (Theorem \ref{thm:existence}, \ref{thm:degiorgi}, \ref{thm:uniqueness} and \ref{thm:PN}) are the main results for this work which completely close the problem of existence and rigidity in a general setting including the original PN model with Poisson ratio $\nu\in(-1/2,1/3)$. Following this logic, we will first conduct a preliminary analysis in Section \ref{sec:preliminary} to {assist} readers {to} bridge some gaps in understanding the derivation of \eqref{eq:generalized1d} and be aware of some important properties of the linear operator $\mathcal{L}$. Then we prove Theorem \ref{thm:existence} in Section \ref{sec:existence} and Theorem \ref{thm:degiorgi}, \ref{thm:uniqueness} and \ref{thm:PN} in Section \ref{sec:proofofdg}. Finally, the spectral analysis of operator $L$ is established in Section \ref{sec:spectrum} which proves that 0 is simple and the principle eigenvalue of $L$. For facts in functional analysis and details in the spectral analysis, readers may refer to Appendix \ref{sec:funana}; for proofs of some lemmas in the proof of the theorems, readers may refer to Appendix \ref{sec:proof}.

\section{Preliminary analysis}\label{sec:preliminary}
In this section, we will first provide some details of the derivation of the reduced scalar equation \eqref{eq:reduced1d}, then discuss three important properties that will be used in the proof of the three theorems.
\subsection{Derivation of the reduced scalar equation}
Denote the Fourier transform of $u_i^+(x,z),i=1,2,3$ as $\hat{u}_i^+(\b k),i=1,2,3$ where $\b{k}=(k_1,k_2), k_1\in\R, k_2\in 2\pi\mathbb{Z}$ is the frequency vector. Given $\b{u}$ that satisfies equation \eqref{eq:EL}, one can rewrite $(\sigma_{12}^+,\sigma_{32}^+)$ on $\Gamma'$ as a linear transform of $(u_1^+(\b{k}),u_3^+(\b{k}))$ on the Fourier side:
\begin{align}\label{eq:Fourierside}
\begin{pmatrix}
\ \hat{\sigma}_{12}^+(\b{k})\ \\
\ \hat{\sigma}_{32}^+(\b{k})\ \\
\end{pmatrix}
=-\b A
\begin{pmatrix}
\ \hat{u}_{1}^+(\b{k})\ \\
\ \hat{u}_{3}^+(\b{k})\ \\
\end{pmatrix}
:=-2G
\begin{pmatrix}
\ \left(\dfrac{k_2^2}{|\b{k}|}+\dfrac{1}{1-\nu}\dfrac{k_1^2}{|\b{k}|}\right)\hat{u}_1^+(\b{k})+\dfrac{\nu}{1-\nu}\dfrac{k_1k_2}{{|\b{k}|}}\hat{u}_3^+(\b{k})\ \\
\ \dfrac{\nu}{1-\nu}\dfrac{k_1k_2}{|\b{k}|}\hat{u}_1^+(\b{k})+\left(\dfrac{k_1^2}{|\b{k}|}+\dfrac{1}{1-\nu}\dfrac{k_2^2}{|\b{k}|}\right)\hat{u}_3^+(\b{k})\ \\
\end{pmatrix}.
\end{align}
Details of this derivation can be found in Appendix \cite{dong2020existence}.

From equation \eqref{eq:Fourierside}, the Euler-Lagrangian equation \eqref{eq:EL} can be rewritten as an equation of $u_1^+(x,z),u_3^+(x,z)$ on $\Gamma$, i.e.
\begin{align}\label{eq:gammasys}
-\mathcal{A}\begin{pmatrix}
u_1^+(x,z)\\
u_3^+(x,z)\\
\end{pmatrix}=
\begin{pmatrix}
\dfrac{\p \gamma}{\p u_1}(u_1^+,u_3^+)\\
\dfrac{\p \gamma}{\p u_3}(u_1^+,u_3^+)\\
\end{pmatrix}.
\end{align}
Here $\mathcal{A}$ is the nonlocal differential operator with Fourier symbol $\b A$.

A further simplification can be realized on equation \eqref{eq:gammasys} due to independence of $\gamma$ with $u_3$, i.e. $\dfrac{\p\gamma}{\p u_3}=0$. This independence reduces equation \eqref{eq:gammasys} into an equation of $u_1$. On the Fourier side, the second component in \eqref{eq:gammasys} indicates that we can represent $\hat{u}_3$ by $\hat{u}_1$, i.e. 
\begin{align}
\hat{u}_3(k)=-\dfrac{\nu k_1k_2}{(1-\nu)k_1^2+k_2^2}\hat{u}_1(k).
\end{align}
Substituting this equality to the first component in \eqref{eq:Fourierside} yields
\begin{align*}
\hat{\sigma}_{12}(\b k)&=-2G\left[\left(\dfrac{k_2^2}{|\b k|}+\dfrac{1}{1-\nu}\dfrac{k_1^2}{|\b k|}\right)\hat{u}_1(\b k)+\dfrac{\nu}{1-\nu}\dfrac{k_1k_2}{|\b k|}\hat{u}_3(\b k)\right]\\
&=-\dfrac{2G|\b k|^3}{(1-\nu)k_1^2+k_2^2}\hat{u}_1(\b k).
\end{align*}
Now denote $\mathcal{L}:H^1(\R\times\T)\subset L^2(\R\times\T)\to L^2(\R\times\T)$ the linear operator with Fourier symbol $\dfrac{|\b k|^3}{(1-\nu)k_1^2+k_2^2}$. Then the first component of equation $\eqref{eq:gammasys}$ is in fact an equation of $u_1$, i.e. equation \eqref{eq:reduced1d}:
\begin{align}
\mathcal{L}u_1+\dfrac{\gamma'(u_1)}{2G}=0.
\end{align}
This equation is the reduced scalar equation.

\subsection{Properties of $\mathcal{L}$}
Assumption \eqref{A} ensures that $\mathcal{L}$ is a self-adjoint operator defined on $H^1(\Omega_d)\subset L^2(\Omega_d)$ and maps to $L^2(\Omega_d)$ (see Lemma \ref{lmm:selfadjoint}). By assumption \eqref{B} and \eqref{C}, one can easily conclude that kernel {$H$}  satisfies that for any $\b{z}\neq\b{0}$, 
\begin{align}\label{condition:G2}
0<\dfrac{m}{|\b{z}|^{d+1}}\leq {H}(\b{z})&\leq \dfrac{M}{|\b{z}|^{d+1}}.
\end{align}
Here $m$ and $M$ are positive constants. Indeed, for any non-zero $\b{z}$, we have
\begin{align*}
{H}(\b{z})=\dfrac{1}{|\b{z}|^{n+1}}{H}\left(\dfrac{\b{z}}{|\b{z}|}\right)
\end{align*}
and ${H}(\b{z})$ has a positive lower bound $m$ and a positive upper bound $M$ on the compact set $\mathbb{S}^{d-1}$. So \eqref{condition:G2} holds.

Furthermore, we will prove three important properties of the linear operator $\mathcal{L}$ which play critical roles in the proof of Theorem \ref{thm:existence}, \ref{thm:degiorgi} and \ref{thm:uniqueness}.

First, positivity of {$H$} ensures that if $f$ attains global maximum at point $(x_0,\b{y}_0)\in\Omega_d$, then $\mathcal{L}f|_{(x_0,\b{y}_0)}\geq 0$. We call it the maximal principle of operator $\mathcal{L}$:
\begin{lemma}\label{lmm:max} (Maximal principle) 
	Suppose that $f\in \dot{H}^1(\Omega_d)$ attains global maximum at $(x_M,\b{y}_M)$ and  global minimum at $(x_m,\b{y}_m)$ on $\Omega_d$. Then
	\begin{align*}
	\mathcal{L}f|_{(x_M,\b{y}_M)}\geq 0,\quad \mathcal{L}f|_{(x_m,\b{y}_m)}\leq 0.
	\end{align*}
	The equality holds if and only if $f$ is constant.
\end{lemma}
\begin{proof}
	By positivity of $K$, we know that
	\begin{align*}
	\mathcal{L}f|_{(x_M,\b{y}_M)}&=\int_{\Omega_d}(f(x_M,\b{y}_M)-f(x,\b{y}))K(x_M-x,\b{y}_M-\b{y})\mathrm{d}x\mathrm{d}\b{y}\geq 0,\\
	\mathcal{L}f|_{(x_m,\b{y}_m)}&=\int_{\Omega_d}(f(x_m,\b{y}_m)-f(x,\b{y}))K(x_m-x,\b{y}_m-\b{y})\mathrm{d}x\mathrm{d}\b{y}\leq 0.
	\end{align*}
	Thus the inequality holds and the equality holds if and only if $f(x,\b{y})$ is constant.
\end{proof}
We emphasize that this property of operator $\mathcal{L}$ plays an important role in the proof of the De Giorgi conjecture (see Section \ref{sec:DGproof}).

Second, homogeneity of {$H$} ensures that if $f(x,\b{y})=f(x)$, i.e. $f$ is a simple 1D profile independent with variable $\b{y}$, then there exists constant $c_{\mathcal{L}}$ such that $\mathcal{L}f|_{(x,\b{y})}=c_{\mathcal{L}}((-\p_{xx})^{1/2}f)(x)$. 
\begin{lemma}\label{lmm:propertyofL2}
	Suppose that $f\in \dot{H}^1(\Omega_d)$ satisfies $f(x,\b{y})=f(x)$. Then there exists a constant $c_{\mathcal{L}}>0$ such that
	\begin{align*}
	\mathcal{L}f|(x,\b{y})=c_{\mathcal{L}}((-\p_{xx})^{1/2}f)(x).
	\end{align*}
\end{lemma}
\begin{proof}
	Consider $g(x)$ which is defined as 
	\begin{align*}
	g(x)=\int_{\R^{d-1}}{H}(x,\b{y})\mathrm{d}\b{y}.
	\end{align*} 
	Then for any $x\neq 0$, a change of variable {implies} that
	\begin{align*}
	g(x)=\int_{\R^{d-1}}{H}(x,\b{y})\mathrm{d}\b{y}=\int_{\R^{d-1}}|x|^{-d-1}{H}(1,\b{y}/x)\mathrm{d}\b{y}=\int_{\R^{d-1}}|x|^{-2}{H}(1,\b{y})\mathrm{d}\b{y}=\dfrac{g(1)}{|x|^2}.
	\end{align*}
	So $g(x)=g(1)|x|^{-2}$ is the kernel of half Laplacian for {one} dimension. Therefore by Fubini's theorem, if $f(x,\b{y})=f(x)$,  we have
	\begin{align*}
	\mathcal{L}f|_{(x,\b{y})}&=\mathrm{P.V.}\int_{\R}\int_{\T^{d-1}}(f(x,\b{y})-f(x',\b{y}'))K(x-x',\b{y}-\b{y}')\mathrm{d}\b{y}'\mathrm{d}x'\\
	&=\mathrm{P.V.}\int_{\R}(f(x)-f(x'))\int_{\T^{d-1}}K(x-x',\b{y}-\b{y}')\mathrm{d}\b{y}'\mathrm{d}x'.
	\end{align*}
	By \eqref{eq:kernel}, we have
	\begin{align*}
	\int_{\T^{d-1}}K(x-x',\b{y}-\b{y}')\mathrm{d}\b{y}'&=\int_{\T^{d-1}}\sum_{\b{j}\in\mathbb{Z}^{d-1}}{H}(x-x',\b{y}-\b{y}'+\b{j})\mathrm{d}\b{y}'\\
	&=\int_{\R^{d-1}}{H}(x-x',\b{y}-\b{y}')\mathrm{d}\b{y}'.
	\end{align*}
	Substituting this back to the formula of $\mathcal{L}f$, we have 
	\begin{align*}
	\mathcal{L}f|_{(x,\b{y})}&=\mathrm{P.V.}\int_{\R}(f(x)-f(x'))\int_{\R^{d-1}}{H}(x-x',\b{y}-\b{y}')\mathrm{d}\b{y}'\mathrm{d}x'\\
	&=g(1)\cdot \mathrm{P.V.}\int_{\R}\dfrac{f(x)-f(x')}{|x-x'|^2}\mathrm{d}x'\\
	&=g(1)\cdot((-\p_{xx})^{1/2}f)(x).
	\end{align*}
	So for any $f$ that is 1D profile, $\mathcal{L}$ acting on $f$ is just $(-\p_{xx})^{1/2}$ acting on $f$ up to a constant. 
\end{proof}

Third, Assumption \eqref{A} (equation \eqref{condition:L}) ensures the following equivalence of semi-norms:
\begin{lemma}\label{lmm:propertyL3}
	There exist positive constants $c_1,c_2,C_1,C_2$ such that 
	\begin{equation}\label{eq:normequivalence}
	\begin{gathered}
	c_1\|u\|_{\dot{H}^{1}(\Omega_d)}\leq\|\mathcal{L}u\|_{L^2(\Omega_d)}\leq C_1\|u\|_{\dot{H}^{1}(\Omega_d)},\\
	c_2\|u\|_{\dot{H}^{1/2}(\Omega_d)}^2\leq\int_{\Omega_d}\int_{\Omega_d}|u(\b{w})-u(\b{w'})|^2K(\b{w}-\b{w}')\mathrm{d}\b{w}\mathrm{d}\b{w'}\leq C_2\|u\|_{\dot{H}^{1/2}(\Omega_d)}^2.
	\end{gathered}	
	\end{equation}
\end{lemma}
\begin{proof}
	By Plancherel's theorem, we have
	\begin{align*}
	\|\mathcal{L}u\|_{L^2(\Omega_d)}=\|\widehat{\mathcal{L}u}(\b{\nu})\|_{L^2(\Omega_d')}=\|\sigma_{\mathcal{L}}(\b{\nu})\hat{u}(\b{\nu})\|_{L^2(\Omega_d')}.
	\end{align*}
	Then by \eqref{condition:L}, we know
	\begin{align*}
	\|\widehat{\mathcal{L}u}(\b{\nu})\|_{L^2(\Omega_d')}\leq C\||\b{\nu}|\hat{u}(\nu)\|_{L^2(\Omega_d')}=C\|u\|_{\dot{H}^1(\Omega)},\\
	\|\widehat{\mathcal{L}u}(\b{\nu})\|_{L^2(\Omega_d')}\geq c\||\b{\nu}|\hat{u}(\nu)\|_{L^2(\Omega_d')}=c\|u\|_{\dot{H}^1(\Omega)}.
	\end{align*}
	Here $C$ and $c$ are constants in \eqref{condition:L}. So $\|\mathcal{L}u\|_{L^2(\Omega_d)}$ is equivalent to $\|u\|_{\dot{H}^1(\Omega_d)}$. Moreover, by symmetry assumption of $K$, we have
	\begin{align*}
	\langle u,\mathcal{L}u\rangle_{L^2(\Omega_d)}&=\int_{\Omega_d}\int_{\Omega_d}u(\b{w})(u(\b{w})-u(\b{w'}))K(\b{w}-\b{w}')\mathrm{d}\b{w}\mathrm{d}\b{w'}\\
	&=\dfrac{1}{2}\int_{\Omega_d}\int_{\Omega_d}|u(\b{w})-u(\b{w'})|^2K(\b{w}-\b{w}')\mathrm{d}\b{w}\mathrm{d}\b{w'}.
	\end{align*}
	By properties of the Fourier transform, we have
	\begin{align*}
	\langle u,\mathcal{L}u\rangle_{L^2(\Omega_d)}=\langle \hat{u},\widehat{\mathcal{L}u}\rangle_{L^2(\Omega_d')}=\langle \hat{u},\sigma_{\mathcal{L}}({\b{\nu}})\hat{u}\rangle_{L^2(\Omega_d')}
	\end{align*}
	hence by \eqref{condition:L},
	\begin{align*}
	\langle u,\mathcal{L}u\rangle_{L^2(\Omega_d)}\leq C\||\b{\nu}|^{1/2}\hat{u}(\b{\nu})\|^2_{L^2(\Omega_d')}=C\|u\|_{\dot{H}^{1/2}(\Omega)}^2,\\
	\langle u,\mathcal{L}u\rangle_{L^2(\Omega_d)}\geq c\||\b{\nu}|^{1/2}\hat{u}(\b{\nu})\|^2_{L^2(\Omega_d')}=c\|u\|_{\dot{H}^{1/2}(\Omega)}^2.
	\end{align*}
	Thus \eqref{eq:normequivalence} holds.
\end{proof}

\section{Existence of minimizers}\label{sec:existence}
In this section, we will prove Theorem \ref{thm:existence}. Recall the energy functional defined in \eqref{def:energyfunctional}, i.e.
\begin{align*}
	F(u)&=\dfrac{1}{2}\int_{\Omega_d}\int_{\Omega_d}{|(u(x,\b{y})-u(x',\b{y}')|^2}K(x-x',\b{y}-\b{y}')\\
	&-{|(\eta(x,\b{y})-\eta(x',\b{y}')|^2}K(x-x',\b{y}-\b{y}')\mathrm{d}x\mathrm{d}\b{y}\mathrm{d}x'\mathrm{d}\b{y'}+\int_{\Omega_d}\gamma(u(x,\b y))\mathrm{d}x\mathrm{d}\b{y}.
\end{align*}
We first rewrite this energy functional. In fact, by Lemma \ref{lmm:propertyofL2}, if we denote $v=u-\eta$, then we can rewrite $F$ as
\begin{align}\label{def:redefine2}
F(u)&=\dfrac{1}{2}\int_{\Omega_d}\int_{\Omega_d}{|(\eta(x,\b{y})+v(x,\b{y})-\eta(x',\b{y}')-v(x',\b{y}'))|^2}K(x-x',\b{y}-\b{y}')\nonumber\\
&-{|(\eta(x,\b{y})-\eta(x',\b{y}')|^2}K(x-x',\b{y}-\b{y}')\mathrm{d}x\mathrm{d}\b{y}\mathrm{d}x'\mathrm{d}\b{y'}+\int_{\Omega_d}\gamma(u(w))\mathrm{d}\b{w}\nonumber\\
&=\dfrac{1}{2}\int_{\Omega_d}\int_{\Omega_d}{|v(x,\b{y})-v(x',\b{y}'))|^2}K(x-x',\b{y}-\b{y}')\mathrm{d}x\mathrm{d}\b{y}\mathrm{d}x'\mathrm{d}\b{y'}\nonumber\\
&+{2}c_{\mathcal{L}}\int_{\Omega_d}v(x,\b{y})(-\p_{xx})^{1/2}\eta(x)\mathrm{d}x\mathrm{d}\b{y}+\int_{\Omega_d}\gamma(u(w))\mathrm{d}\b{w}.
\end{align}   
Here $c_{\mathcal{L}}$ is the constant in Lemma \ref{lmm:propertyofL2}. From \eqref{def:redefine2} we see that subtraction of $\eta$ in the definition \eqref{def:energyfunctional} ensures that $F$ is finite if $u$ is bounded and satisfies $u-\eta\in H^{1/2}(\Omega)$. For the sake of convenience, we will switch between \eqref{def:energyfunctional} and \eqref{def:redefine2} when using functional $F$. 

The main idea in the proof of Theorem \ref{thm:existence} is to first slightly modify the minimizing problem on a subset of $\mathcal{A}$ denoted as $\mathcal{A}_I$. $\mathcal{A}_I$ is defined as
\begin{align}\label{def:setAI}
	\mathcal{A}_I:=\{u\in\mathcal{A}: u=\eta\ \mathrm{on}\ \Omega_d\setminus\Omega_d^I\}.
\end{align}
Here $\Omega_d^I=I\times\T^{d-1},\ I=(a,b)$ where $a<-1$ and $b>1$ are real numbers. By definition of minimizers, a minimizer $u_I\in\mathcal{A}_I$ solves the following Dirichlet problem in weak sense:
\begin{align}\label{eq:Dirichlet_2d}
\begin{cases}
\mathcal{L}u(x,\b{y})+\gamma'(u(x,\b{y}))=0,\ (x,\b{y})\in\Omega_d^I\\
u(x,\b{y})=1,\ x\in[b,+\infty),\\
u(x,\b{y})=-1,\ x\in(-\infty,a].\\
\end{cases}
\end{align} 
For a minimizer $u_I$, result similar to Theorem \ref{thm:existence} can be proved, which is summarized in the following proposition:
\begin{proposition}\label{prop:finiteexistence}
	Suppose that $\gamma\in C^2(\R)$ is a double-well {type} potential satisfying condition \eqref{condition:potential}. Define function set $\mathcal{A}$ as in \eqref{def:funcsetA} and energy functional $F$ as in \eqref{def:energyfunctional}. Then:
	\begin{enumerate}[(i)]
		\item (existence) There exists $u_I\in\mathcal{A}_I$ such that $F(u_I)=\min\limits_{u\in\mathcal{A}_I}F(u)$. In particular, $u_I$ is a weak solution to \eqref{eq:Dirichlet_2d}.
		\item (monotonicity) $u_I$ in $(i)$ satisfies that for any $\tau_1>0$, $u_I(x+\tau_1,\b y)\geq u_I(x,\b y)$ holds for a.e. $(x,\b{y})\in\Omega_d^I$, i.e. $u_I(x,y)$ is increasing in $x$ direction.
		\item (symmetry) $u_I$ in $(i)$ satisfies that for any $\b{\tau}_2\in\R^{d-1}$, $u_I(x,\b y+\b{\tau}_2)= u_I(x,\b y)$ holds for a.e. $(x,\b{y})\in\Omega_d^I$, i.e. $u_I(x,\b{y})=u_I(x)$ is a 1D profile. 
	\end{enumerate}
\end{proposition}
To prove monotonicity and symmetry, one needs to utilize a critical technique: the energy decreasing rearrangement method in \cite{palatucci2013local} which is based on the rearrangement inequality. After constructing $\{u_I\}$ in Proposition \ref{prop:finiteexistence}, a minimizer of $F(u)$ on $\mathcal{A}$ will be constructed using these minimizers on finite intervals and Theorem \ref{thm:existence} is ready to be proved.

Following this logic, we will first carefully introduce the energy decreasing rearrangement method utilized in \cite{palatucci2013local} in Section \ref{sec:rearrangement}. This tool is prepared for the proof of Proposition \ref{prop:finiteexistence} in Section \ref{sec:finiteinterval} which ensures existence of the minimizer on $\mathcal{A}_I$. Then in Section \ref{sec:technicallemmas}, we will introduce several technical lemmas before proving Theorem \ref{thm:existence} in Section \ref{sec:proofofexistence}.
\subsection{Energy decreasing rearrangement}\label{sec:rearrangement}
We denote $a_+=\max\{a,0\}$ and $a_-=-\min\{a,0\}$, i.e., $a_+$ and $a_-$ represent the positive part and the negative part of $a$ respectively. In this section, we will introduce the energy decreasing rearrangement method that is used in \cite{palatucci2013local}. In fact, this method relies on the following elementary equality:
\begin{lemma}\label{lmm:rearrangeinequality}
	(rearrangement) Suppose that $a_1,a_2,b_1,b_2$ are four real numbers. Denote $a=\min\{a_1,a_2\}, A=\max\{a_1,a_2\}, b=\min\{b_1,b_2\}$ and $B=\max\{b_1,b_2\}$. Then the following inequality holds:
	\begin{align}\label{ineq:rearrange}
	ab+AB-a_1b_1-a_2b_2 = (a_1-a_2)_+(b_1-b_2)_- + (a_1-a_2)_-(b_1-b_2)_+ \geq 0. 
	\end{align}
	In particular, $ab+AB-a_1b_1-a_2b_2=0$ if and only if $(a_1-a_2)(b_1-b_2)\geq 0$.
\end{lemma} 
Readers may refer to Appendix \ref{sec:proof} for proof of this inequality. Now we are ready to introduce the energy decreasing rearrangement method. 
\begin{lemma}\label{lmm:energydecreasing}
	(energy decreasing rearrangement in \cite{palatucci2013local}) Suppose that $u,v$ belong to set $\mathcal{A}$ which is defined as in \eqref{def:funcsetA}. Define $m(\b{w})=\min\{u(\b{w}),v(\b{w})\}$ and $M(\b{w})=\max\{u(\b{w}),v(\b{w})\}$. Then 
	\begin{align}\label{eq:energydecreasingrearrangement}
	&\ \ \ \ [F(u(\b{w}))+F(v(\b{w}))]-
	[F(m(\b{w}))+F(M(\b{w}))]\\\nonumber
	&=\int_{\Omega_d}\int_{\Omega_d}K(\b w-\b w')[(u-v)_+(\b w)(u-v)_-(\b w')+(u-v)_-(\b w)(u-v)_+(\b w')]\mathrm{d}\b w\mathrm{d}\b w'.
	\end{align}
	In particular, $F(u(\b w))+F(v(\b w))=F(M(\b w))+F(m(\b w))$ holds if and only if 
	\begin{align}
	(u(\b{w})-v(\b{w}))(u(\b{w}')-v(\b{w}'))\geq 0
	\end{align}
	holds for almost every $\b w,\b w'$ in $\Omega_d$, i.e. either $u(\b{w})\geq v(\b{w})$ or $u(\b{w})\leq v(\b{w})$ holds a.e. in $\Omega_d$.
\end{lemma}
\begin{proof}
	Recall energy functional defined in \eqref{def:energyfunctional}, i.e.
	\begin{align*}
	F(u)&=\dfrac{1}{2}\int_{\Omega_d}\int_{\Omega_d}{|(u(\b{w})-u(\b{w}')|^2}K(\b{w}-\b{w}')\\
	&-{|(\eta(\b{w})-\eta(\b{w}')|^2}K(\b{w}-\b{w}')\mathrm{d}\b{w}\mathrm{d}\b{w'}+\int_{\Omega_d}\gamma(u(\b{w}))\mathrm{d}\b{w}.
	\end{align*}
	since $\gamma(u(\b{w}))$ is a local term, we have 
	\begin{align*}
	\int_{\Omega_d}\gamma(u(\b{w}))+\gamma(v(\b{w}))\mathrm{d}\b{w}=\int_{\Omega_d}\gamma(m(\b{w}))+\gamma(M(\b{w}))\mathrm{d}\b{w}.
	\end{align*}
	So we only need to compare the convolution term. A straightforward calculation implies that 
	\begin{align*}
	&\ \ \ \ [F(u(\b{w}))+F(v(\b{w}))]-
	[F(m(\b{w}))+F(M(\b{w}))]\\
	&=\dfrac{1}{2}\int_{\Omega_d}\int_{\Omega_d}K(\b w-\b w')[|u(\b{w})-u(\b{w}')|^2+|v(\b{w})-v(\b{w}')|^2\\
	&-|m(\b{w})-m(\b{w}')|^2-|M(\b{w})-M(\b{w}')|^2]\mathrm{d}\b w\mathrm{d}\b w'.
	\end{align*}
	By the definition of $m$ and $M$, we know that
	\begin{align*}
	u(\b{w})^2+v(\b{w})^2=m(\b{w})^2+M(\b{w})^2
	\end{align*}
	holds for every $\b{w}\in\Omega_d$, thus 
	\begin{align*}
	&\ \ \ \ [F(u(\b{w}))+F(v(\b{w}))]-
	[F(m(\b{w}))+F(M(\b{w}))]\\
	&=\int_{\Omega_d}\int_{\Omega_d}K(\b w-\b w')[M(\b w)M(\b w')+m(\b w)m(\b w')-u(\b w)u(\b w')-v(\b w)v(\b w')]\mathrm{d}\b w\mathrm{d}\b w'.
	\end{align*}
	Let $a_1=u(\b{w}),\ a_2=v(\b{w}),\ b_1=u(\b{w}'),\ b_2=v(\b{w}')$, then in terms of Lemma \ref{lmm:rearrangeinequality} we know that
	\begin{align*}
	m(\b{w})=a,\ M(\b{w})=A,\  m(\b{w}')=b,\ M(\b{w}')=B.
	\end{align*}
	By Lemma \ref{lmm:rearrangeinequality}, we have
	\begin{align*}
	&\ \ \ \ [F(u(\b{w}))+F(v(\b{w}))]-
	[F(m(\b{w}))+F(M(\b{w}))]\\\nonumber
	&=\int_{\Omega_d}\int_{\Omega_d}K(\b w-\b w')[(u-v)_+(\b w)(u-v)_-(\b w')+(u-v)_-(\b w)(u-v)_+(\b w')]\mathrm{d}\b w\mathrm{d}\b w'.
	\end{align*}
	Thus \eqref{eq:energydecreasingrearrangement} holds. So in terms of integrating $\b{w},\b{w}'$, $F(u(\b w))+F(v(\b w))=F(m(\b w))+F(M(\b w))$ holds if and only if $(u(\b{w})-v(\b{w}))(u(\b{w}')-v(\b{w}'))\geq 0$ holds a.e. in $\Omega_d$. This concludes the proof.
\end{proof}
Lemma \ref{lmm:energydecreasing} ensures that if we are given two functions $u,v$ defined on $\Omega_d$, we can construct a pair $m,M$ such that they have a total energy less than $F(u)+F(v)$. Here comes the name of this tool: the energy decreasing property of this construction $m,M$ is realized by the precedent rearrangement (Lemma \ref{lmm:rearrangeinequality}), so we name it as energy decreasing rearrangement method. Now we are ready to prove Proposition \ref{prop:finiteexistence} using Lemma \ref{lmm:energydecreasing}.

\subsection{Relationship with the increasing rearrangement}

Clarification on this rearrangement technique is necessary to help readers distinguish it from other similar tools. Another rearrangement skill broadly utilized in the calculus of variations is the increasing rearrangement, which was first introduced in \cite{riesz1930inegalite}. Given $u:\R\to\R$ satisfying $\lim\limits_{x\to\pm\infty}u(x)=\pm 1$, the increasing rearrangement of $u$, denoted as $u^*$, is an increasing function with sublevel sets which are of same volume as those of $u$, i.e.,
\begin{align} \label{def:u*}
	\{x: t\leq u^*(x)\}=\{x: t\leq u(x)\}^*\ \mathrm{for}\  \mathrm{every}\ t\in(-1,1).
\end{align}
Here the rearrangement of a Borel set $A\in\R$, i.e. $A^*$, is defined as 
\begin{align} \label{def:rearrangmentofsets}
	A^*:=[c,\infty),\ c:=b-|A\cap [a,b]|\ \mathrm{for}\  \mathrm{every}\ A\ \mathrm{satisfying}\ [b,\infty)\subset A\subset [a,\infty).
\end{align} 
The machinery of the increasing rearrangement is exactly the same as that of the cumulative density function (CDF) matching approach. 
The measure-preserving property maintains local functional energies (e.g., the double-well potential), while the monotonicity reduces the convolution-type non-local functional energy (e.g., the reduction of the elastic energy on the slip plane).

Due to this energy reduction property, the increasing rearrangement is also employed to minimize functional energies that share common structures with $F$ in \eqref{def:energyfunctional} \cite{alberti1998nonlocal}. Although results derived in \cite{alberti1998nonlocal} are similar to ours, the context of \cite{alberti1998nonlocal} is much different in the sense that the convolution kernel $J(h)$ satisfies
\begin{align}
J(h)\in L^1(\R^d), J(h)|h|\in L^1(\R^d),
\end{align}
while in the current context we have 
\begin{align}
K(h)\sim|h|^{-d-1}\notin L^1(\R^d),\ K(h)|h|\sim |h|^{-d}\notin L^1(\R^d).
\end{align}
Therefore, the availability of the increasing rearrangement in our setting is indirect. Application of the increasing rearrangement was also considered in \cite{palatucci2013local} in which the authors commented that it worked for $(-\Delta)^{\alpha}, \alpha\in(1/2,1)$ instead of the critical case $(-\Delta)^{1/2}$ which is exactly in the PN model.

In contrast, originated in \cite{palatucci2013local}, Lemma \ref{lmm:energydecreasing} is powerful in this critical case (also other non-critical cases as discussed in \cite{palatucci2013local}) with a much simpler and elementary proof compared to the increasing rearrangement \cite{alberti2000some}. The main difference between Lemma \ref{lmm:energydecreasing} and the increasing rearrangement is that the former rearranges a pair of profiles $u$ and $v$ while the latter merely works on a single candidate $u$. Therefore, the hidden mechanisms of these two methods are totally different and readers should be aware of this discrepancy.

\subsection{Minimizers on finite intervals}\label{sec:finiteinterval}
Before proving Proposition \ref{prop:finiteexistence}, we introduce the translation-invariant property of the energy functional $F$ which is applied in the proof.
\begin{lemma}\label{lmm:translationinvariant}
	(translation invariant) Consider $F$ in \eqref{def:energyfunctional}. Then for any $(c_1,\b{c_2})\in\Omega_d$, we have 
	\begin{align*}
	F(u(x+c_1,\b{y}+\b{c_2}))=F(u(x,\b{y})),
	\end{align*}
	i.e. $F$ is invariant under any translation.
\end{lemma}
The proof of this lemma only relies on some elementary computations of integrals using Lemma \ref{lmm:propertyofL2}. Readers can refer to Appendix \ref{sec:proof} for detail. We will again use this invariant property later to prove the lower boundedness of $F$ on $\mathcal{A}$. 

Now we are ready to prove Proposition \ref{prop:finiteexistence} which addresses the minimizer of $F$ on the set $\mathcal{A}_I$ defined in \eqref{def:setAI}:
\begin{align*}
	\mathcal{A}_{I}=\{u\in\mathcal{A}:u=\eta\ \mathrm{on}\ \Omega_d\setminus\Omega_d^I\}.
\end{align*}
We aim to prove that there exists a minimizer $u_I$ of $F$ on set $\mathcal{A}_I$. 
\begin{proof}[Proof of Proposition \ref{prop:finiteexistence}]
	First we prove statement (i). Notice that for any $u\in\mathcal{A}_I$, $F(u)$ is uniformly bounded from below by a constant that depends on $\eta(x)$, i.e.
	\begin{align*}
		F(u) & =\dfrac{1}{2}\int_{\Omega_d}\int_{\Omega_d}(|(u(x,\b{y})-u(x',\b{y}')|^2)K(x-x',\b{y}-\b{y}')\\
		&-{|(\eta(x,\b{y})-\eta(x',\b{y}')|^2}K(x-x',\b{y}-\b{y}')\mathrm{d}x\mathrm{d}x'\mathrm{d}\b{y}\mathrm{d}\b{y}'+\int_{\Omega_d}\gamma(u(x,\b{y}))\mathrm{d}x\mathrm{d}\b{y}\\
		&\geq -\dfrac{1}{2}\int_{\Omega_d^I}\int_{\Omega_d^I}{|(\eta(x,\b{y})-\eta(x',\b{y}')|^2}K(x-x',\b{y}-\b{y}')\mathrm{d}x\mathrm{d}x'\mathrm{d}\b{y}\mathrm{d}\b{y}'\\
		&-\int_{(\Omega_d^I)^c}\int_{\Omega_d^I}{|(\eta(x,\b{y})-\eta(x',\b{y}')|^2}K(x-x',\b{y}-\b{y}')\mathrm{d}x\mathrm{d}x'\mathrm{d}y\mathrm{d}y'\\
		&= -\dfrac{c_{\mathcal{L}}}{2}\int_{I}\int_{I}\dfrac{(\eta(x)-\eta(x'))^2}{(x-x')^2}\mathrm{d}x\mathrm{d}x'-c_{\mathcal{L}}\int_{I}\int_{I^c}\dfrac{(\eta(x)-\eta(x'))^2}{(x-x')^2}\mathrm{d}x'\mathrm{d}x>-\infty.
	\end{align*}
	The last equality is by Lemma \ref{lmm:propertyofL2}. $c_{\mathcal{L}}>0$ is the constant in Lemma \ref{lmm:propertyofL2}. Therefore, there exists a minimizing sequence $\{u_n\}\subset\mathcal{A}_I$ such that 
	\begin{align}
		F(u_n)\to C_I:=\inf\limits_{u\in\mathcal{A}_I}F(u)\ \mathrm{as}\ n\to\infty.
	\end{align} 
	
	For any $u\in\mathcal{A}$, consider 
	\begin{align*}
		\tilde{u}=\max\{\min\{u,1\},-1\},
	\end{align*}
	i.e. $\tilde{u}$ is the cut-off of $u$ from below by $-1$ and from above by 1. Then $\tilde{u}\in\mathcal{A}$ and satisfies
	\begin{align*}
		F(\tilde{u})\leq F(u)
	\end{align*}
	by definition of $F$. So we can assume $|u_n|\leq 1$ without loss of generality. 
	
	Denote $v_n=u_n-\eta$. Then $v_n$ is supported on $\Omega_d^I$ since $u_n\in\mathcal{A}_I$. This indicates that $\|v_n\|_{L^2(\Omega_d)}^2\leq 4(b-a)$, i.e. $\{v_n\}$ is uniformly bounded in $L^2(\Omega_d)$. Moreover, $\{v_n\}$ is also uniformly bounded in $H^{1/2}(\Omega_d)$ since by Lemma \ref{lmm:propertyL3}, 
	\begin{align*}
		\|v_n\|_{\dot{H}^{1/2}(\Omega_d)}^2 &\leq  \dfrac{1}{c_2}\int_{\Omega}\int_{\Omega}{|(v_n(x,\b{y})-v_n(x',\b{y}')|^2}K(x-x',\b{y}-\b{y}')\mathrm{d}x\mathrm{d}x'\mathrm{d}\b{y}\mathrm{d}\b{y}'.		
	\end{align*}
	Meanwhile, using the definition of $F$ in \eqref{def:redefine2} and the Cauchy-Schwartz inequality, we have
	\begin{align*}
	&\ \ \  \dfrac{1}{c_2}\int_{\Omega}\int_{\Omega}{|(v_n(x,\b{y})-v_n(x',\b{y}')|^2}K(x-x',\b{y}-\b{y}')\mathrm{d}x\mathrm{d}x'\mathrm{d}\b{y}\mathrm{d}\b{y}\\
	&= \dfrac{2}{c_2}F(u_n)-\dfrac{{4c_{\mathcal{L}}}}{c_2}\int_{\Omega_d}v_n(x,\b{y})(-\p_{xx})^{1/2}\eta(x)\mathrm{d}x\mathrm{d}\b{y}-\dfrac{2}{c_2}\int_{\Omega_d}\gamma(u)\mathrm{d}\b{w}\\
	&\leq \dfrac{2}{c_2}F(u_n)+\dfrac{{2c_{\mathcal{L}}}}{c_2}\|v_n\|^2_{L^2(\Omega_d)}+\dfrac{{2c_{\mathcal{L}}}}{c_2}\|(-\p_{xx})^{1/2}\eta(x)\|^2_{L^2(\R)}.\\
	&\leq C'.
	\end{align*}
	Here $c_{\mathcal{L}}$ is the constant in Lemma \ref{lmm:propertyofL2} and $c_2$ is the constant in Lemma \ref{lmm:propertyL3}. $C'$ is a constant that only depends on $a,b$ and $\eta$ but independent with any certain minimizing sequence. Therefore, $\{v_n\}$ is uniformly bounded in $H^{1/2}(\Omega_d)$.
	
	Now we are ready to prove that $u_I$ is indeed a minimizer. Uniform boundedness of $\{v_n\}$ in $H^{1/2}(\Omega_d)$ implies that there exists $v_I\in H^{1/2}(\Omega_d)$ supported on $\Omega_d^I$ such that $v_n\rightharpoonup v_I$ in $H^{1/2}(\Omega_d)$. Hence $u_I:=v_I+\eta\in\mathcal{A}_I$ and up to a subsequence,
	\begin{align*}
		u_n&\to u_I,\ v_n\to v_I\quad\mathrm{a.e.}\ \mathrm{in}\  \Omega_d,\\
		u_n&\to u_I,\ v_n\to v_I\quad\mathrm{in}\ L^2(\Omega_d^I).
	\end{align*}
	Therefore, by Fatou's lemma, the strong $L^2$ convergence and the definition of $F$ in \eqref{def:redefine2}, we know that 
	\begin{align*}
	C_I &= \liminf\limits_{n\to\infty}F(u_n)\\
	    &= \liminf\limits_{n\to\infty} \dfrac{1}{2}\int_{\Omega_d}\int_{\Omega_d}{|(v_n(x,\b{y})-v_n(x',\b{y}')|^2}K(x-x',\b{y}-\b{y}')\mathrm{d}x\mathrm{d}x'\mathrm{d}\b{y}\mathrm{d}\b{y}'\\
	    &+{2c_{\mathcal{L}}}\int_{\Omega_d}v_n(x,\b{y})(-\p_{xx})^{1/2}\eta(x)\mathrm{d}x\mathrm{d}\b{y}+\int_{\Omega_d}\gamma(u_n)\mathrm{d}x\mathrm{d}\b{y}\\
	    &\geq \dfrac{1}{2}\int_{\Omega_d}\int_{\Omega_d}{|(v_I(x,\b{y})-v_I(x',\b{y}')|^2}K(x-x',\b{y}-\b{y}')\mathrm{d}x\mathrm{d}x'\mathrm{d}\b{y}\mathrm{d}\b{y}'\\
	    &+{2c_{\mathcal{L}}}\int_{\Omega_d}v_I(x,\b{y})(-\p_{xx})^{1/2}\eta(x)\mathrm{d}x\mathrm{d}\b{y}+\int_{\Omega_d}\gamma(u_I)\mathrm{d}x\mathrm{d}\b{y}\\
	    &= F(u_I)\geq C_I.
	\end{align*} 
	Thus $u_I\in\mathcal{A}_I$ is indeed a minimizer of $F$ on set $\mathcal{A}_I$. In particular, it is a weak solution to \eqref{eq:Dirichlet_2d} by a simple calculation of the first variation of the energy functional $F$. This proves (i).
	
	We will use the energy decreasing rearrangement method (Lemma \ref{lmm:energydecreasing}) to prove (ii) and (iii). First we prove (ii). For any given $\tau>0$, consider $v(x,\b{y})=u_I(x+\tau,\b{y})$. Denote $m(x,\b{y})=\min\{u_I(x,\b{y}),v(x,\b{y})\}$ and $M(x,\b{y})=\max\{u_I(x,\b{y}),v(x,\b{y})\}$. Then by Lemma \ref{lmm:energydecreasing}, we know that 
	\begin{align*}
		F(m)+F(M)\leq F(u_I)+F(v).
	\end{align*}
	
	This inequality is in fact an equality. Notice that $M(x,\b{y})=1$ if $x\geq b-\tau$ and $M(x,\b{y})=-1$ if $x\leq a-\tau$, so $M(x,\b{y})\in\mathcal{A}_{(a-\tau,b-\tau)}$. By the translation invariant property (Lemma \ref{lmm:translationinvariant}), we know that $v(x,\b{y})=u_I(x+\tau,\b{y})$ is in fact a minimizer of $F$ on $\mathcal{A}_{(a-\tau,b-\tau)}.$ Thus
	\begin{align*}
		F(M)\geq F(v).
	\end{align*}
	Note that $m(x,\b{y})=1$ if $x\geq b$ and $m(x,\b{y})=-1$ if $x\leq a$, so $m(x,\b{y})\in\mathcal{A}$. Then by minimality of $u_I$ we know that 
	\begin{align*}
		F(m)\geq F(u_I).
	\end{align*} 
	Therefore, we have 
	\begin{align*}
		F(v)+F(u_I)\leq F(m)+F(M) \leq F(u_I)+F(v).
	\end{align*}
	Thus  
	\begin{align*}
		F(v)+F(u_I)=F(m)+F(M).
	\end{align*}
	By Lemma \ref{lmm:energydecreasing}, this equality holds if and only if either $u_I(x,\b{y})\geq v(x,\b{y})$ or $u_I(x,\b{y})\leq v(x,\b{y})$ holds almost surely in $\Omega_d$. By the boundary condition and that $|u_I|\leq 1$, we know that the former is true, i.e.
	\begin{align*}
		u_I(x,\b{y})\leq v(x,\b{y})=u_I(x+\tau,\b{y}).
	\end{align*}
	This inequality holds for a.e. $(x,\b y)\in \Omega_d^I$ for arbitrary $\tau>0$. This proves (ii).
	
	Eventually, we prove (iii). Again we will adopt Lemma \ref{lmm:energydecreasing}, i.e. the energy decreasing rearrangement method. Unlike the case in the proof of (ii) where we only consider translation in $x$ direction, we consider translation in both $x$ and $\b{y}$ direction, but with $x$ direction still positive. For any given $(\tau_1,\b{\tau_2})$ such that $\tau_1>0,\b{\tau_2}\in\T^{d-1}$, consider $w(x,\b{y})=u_I(x+\tau_1,\b{y}+\b{\tau_2})$. As in the proof of (ii), by considering the minimum and maximum of $u_I$ and $w$, we conclude that
	\begin{align}\label{eq:compare}
		 u_I(x+\tau_1,\b{y}+\b{\tau_2})\geq u_I(x,\b{y})
	\end{align}
	holds for almost every $(x,\b{y}) \in \Omega_d^I$. 
	
	Now let $\tau_1\to 0$. For any $\b{w}\in\Omega_d$, denote $S_{\epsilon}(\b{w})$ the square with length $\epsilon$ centered at $\b{w}$. Then for any $(x,\b{y})\in\Omega_{I}$ (not almost every but every) and $\epsilon>0$, by inequality \eqref{eq:compare}, we have
	\begin{align*}
		\dfrac{1}{\epsilon^d}\int_{S_{\epsilon}(x,\b{y})}u_I(s,\b{t}+\b{\tau_2})\mathrm{d}s\mathrm{d}\b{t}&=\lim\limits_{\tau_1\to 0^+}\dfrac{1}{\epsilon^d}\int_{S_{\epsilon}(x,\b{y})}u_I(s+\tau_1,\b{t}+\b{\tau_2})\mathrm{d}s\mathrm{d}\b{t}\\
		&\geq\dfrac{1}{\epsilon^d}\int_{S_{\epsilon}(x,\b{y})}u_I(s,\b{t})\mathrm{d}s\mathrm{d}\b{t}
	\end{align*}
	Then let $\epsilon\to 0$ and by Lebesgue's differential theorem, we have
	\begin{align*}
		u_I(x,\b{y}+\b{\tau_2})\geq u(x,\b{y})
	\end{align*}
	holds for a.e. $(x, \b y)\in\Omega_d^I$. This holds for arbitrary $\b{\tau_2}\in\T^{d-1}$ without specific assignment of sign of each component. Then taking both $\b{\tau_2}$ and $-\b{\tau_2}$ in the translation concludes that 
	\begin{align*}
		u_I(x,\b{y}+\b{\tau_2})=u_I(x,\b{y})
	\end{align*}
	holds for a.e. $(x, \b y)\in\Omega_d^I$. This closes the whole proof of Proposition \ref{prop:finiteexistence}.
\end{proof}

\subsection{Technical lemmas}\label{sec:technicallemmas}
Before proving the existence theorem, i.e. Theorem \ref{thm:existence}, we will first provide several technical lemmas whose proofs are attached in Appendix \ref{sec:proof}. These lemmas finally lead to the fact that $F$ is lower bounded on $\mathcal{A}$. This enables the application of the direct method in calculus of variations in the proof of Theorem \ref{thm:existence}. 

Lemma \ref{lmm:approx} addresses an approximation property: 
\begin{lemma}\label{lmm:approx} 
	For any $u\in\mathcal{A}$ such that $|u|\leq 1$, there exist a sequence $\{u_n\}\subset\mathcal{A}$ and positive constants $\{M_n\}$ such that
	\begin{align*}
		 u_n-\eta\in C^{\infty}(\Omega_d),\  u_n=\eta\ \mathrm{on}\ |x|>M_n,
	\end{align*}
	and
	\begin{align*}
		F(u_n)\to F(u)\ \mathrm{as}\ n\to\infty.
	\end{align*}
\end{lemma}
\begin{proof}
	Given $u\in\mathcal{A}$ such that $|u|\leq 1$, because $u-\eta\in H^{1/2}(\Omega_d)$, so standard density argument (see \cite{adams2003sobolev,stein1970singular}) claims that there exists ${u_n}$ and $\{M_n\}$ such that $u_n-\eta\in C^{\infty}(\Omega_d)$, $u_n=\eta$ on $|x|>M_n$ and $\|(u-\eta)-(u_n-\eta)\|_{H^{1/2}(\Omega_d)}^2\to 0.$ as $n\to\infty$. Therefore,
	\begin{align*}
	u_n-\eta&\to u-\eta\quad\mathrm{in}\ H^{1/2}(\Omega_d),\\
	u_n-\eta&\to u-\eta\quad\mathrm{in}\ L^2(\Omega_d).
	\end{align*}
	Denote $v=u-\eta$, $v_n=u_n-\eta$. Then by \eqref{def:redefine2}, we have
	\begin{equation}\label{eq:rewriteF}
		\begin{aligned}
			F(u)&=\dfrac{1}{2}\int_{\Omega_d}\int_{\Omega_d}|v(\b{w})-v(\b{w'})|^2K(\b{w}-\b{w}')\mathrm{d}\b{w}\mathrm{d}\b{w'}\\
			&+c_{\mathcal{L}}\int_{\Omega_d}(u-\eta)(-\p_{xx})^{1/2}\eta(x)\mathrm{d}x\mathrm{d}\b{y}+\int_{\Omega_d}\gamma(u)\mathrm{d}x\mathrm{d}\b{y},
		\end{aligned}
	\end{equation} 
	here $c_{\mathcal{L}}$ is the constant in Lemma \ref{lmm:propertyofL2}. Then by Lemma \ref{lmm:propertyL3} and convergence in $H^{1/2}(\Omega_d)$ and $L^2(\Omega_d)$, we have
	\begin{align*}
		\int_{\Omega_d}\int_{\Omega_d}|v_n(\b{w})-v_n(\b{w'})|^2K(\b{w}-\b{w}')\mathrm{d}\b{w}\mathrm{d}\b{w'}&\to\int_{\Omega_d}\int_{\Omega_d}|v(\b{w})-v(\b{w'})|^2K(\b{w}-\b{w}')\mathrm{d}\b{w}\mathrm{d}\b{w'}\\
		\int_{\Omega_d}(u_n-\eta)(-\p_{xx})^{1/2}\eta(x)\mathrm{d}x\mathrm{d}\b{y}&\to\int_{\Omega_d}(u-\eta)(-\p_{xx})^{1/2}\eta(x)\mathrm{d}x\mathrm{d}\b{y}
	\end{align*}
	as $n\to \infty$.
	
	For the non-linear potential term in \eqref{eq:rewriteF}, the mean value theorem ensures that there exist $\theta(x,y)$ and $\tilde{\theta}(x,y)\in[0,1]$ such that
	\begin{align*}
		&\ \ \  \left|\int_{\Omega_d}\gamma(u_n)-\gamma(u)\mathrm{d}x\mathrm{d}\b{y}\right|\\
		&\leq \int_{\Omega_d}|\gamma'(\theta u+(1-\theta)u_n)||u-u_n|\mathrm{d}x\mathrm{d}\b{y}\\
		&\leq \int_{\Omega_d}|\gamma'(\eta+\theta (u-\eta)+(1-\theta)(u_n-\eta))||u-u_n|\mathrm{d}x\mathrm{d}\b{y}\\
		&\leq \int_{\Omega_d}|\gamma'(\eta)||u-u_n|\mathrm{d}x\mathrm{d}\b{y}\\
		&+\int_{\Omega_d}|\gamma''(\eta+\tilde{\theta}\theta (u-\eta)+\tilde{\theta}(1-\theta)(u_n-\eta))||\theta (u-\eta)+(1-\theta)(u_n-\eta)||u-u_n|\mathrm{d}x\mathrm{d}\b{y}.
	\end{align*}
	Because $|u|\leq 1$, we can assume $|u_n|\leq 1$ without loss of generality. Thus
	\begin{align*}
		|\gamma''(\eta+\tilde{\theta}\theta (u-\eta)+\tilde{\theta}(1-\theta)(u_n-\eta))|
	\end{align*}
	is uniformly bounded in $\Omega_d$. Also notice that $\gamma'(\eta)=0$ for $|x|>1$, then by the Cauchy-Schwartz inequality, there exists $C>0$ that only depends on $u,\eta$ such that
	\begin{align*}
		\left|\int_{\Omega_d}\gamma(u_n)-\gamma(u)\mathrm{d}x\mathrm{d}\b{y}\right|&\leq\int_{-1}^1\int_{\T^{d-1}}|\gamma'(\eta)||u-u_n|\mathrm{d}\b{y}\mathrm{d}x\\
		&+C(\|u-\eta\|_{L^2(\Omega_d)}+\|u_n-\eta\|_{L^2(\Omega_d)})\|u-u_n\|_{L^2(\Omega_d)}\\
		&\leq C'\|u-u_n\|_{L^2(\Omega_d)}\to 0. 
	\end{align*}
	Here $C'$ is a constant that only depends on $\gamma,\eta$ and $u$. This closes the proof.
\end{proof}
The following lemma claims that we can use the nonlinear potential to control $L^2$ norm of $u-\eta$.
\begin{lemma}\label{lmm:l2bound}
	Suppose that $u$ is a non-decreasing function on $\R$ such that $u(x)=v(x)+\eta(x)$ is non-decreasing, $|u(x)|\leq 1$ for all $x\in\R$ and $u(0)=0$. $\gamma\in C^2(\R)$ satisfies \eqref{condition:potential}. Then there exist constants $C_1$ and $C_2$ such that 
	\begin{align*}
	\int_{\R}\gamma(u(x))\mathrm{d}x+C_1\geq C_2\|v\|_{L^2}^2. 
	\end{align*}
	Here $C_1>0$ and $C_2>0$ only depend on $\gamma(x)$ and are independent with $v$.
\end{lemma} 
\begin{proof}
	According to \eqref{condition:potential}, $\gamma''(\pm 1)>0$ and $\gamma$ attains strict minimum at $-1$ and $1$, so there exists $C_1>0$ such that 
	\begin{align*}
		&\gamma(x)\geq C_1(x-1)^2,\ \mathrm{if}\ x\in[0,1],\\
		&\gamma(x)\geq C_1(x+1)^2,\ \mathrm{if}\ x\in[-1,0].
	\end{align*}
	Remember that $u(x)$ is non-decreasing, $u(0)=0$ and $-1\leq \eta(x)\leq 1$, so 
	\begin{align*}
		-1&\leq v(x)\leq 0,\ \mathrm{if}\ x\geq 1,\\
		0&\leq v(x)\leq 1,\ \mathrm{if}\ x\leq -1.
	\end{align*}
	Therefore, we have
	\begin{align*}
	\int_{\R}\gamma(u(x))\mathrm{d}x &\geq \int_{-\infty}^{-1}\gamma(v(x)-1)\mathrm{d}x+\int_{1}^{+\infty}\gamma(v(x)+1)\mathrm{d}x\\
	&\geq C_1\int_{-\infty}^{-1}v(x)^2\mathrm{d}x+C_1\int_{1}^{+\infty}v(x)^2\mathrm{d}x\\
	&\geq C_1\|v\|_{L_2}^2-2C_1.
	\end{align*}
\end{proof}
Using these technical lemmas, we are ready to prove Theorem \ref{thm:existence}.

\subsection{Proof of Theorem \ref{thm:existence}: existence of the minimizers}\label{sec:proofofexistence}
As stated in previous sections, we will use the calculus of variations to prove Theorem \ref{thm:existence} by minimizing $F$ on set $\mathcal{A}$. Proved in Proposition \ref{prop:finiteexistence}, a key property of the minimizers $u_I$ is that for and $\tau_1>0,\ \b{\tau}_2\in\T^{d-1}$,
\begin{equation}\label{condition:key}
	\begin{aligned}
	u_I(x+\tau_1,\b y)&\geq u_I(x,\b y)\ a.e.\ \mathrm{in}\ \Omega_d,\\
	u_I(x,\b y+\b{\tau}_2)&=u_I(x,\b y)\ a.e.\ \mathrm{in}\ \Omega_d.
	\end{aligned}
\end{equation}
To prove lower boundedness of $F$ with the help of \eqref{condition:key}, we consider the following subset of $\mathcal{A}$ which is much finer than $\mathcal{A}$:
\begin{align}\label{def:setB}
\mathcal{B}:=\left\{u\in\mathcal{A}:\ u\  \mathrm{satisfies}\ \eqref{condition:key},\ |u|\leq 1\ \mathrm{and}\ u(0)=0\right\}.
\end{align} 
This definition is inspired by Proposition \ref{prop:finiteexistence} and preceding technical lemmas: according to Proposition \ref{prop:finiteexistence}, we know that $u_I\in\mathcal{B}$ if $u_I(0)=0$. Here $u_I$ is the minimizer of $F$ on $\mathcal{A}_I$ which is constructed in Proposition \ref{prop:finiteexistence}. Through the bridge of set $\mathcal{B}$, we will prove that:
\begin{lemma}\label{lmm:energylowerbound}
Consider set $\mathcal{A}$ in \eqref{def:funcsetA}, set $\mathcal{B}$ in \eqref{def:setB}, and functional energy $F$ in \eqref{def:energyfunctional}. Then: 
	\begin{enumerate}[(i)]
		\item $\inf\limits_{u\in\mathcal{A}}F(u)=\inf\limits_{u\in\mathcal{B}}F(u)$.
		\item There exist positive constants $C_3$ and $C_4$ that only depend on $\eta$ and $\gamma$ such that for any $u\in\mathcal{B}$,
		\begin{align*}
		F(u)\geq C_3\|u-\eta\|^2_{H^{1/2}(\Omega_d)}-C_4.
		\end{align*}
	\end{enumerate}
\end{lemma}
\begin{proof}
	We first prove (i). Because $\mathcal{B}\subset\mathcal{A}$, so we have $\inf\limits_{u\in\mathcal{A}}F(u)\leq\inf\limits_{u\in\mathcal{B}}F(u)$. Hence we only need to prove that $\inf\limits_{u\in\mathcal{A}}F(u)\geq\inf\limits_{u\in\mathcal{B}}F(u)$.
	
	Consider $\tilde{u}=\max\{\min\{u,1\},-1\}$, i.e. the cut-off of $u$ by 1 from above and by $-1$ from below. Then $\tilde{u}$ is also in $\mathcal{A}$ and satisfies that $F(\tilde{u})\leq F(u)$. So we only need to consider those $u\in\mathcal{A}$ such that $|u|\leq 1$.
	
	By Lemma \ref{lmm:approx}, for any $\epsilon>0$, there exists $u_1\in C^{\infty}(\Omega)$ and $M>0$ such that $u_1=\eta$ on $|x|>M$ and 
	\begin{align*}
		F(u)>F(u_1)-\epsilon.
	\end{align*}
	Then according to the definition of $\mathcal{A}_I$ in \eqref{def:setAI}, we know that $u_1\in\mathcal{A}_I$. By Proposition \ref{prop:finiteexistence}, we know that $F(u_1)\geq F(u_M)$ where $u_M$ is a minimizer of $F$ on $\mathcal{A}_I$ satisfying that $u_M$ is a 1D profile and increasing in $x$ direction. Therefore
	\begin{align*}
	F(u)\geq F(u_1)-\epsilon\geq F(u_M)-\epsilon.
	\end{align*}   
	By the translation-invariant property (Lemma \ref{lmm:translationinvariant}), we have 
	\begin{align*}
		F(u_M)=F(u^*_M)
	\end{align*}
	where $u^*_M$ is a translation of $u_M$ that crosses $(0,0)$, i.e.
	\begin{align*}
		u^*_M(x)=u_M(x+c),\ u^*_M(0)=0.
	\end{align*}
	By definition, we know that $u^*_M\in\mathcal{B}$. Thus for any $u\in\mathcal{A}$ and $\epsilon>0$, there exists $u^*_M\in\mathcal{B}$ such that
	\begin{align*}
	F(u)\geq F(u)-\epsilon\geq F(u_M^*)-\epsilon.
	\end{align*}
	Thus $\inf\limits_{u\in\mathcal{A}}F(u)\geq\inf\limits_{u\in\mathcal{B}}F(u)-\epsilon$. By arbitrariness of $\epsilon$, we have $ \inf\limits_{u\in\mathcal{A}}F(u)\geq\inf\limits_{u\in\mathcal{B}}F(u)$.
	
	Now we prove (ii). By Lemma \ref{lmm:l2bound}, for any $u\in\mathcal{B}$, there exist $C_1$ and $C_2$ that only depend on $\gamma$ such that 
	\begin{align*}
	\int_{\Omega_d}\gamma(u)\mathrm{d}x\mathrm{d}\b{y}\geq  C_1\|v\|_{L^2(\Omega_d)}^2-C_2
	\end{align*}
	where $v=u-\eta$. Therefore, using the expression of $F$ in \eqref{def:redefine2}, the Cauchy-Schwartz inequality, Lemma \ref{lmm:propertyL3} and Lemma \ref{lmm:propertyofL2}, we have
	\begin{align*}
	F(u)
	&=\dfrac{1}{2}\int_{\Omega_d}\int_{\Omega_d}|v(\b{w})-v(\b{w}')|^2K(\b{w}-\b{w}')\mathrm{d}\b{w}\mathrm{d}\b{w}'+{2c_{\mathcal{L}}}\int_{\Omega_d}v(x,\b{y}){(-\p_{xx})^{1/2}}\eta(x)\mathrm{d}x\mathrm{d}\b{y}\\
	&+\int_{\Omega_d}\gamma(u(x,\b{y}))\mathrm{d}x\mathrm{d}\b{y}\\
	&\geq \dfrac{c_2}{2}\|v\|^{2}_{\dot{H}^{1/2}(\Omega_d)}+C_1\|v\|_{L^2}^2-C_2-{\dfrac{2c_{\mathcal{L}}^2}{C_1}}\|(-\p_{xx})^{1/2}\eta(x)\|^2_{L^2(\Omega_d)}-\dfrac{C_1}{2} \|v\|_{L^2}^2\nonumber\\
	&\geq C_3\|v\|^2_{{H}^{1/2}(\Omega_d)}-C_4. 
	\end{align*}
	Here $c_2$ is the constant in Lemma \ref{lmm:propertyL3}, $c_{\mathcal{L}}$ is the constant in Lemma \ref{lmm:propertyofL2}, and
	\begin{align*}
		C_3=\min\left\{\dfrac{c_2}{2},\dfrac{C_1}{2}\right\},C_4=C_2+{\dfrac{2c_{\mathcal{L}}^2}{C_1}}\|(-\p_{xx})^{1/2}\eta(x)\|^2_{L^2(\Omega_d)}
	\end{align*} 
	are constants that only depend on $\eta$, $\gamma$ and the operator $\mathcal{L}$. This concludes the proof. 
\end{proof}
Lemma \ref{lmm:energylowerbound} in fact provides insightful corollaries: first, we have 
\begin{align*}
	\inf\limits_{u\in\mathcal{A}}F(u)=\inf\limits_{u\in\mathcal{B}}F(u)\geq -C_4.
\end{align*}
Thus $F$ is lower bounded in $\mathcal{A}$. Moreover, according to (ii), functional $F(u)$ can be used to bound $H^{1/2}(\Omega_d)$ norm of $u-\eta$ for any $u\in\mathcal{B}$. In the proof of Theorem \ref{thm:existence}, this observation will be used to find an a.e. limit of the minimizing sequence which is proved to be a minimizer of $F$ on $\mathcal{A}$. Now we are ready to prove Theorem \ref{thm:existence}.
\begin{proof}[Proof of Theorem \ref{thm:existence}]
We will first prove (i). By Lemma \ref{lmm:energylowerbound}, we know that
\begin{align*}
	\inf\limits_{u\in\mathcal{A}}F(u)=\inf\limits_{u\in\mathcal{B}}F(u)\geq -C_4>-\infty.
\end{align*}
Denote $c=\inf\limits_{u\in\mathcal{A}}F(u)=\inf\limits_{u\in\mathcal{B}}F(u)$. Then there exists $\{u_n\}\subset\mathcal{B}$ such that $F(u_n)\to c$ as $n\to\infty$. Again by Lemma \ref{lmm:energylowerbound} (ii), we know that $\|u_n-\eta\|_{H^{1/2}(\Omega_d)}$ is uniformly bounded. Thus there exists $u^*$ such that up to a subsequence,
\begin{align*}
	u_n-\eta&\to u^*-\eta\quad\mathrm{a.e.}\ \mathrm{in}\ \Omega_d,\\
	u_n-\eta&\rightharpoonup u^*-\eta\quad\mathrm{in}\   H^{1/2}(\Omega_d).
\end{align*}
Denote $v_n=u_n-\eta$ and $v^*=u^*-\eta$. Then $v_n\to v^*$ a.e. in $\Omega_d$ and $v_n\rightharpoonup v^*$ in $H^{1/2}(\Omega_d)$.

In fact, $u^*$ is a minimizer of $F$ on $\mathcal{A}$. By Fatou's lemma, we know that
\begin{align}\label{ineq:fatou}
	\liminf\limits_{n\to\infty}\dfrac{1}{2}\int_{\Omega_d}\int_{\Omega_d}{|(v_n(x,\b{y})-v_n(x',\b{y}')|^2}K(x-x',\b{y}-\b{y}')\mathrm{d}x\mathrm{d}\b{y}\mathrm{d}x'\mathrm{d}\b{y}'+\int_{\Omega_d}\gamma(u_n(x,\b{y}))\mathrm{d}x\mathrm{d}\b{y}\nonumber\\
	\geq  \dfrac{1}{2}\int_{\Omega_d}\int_{\Omega_d}{|(v^*(x,\b{y})-v^*(x',\b{y}')|^2}K(x-x',\b{y}-\b{y}')\mathrm{d}x\mathrm{d}\b{y}\mathrm{d}x'\mathrm{d}\b{y}'+\int_{\Omega_d}\gamma(u^*(x,\b{y}))\mathrm{d}x\mathrm{d}\b{y}.
\end{align} 
Meanwhile, since $v_n\to v^*$ weakly in $H^{1/2}(\Omega_d)$, hence also converges weakly in $L^2(\Omega_d)$, thus
\begin{align}\label{eq:weakl2}
	\lim\limits_{n\to\infty}\int_{\Omega_d}v_n(x,\b{y})(-\p_{xx})^{1/2}\eta(x)\mathrm{d}x\mathrm{d}\b{y}=\int_{\Omega_d}v^*(x,\b{y})(-\p_{xx})^{1/2}\eta(x)\mathrm{d}x\mathrm{d}\b{y}.
\end{align}
Substituting \eqref{ineq:fatou} and \eqref{eq:weakl2} into the following equality, we have
\begin{align*}
c &= \liminf\limits_{n\to\infty}F(u_n)\\
&=\liminf\limits_{n\to\infty}\dfrac{1}{2}\int_{\Omega_d}\int_{\Omega_d}{|(v_n(x,\b{y})-v_n(x',\b{y}')|^2}K(x-x',\b{y}-\b{y}')\mathrm{d}x\mathrm{d}\b{y}\mathrm{d}x'\mathrm{d}\b{y}'\\
&+{2c_{\mathcal{L}}}\int_{\Omega_d}v_n(x,\b{y})(-\p_{xx})^{1/2}\eta(x)\mathrm{d}x\mathrm{d}\b{y}+\int_{\Omega_d}\gamma(u_n(x,\b{y}))\mathrm{d}x\mathrm{d}\b{y}\\
&\geq \dfrac{1}{2}\int_{\Omega_d}\int_{\Omega_d}{|(v^*(x,\b{y})-v^*(x',\b{y}')|^2}K(x-x',\b{y}-\b{y}')\mathrm{d}x\mathrm{d}\b{y}\mathrm{d}x'\mathrm{d}\b{y}'\\
&+{2c_{\mathcal{L}}}\int_{\Omega_d}v^*(x,\b{y})(-\p_{xx})^{1/2}\eta(x)\mathrm{d}x\mathrm{d}\b{y}+\int_{\Omega_d}\gamma(u^*(x,\b{y}))\mathrm{d}x\mathrm{d}\b{y}\\
&=F(u^*).
\end{align*}
So $u^*$ is in fact a minimizer of $F$ on $\mathcal{A}$. In particular, it is a weak solution to \eqref{eq:generalized1d}. This proves (i).

Now we prove (ii). Because $u_n\in\mathcal{B}$, so $|u_n|\leq 1$ and they are all 1D functions and non-decreasing in $x$ direction, so the a.e. limit $u^*$ is also non-decreasing in $x$ direction and is a 1D profile satisfying $|u^*|\leq 1$. Thus \eqref{condition:key}
holds for almost every $(x,\b{y})\in\Omega_d$. To prove that $u^*-\eta\in H^1(\Omega_d)$, we first show that $\gamma'(u^*)\in L^2(\Omega)$. This is true by the mean value theorem and \eqref{condition:key}:
\begin{equation}\label{eq:L2solution}
	\begin{aligned}
	\int_{\Omega_d}|\gamma'(u^*)|^2\mathrm{d}x\mathrm{d}\b{y} &= \int_{\R}|\gamma'(\eta+u^*-\eta)|^2\mathrm{d}x\\
	&= \int_{\R}|\gamma'(\eta)+\gamma''(\eta+\theta(u^*-\eta))(u^*-\eta)|^2\mathrm{d}x\\
	&\leq 2\int_{\R}|\gamma'(\eta)|^2\mathrm{d}x+2\int_{\R}|\gamma''(\eta+\theta(u^*-\eta))|^2|u^*-\eta|^2\mathrm{d}x\\
	&\leq C_1'+C_2'\|v^*\|_{L^2}.
	\end{aligned}
\end{equation}
So $\gamma'(u^*)\in L^2(\Omega_d)$. Remember that $\mathcal{L}u^*+\gamma'(u^*)=0$, so by Lemma \ref{lmm:propertyL3}, we have
\begin{align*}
	c\|u^*\|_{\dot{H}^1(\Omega_d)}\leq\|\mathcal{L}u^*\|_{L^2(\Omega_d)}=\|\gamma'(u^*)\|_{L^2(\Omega_d)}.
\end{align*}
So $u^*\in\dot{H}^1(\Omega_d)$. Moreover, $\mathcal{L}\eta\in L^2(\Omega_d)$ by Lemma \ref{lmm:propertyofL2}, so $u^*-\eta\in H^1(\Omega_d)$. In particular, $u^*$ solves equation \eqref{eq:reduced1d} in $L^2$ sense. By \eqref{condition:key}, we know that $u^*$ is a 1D profile, so $u^*-\eta\in H^1(\Omega_d)$ implies that $\lim\limits_{x\to\pm\infty}u^*(x,\b{y})=\pm 1$ holds uniformly in $\b{y}$. So (ii) holds. 

Finally, we prove (iii) and (iv). Boundedness of $u^*$ implies that $\gamma''(u^*)\nabla u^*\in L^2(\Omega_d)$. Thus $\mathcal{L}u^*=-\gamma'(u^*)\in H^1(\Omega_d)$ and $u^*\in\dot{H}^2(\Omega_d)$. Remember that $u^*$ is a 1D profile, so by embedding $H^2(\R)\subset C^1(\R)$, we know that \eqref{condition:key} implies that 
\begin{align*}
	\nabla_{\b{y}}u^*(x,\b y)=0,\ \dfrac{\p u^*(x,\b y)}{\p x}\geq 0
\end{align*}
holds for any $(x,\b y)\in\Omega_d$. So (iv) is proved and (iii) is partially proved except the strict monotonicity.

To prove the strict monotonicity, suppose that $\dfrac{\p u^*(x_0,\b{y}_0)}{\p x}=0$ for some $(x_0,\b{y}_0)\in\Omega_d$, taking derivative on both sides of \eqref{eq:generalized1d} yields 
\begin{align*}
	\mathcal{L}\dfrac{\p u^*(x_0,\b{y}_0)}{\p x} = -\gamma''(u^*)\dfrac{\p u^*(x_0,\b{y}_0)}{\p x}=0.
\end{align*} 
Thus $\mathcal{L}\dfrac{\p u^*}{\p x}=0$ at $(x_0,\b{y}_0)$. However, since $\dfrac{\p u^*}{\p x}\geq 0$, we know that $\dfrac{\p u^*}{\p x}$ attains minimum at $(x_0,\b{y}_0)$. Then by Lemma \ref{lmm:max}, we know that $\dfrac{\p u^*}{\p x}=0$, i.e. $u^*$ is a constant. This contradicts with the far field limit of $u^*$. So $\dfrac{\p u^*(x,\b{y})}{\p x}>0$. This concludes the whole theorem.
\end{proof}

\section{The De Giorgi Conjecture and uniqueness of solutions}\label{sec:proofofdg}
In Theorem \ref{thm:existence}, we prove that there exists a minimizer $u^*$ of functional $F$ on set $\mathcal{A}$ who satisfies that $u^*-\eta\in H^1(\Omega_d)$ and for any $(x,\b{y})\in\Omega_d$, we have
\begin{align*}
	\nabla_{\b{y}}u(x,\b{y})=0,\ \dfrac{\p u^*(x,\b{y})}{\p x}>0.
\end{align*}
In particular, we have $\lim\limits_{x\to\pm\infty}u^*(x,\b{y})=\pm 1$. As Definition \ref{def:layersolutionwhole}, we keep the same definition of layer solutions for \eqref{eq:generalized1d}.
\begin{definition}\label{def:layersolution}
	We call that $u:\Omega_d\to\R$ is a layer solution to \eqref{eq:generalized1d}, i.e.
	\begin{align*}
	\mathcal{L}u+\gamma'(u)=0,
	\end{align*} 
	if for any $(x,\b{y})\in\Omega_d$, 
	\begin{align}\label{condition:layer}
	\dfrac{\p u(x,\b{y})}{\p x}>0,\ \lim\limits_{x\to\pm\infty}u(x,\b{y})=\pm 1.
	\end{align}
\end{definition}
As far as we know, results parallel to the De Giorgi conjecture that address the vectorial case, i.e. system \eqref{eq:EL}, and \eqref{eq:generalized1d} are still wanting and lack of exploration. In this section, we will prove Theorem \ref{thm:degiorgi} which fills in this blank: all layer solutions to \eqref{eq:EL} or \eqref{eq:generalized1d} with $H^1$ regularity are in fact 1D profiles if we further assume $\gamma\in C^{\infty}(\R)$.

In \cite{cabre2005layer} and related literatures on the De Giorgi conjecture, the standard approach to prove this type of symmetry result is to first derive some Schauder estimates for weak solutions and then using Liouville type theorems to prove 1D symmetry. For example, authors in \cite{cabre2005layer} first derived $C^{2,\alpha}$ regularity for layer solutions by careful application of theories on elliptic PDEs. Then they noticed the following lemma (see also Lemma 2.6 in \cite{cabre2005layer}), a Liouville type lemma:
\begin{lemma} (a Liouville type theorem)
	Let $\varphi\in L_{loc}^{\infty}(\overline{\R^d_+})$ be a positive function, not necessarily bounded on all of $\R^d_+$. Suppose that $\sigma\in H_{loc}^1(\overline{\R^d_+})$ satisfies
	\begin{align*}
	\begin{cases}
	-\sigma\mathrm{div}(\varphi^2\nabla\sigma)\leq 0&\ \mathrm{in\ }\mathbb{R}^d_+, \\
	\sigma\dfrac{\p\sigma}{\p n}\leq 0&\ \mathrm{on\ }\p\mathbb{R}^d_+
	\end{cases}
	\end{align*}
	in the weak sense. Assume that, for every $R>1$, we have
	\begin{align*}
		\int_{B^+_R}(\varphi\sigma)^2\mathrm{d}x\leq CR^2
	\end{align*}
	for some constant $C$ independent of $R$. Then $\sigma$ is a constant.
\end{lemma}
Applying this lemma to function $\sigma=u_{y_i}/u_{x},\ i=1,2,...,d-1$, where $x$ direction is the monotone direction for the layer solution $u$, they proved the following lemma (see also Lemma 4.2 in \cite{cabre2005layer}): 
\begin{lemma}\label{lmm:insightful}
	Suppose that $\gamma\in C^{2,\alpha}(\R)$ is a double-well potential satisfying $\eqref{condition:potential}$. Assume that $d\leq 3$ and that $u$ is a bounded solution of 
	\begin{align*}
	\begin{cases}
	\Delta u=0&\ \mathrm{in\ }\mathbb{R}^d_+, \\
	\dfrac{\p u}{\p n}=\gamma'(u)&\ \mathrm{on\ }\p\mathbb{R}^d_+.
	\end{cases}
	\end{align*}
	Then there exists a function $\varphi\in C_{\loc}^1(\overline{\R^d_+})\bigcap C^2(\R^d_+)$ with ${\varphi>0}$ in $\overline{\R^d_+}$ and such that for every $i=1,2,...,d-1$,
	\begin{align*}
		\dfrac{\p u}{\p y_i}=c_i\varphi\quad \mathrm{in\ }\R^d_+
	\end{align*}
	for some constant $c_i$.
\end{lemma}
As a straightforward corollary, the one-dimensional symmetry of solutions to $(-\Delta)^{1/2}u+\gamma'(u)=0$ is also established.

Instead of adopting any Liouville type theorem to prove Theorem \ref{thm:degiorgi}, we will develop a new approach that is first utilized in our previous work \cite{gao2019long} to prove 1D symmetry of layer solutions to \eqref{eq:reduced1d}. Although Liouville type theorem is not employed, we found that the insightful observation provided by Lemma \ref{lmm:insightful} in \cite{cabre2005layer} is significant: as long as one can prove that there exist constants $c_i, i=1,2,...d-1$ such that 
\begin{align*}
	u_{y_i}=c_iu_{x}
\end{align*}
holds, then the profile $u$ is a 1D profile. Remember the discussion in Section \ref{sec:introduction}, $u_x$ and $u_{y_i},i=1,2,...,d-1$ are eigenfunctions of eigenvalue 0 for the linear operator 
\begin{align}
	L:H^1(\Omega_d)\subset L^2(\Omega_d)\to L^2(\Omega_d),\ L\phi=\mathcal{L}\phi+\gamma''(u)\phi.
\end{align} 
Therefore, as long as we can prove that 0 is a simple eigenvalue of $L$, i.e. the eigenspace of 0 is only 1 dimension, then we prove that $u_x$ and $u_{y_i},i=1,2,...,d-1$ are in fact linearly dependent, which indicates 1D symmetry. This is the main idea and approach we will utilize to prove Theorem \ref{thm:degiorgi}. Following this logic, we will first establish proper regularity results for layer solutions $u$ in Section \ref{sec:regularity} and then prove Theorem \ref{thm:degiorgi} in Section \ref{sec:DGproof}. 

\subsection{Regularity results}\label{sec:regularity}
In this section, we will derive some regularity results for layer solutions to equation \eqref{eq:generalized1d} and some properties of elements in the kernel of $L$. Two main results will be derived in this section under assumption $\gamma\in C^{\infty}(\R)$. First, any layer solution of equation \eqref{eq:generalized1d} is in $\dot{H}^n(\Omega_d)$ for any $n>0$ (see Lemma \ref{lmm:regularity}) and in particular, $u$ is smooth with bounded derivatives of any order. Second, eigenfunctions of $L$ with eigenvalue 0 are in $H^n(\Omega_d)$ for any $n>0$ and in particular, they decay to 0 uniformly in $\b{y}$ as $|x|\to\infty$ (see Lemma \ref{lmm:kerneldecay}). 

As a reminder, we assume $\gamma\in C^{\infty}(\R)$ in this section. Even though this is stronger than $C^2$ assumption which is generally considered, this setting indeed covers many important cases. For instance, $\gamma(u)=\dfrac{1}{\pi^2}(\cos(\pi u)+1)$ in the PN model and $\gamma(u)=(1-u^2)^2$ in the Allen-Cahn equation \cite{allen1972ground}. 

Now we begin to prove these two lemmas. All these lemmas only require that $u$ is bounded which is ensured by being a layer solution. Using the Gagliardo-Nirenberg interpolation inequality \cite{nirenberg2011elliptic} and ideas in \cite{michael1999partial} (see Proposition 3.9), we will prove that: 
\begin{lemma}\label{lmm:regularity} Suppose that $\gamma\in C^{\infty}(\R)$ is a double-well potential satisfying \eqref{condition:potential} and $\mathcal{L}$ is the linear operator defined in \eqref{def:convolutionL} satisfying assumption \eqref{A}, \eqref{B}, \eqref{C} and \eqref{D}. For any dimension $d\geq 1$, if $u\in\dot{H}^1(\Omega_d)$ satisfying $u-\eta\in H^1(\Omega_d)$ is a bounded solution to equation \eqref{eq:generalized1d}, i.e.
\begin{align*}
	\mathcal{L}u+\gamma'(u)=0,
\end{align*}
then $u-\eta\in{H}^n(\Omega_d)$ for any $n>0$. In particular, $u$ is in $\dot{H}^{n}(\Omega)$ for any $n>0$ and smooth with bounded derivatives of any order.
\end{lemma}
\begin{proof}
Taking derivative on both sides of the equation yields 
\begin{align*}
\mathcal{L}u_x+\gamma''(u)u_x=0.
\end{align*}
Remember that $u$ is bounded, so is $\gamma''(u)$ by continuity of $\gamma''$. Thus $\gamma''(u)u_x\in L^2(\Omega_d)$ which implies that $\mathcal{L}u_x\in L^2(\Omega_d)$. Thus by Lemma \ref{lmm:propertyL3}, $u_x\in H^1(\Omega_d)$. This also holds for $\nabla_{\b{y}}u(x,\b{y})$. Thus $u\in\dot{H}^2(\Omega_d)$ and $u-\eta\in H^2(\Omega_d)$.

Now we prove by induction that $u-\eta\in{H}^n(\Omega_d)$ for any positive integer $n\geq 3$. Suppose that  $u-\eta\in{H}^m(\Omega_d)$, then
by the Gagliardo-Nirenberg interpolation inequality, we know that for any $1\leq j\leq m$, we have 
\begin{align*}
	\|D^j(u-\eta)\|_{L^p(\Omega)}\leq C\|(u-\eta)\|^{a}_{H^m(\Omega)}\|u-\eta\|_{L^{\infty}(\Omega)}^{1-a}
\end{align*}
Here $p$ and $j/m\leq\alpha\leq 1$ satisfy
\begin{align*}
	\dfrac{1}{p}=\dfrac{j}{d}+a\left(\dfrac{1}{2}-\dfrac{m}{d}\right).
\end{align*}
Take $\alpha =j/m$, then we have $p=2m/j$ and 
\begin{align*}
	D^{j}(u-\eta)\in L^{2m/j}(\Omega_d).
\end{align*}
Notice that $D\eta=0$ if $|x|>1$ and $\eta\in C^{\infty}(\Omega_d)$, so $D^j\eta\in L^{2m/j}(\Omega_d)$, hence
\begin{align*}
	D^ju\in L^{2m/j}(\Omega_d)
\end{align*}
Chain rule implies that for any multi-index $\alpha$ that satisfies $|\alpha|=m$, we have
\begin{align*}
	D^{\alpha}\gamma'(u)=\sum_{\beta_1+...+\beta_k=\alpha}C_{\beta}u^{(\beta_1)}u^{(\beta_2)}...u^{(\beta_k)}\gamma^{(k+1)}(u).
\end{align*}
Here $C_{\beta}$ are constants depending on $\beta=(\beta_1,\beta_2,...,\beta_k)$. Boundedness of $u$ and smoothness of $\gamma$ ensure that $\gamma^{(k+1)}(u)$ is also bounded. Remember that $D^ju\in L^{2m/j}(\Omega_d)$, so we have $u^{(\beta_j)}\in L^{2m/|\beta_j|}(\Omega_d)$ for all $j=1,2,...,k$. Thus by H\"older's inequality, we have
\begin{align*}
\|u^{(\beta_1)}u^{(\beta_2)}...u^{(\beta_k)}\|_{L^q(\Omega_d)}\leq\prod_{j=1}^k\|u^{(\beta_j)}\|_{L^{2m/|\beta_j|}(\Omega_d)}
\end{align*}
where $q$ satisfies
\begin{align*}
	\dfrac{1}{q}=\sum_{j=1}^k\dfrac{|\beta_j|}{2m}=\dfrac{|\alpha|}{2m}=\dfrac{1}{2}.
\end{align*}
Thus $q=2$ and $u^{(\beta_1)}u^{(\beta_2)}...u^{(\beta_k)}\in L^2(\Omega_d)$. Thus $D^{\alpha}\gamma'(u)\in L^2(\Omega_d)$ for any multi-index $\alpha$ that satisfies $|\alpha|=m$, so $D^m\gamma'(u)\in L^2(\Omega_d)$. Therefore,
\begin{align*}
	\|D^m(\mathcal{L}u)\|_{L^2(\Omega_d)}=\|D^m\gamma'(u)\|_{L^2(\Omega_d)}<\infty.
\end{align*}
Thus $D^m(\mathcal{L}u)\in L^2(\Omega_d)$. Then by Lemma \ref{lmm:propertyL3} and Assumption \eqref{A}, we have $u-\eta\in H^{m+1}(\Omega_d)$ and $u\in\dot{H}^{m+1}(\Omega_d)$. Thus by induction, $u-\eta\in H^n(\Omega_d)$ for any $n>0$. In particular, this indicates that $u$ is smooth with bounded derivatives of any order.
\end{proof}

Remember that $\gamma$ is smooth, so Lemma \ref{lmm:regularity} also ensures that $\gamma''(u)$ is smooth and bounded with bounded derivatives of any order. Recall that $u_x$ is a 0 eigenfunction of operator $L$ defined in \eqref{def:operatorL}, so by ellipticity of $\mathcal{L}$ and regularity of $\gamma''(u)$, we can prove that $u_x$, or more generally, any 0 eigenfunction of $L$ should attain $H^k(\Omega_d)$ regularity for any $k>0$. As a direct corollary of Lemma \ref{lmm:regularity}, we have

\begin{lemma}\label{lmm:kerneldecay} Suppose that $\gamma\in C^{\infty}(\R)$ is a double-well potential {satisfying} \eqref{condition:potential} and $\mathcal{L}$ is the linear operator defined in \eqref{def:convolutionL} satisfying assumption \eqref{assumption}. For any dimension $d\geq 1$, if $g\in H^1(\Omega_d)$ satisfies
	\begin{align*}
		\mathcal{L}g+\gamma''(u)g=0,
	\end{align*}
	where $u$ is a bounded solution of \eqref{eq:generalized1d} as in Lemma \ref{lmm:regularity}. Then $g\in H^n(\Omega_d)$ for any $n>0$. In particular, $g$ is smooth and 
	\begin{align*}
		\lim\limits_{|x|\to\infty}g(x,\b{y})=0
	\end{align*}
	holds uniformly in $\b{y}$.
\end{lemma}
\begin{proof}
	By Lemma \ref{lmm:regularity}, we know that $\gamma''(u)$ is smooth with bounded derivatives of any order. Suppose that $g\in H^k(\Omega_d)$ for some $k\geq 1$, then
	\begin{align*}
		\|D^k(\mathcal{L}g)\|_{L^2(\Omega_d)}=\|D^k(\gamma''(u)g)\|_{L^2(\Omega_d)}\leq C_k\|D^kg\|_{L^2(\Omega_d)}<\infty.
	\end{align*}
	Here $C_k$ is a constant that only depends on $k,u$ and $\gamma$. Thus $D^k(\mathcal{L}g)\in L^2(\Omega_d)$, hence by Lemma \ref{lmm:propertyL3} and Assumption \eqref{A}, we know $g\in H^{k+1}(\Omega_d)$. So by induction, $g\in H^n(\Omega_d)$ for any $n>0$. In particular, $g$ is smooth and satisfies that $\lim\limits_{|x|\to\infty}g(x,\b{y})=0$ holds uniformly in $\b{y}$.
\end{proof}  
\begin{remark}
	Although we assume $\gamma\in C^{\infty}(\Omega_d)$, for a given dimension $d$, $\gamma\in C^{d+3}(\Omega_d)$ is sufficient to ensure that layer solutions $u$ and 0 eigenfunctions $g$ of operator $L$ are continuous and $\lim\limits_{|x|\to\infty}g(x,\b y)=0,\lim\limits_{x\to\pm\infty}u(x,\b y)=\pm 1$ hold uniformly in $\b y$. These are the properties we need to prove Theorem \ref{thm:degiorgi}.
\end{remark}

Finishing proving these two lemmas, we are ready to prove Theorem \ref{thm:degiorgi}.
\subsection{Proof of Theorem \ref{thm:degiorgi}: the De Giorgi conjecture}\label{sec:DGproof}
As discussed in the beginning of this section, we will prove Theorem \ref{thm:degiorgi} by proving that the $\mathrm{Ker}(L)$ is only 1 dimension. Here $L$ is the operator defined in \eqref{def:operatorL}. Similar to the proof in \cite{gao2019long}, Lemma \ref{lmm:max}, i.e. the maximal property plays a critical role in concluding linear dependence of $u_x$ and any other function $g$ in the $\mathrm{Ker}(L)$. 

\begin{proof}[Proof of Theorem \ref{thm:degiorgi}]
We will prove that if a non-trivial $g\in H^1(\Omega_d)$ satisfies $\mathcal{L}g+\gamma''(u)g=0$, then there exists a constant $c$ such that $g=cu_x(x,\b{y})$.

According to Lemma \ref{lmm:regularity}, $u\in \dot{H}^n(\Omega_d)$ for any $n>0$, so $u$ is continuous. By definition of layer solution (see Definition \ref{def:layersolution}), we know that  $\lim\limits_{x\to\pm\infty}u(x,\b{y})=\pm 1$. 

This limit actually holds uniformly in $\b{y}$ by continuity of $u$. To prove uniformness, by strict monotonicity of $u$, for any $a\in(-1,1)$ and $\b{y}\in\T^{d-1}$, there exists a unique $x\in\R$ such that $u(x,\b{y})=a$. Therefore, for any $a\in(-1,1)$, we consider function
\begin{align*}
	f_a(\b{y}):\T^{d-1}\to\R,\ f_a(\b{y})=\{x: u(x,\b{y})=a\}.
\end{align*}
We prove that $f_a(\b y)$ is continuous. Given $\b{y}\in\T^{d-1}$ and $\epsilon>0$ sufficiently small, since $u(f_a(\b{y}),\b{y})=a$, by strict monotonicity of $u$ w.r.t. $x$, we know that
\begin{align*}
	a_1:=u(f_a(\b{y})+\epsilon,\b{y})>a>a_2:=u(f_a(\b{y})-\epsilon,\b{y}).
\end{align*}
Then there exists $\delta>0$ such that
\begin{align*}
	u(x,\b{y})>a\ \mathrm{if}\ (x,\b{y})\in S_{\delta}(f_a(\bm{y})+\epsilon,\b{y}),\\
	u(x,\b{y})<a\ \mathrm{if}\ (x,\b{y})\in S_{\delta}(f_a(\bm{y})-\epsilon,\b{y}).
\end{align*}
Here $S_{\delta}(\b{w})$ is the square centered at $\b{w}$ with width $\delta$. Then by definition of $f_a$ and monotonicity of $u$, we know that for any ${\b{y}_1}\in\T^{d-1}$ such that $|{\b{y}_1-\b{y}}|<\dfrac{\delta}{2}$, we have $|f_a(\b{y}_1)-f_a(\b{y})|<\epsilon$. Thus $f_a(\b{y}):\T^{d-1}\to\R$ is a continuous function for any $a\in(-1,1)$. So by compactness of $\T^{d-1}$, there exist real numbers $x_a$ and $X_a$ such that
\begin{align*}
	x_a<f_a(\b{y})<X_a.
\end{align*}  
So by monotonicity, we know that for any $\b{y}\in\Omega_d$,
\begin{align*}
	u(x,\b{y})&\geq a\ \mathrm{if}\ x\geq X_a,\\
	u(x,\b{y})&\leq a\ \mathrm{if}\ x\leq x_a. 
\end{align*}
Thus limit $\lim\limits_{x\to\pm\infty}u(x,\b{y})=\pm 1$ holds uniformly in $\b{y}$. 

By Lemma \ref{lmm:kerneldecay}, we know that 
\begin{align*}
	\lim\limits_{x\to\pm\infty}u_x(x,\b{y})=0,\ \lim\limits_{x\to\pm\infty}g(x,\b{y})=0
\end{align*}
hold uniformly in $\b{y}$. Consider $\phi_{\beta}=u_x+\beta g$ and define set
\begin{align}\label{def:setD1}
D_1:=\{\beta<0:\ \phi_{\beta}(\b{\xi})<0\ \mathrm{for\  some}\ \b{\xi}\in\Omega_d\}.
\end{align}
Because $g$ is non-trivial, we assume that $g(x_0,\b{y}_0)>0$ for some $(x_0,\b{y}_0)\in\Omega_d$ without loss of generality. Then $D_1$ is non-empty because 
\begin{align*}
	\beta_1:=-2u_x(x_0,\b{y}_0)/g(x_0,\b{y}_0)\in D_1.
\end{align*} 
Here we use the positivity of $u_x$ in the definition of layer solutions. Therefore, 
\begin{align*}
	\overline{\beta}:=\sup D_1
\end{align*}
is well-defined and satisfies $\overline{\beta}\in[\beta_1,0]$.

We can also prove that for any $\beta\in D_1$, there exists $\b{\xi}_{\beta}\in\Omega_d$ such that $\phi_{\beta}(\b{\xi}_\beta)$ attains a negative minimum. By construction of $D_1$ and Lemma \ref{lmm:kerneldecay}, we know that
\begin{align*}
	\lim\limits_{|x|\to\infty}\phi_\beta(x,\b{y})=0
\end{align*}
holds uniformly in $\b{y}$. Meanwhile, $\phi_{\beta}$ attains a negative minimum. Therefore, there exists $\b{\xi}_\beta=(x_{\beta},\b{y}_{\beta})$ such that $\phi_{\beta}$ attains minimum at $\b{\xi}_{\beta}$ by continuity of $g$ and $u_x$, which is ensured by Lemma \ref{lmm:kerneldecay} and Lemma \ref{lmm:regularity}. 

Moreover, there exists $X_0\in\R$ that only depends on $\gamma$ and $u$ such that $|x_\beta|\leq X_0$ for any $\beta\in D_1$. Notice that $\phi_{\beta}$ satisfies
$\mathcal{L}\phi_{\beta}+\gamma''(u)\phi_{\beta}=0$ since both $g$ and $u_x$ are so, thus 
\begin{align*}
	\gamma''(u(x_\beta,\b{y}_\beta))\phi_{\beta}(x_\beta,\b{y}_\beta)=-\mathcal{L}\phi_{\beta}|_{(x_\beta,\b{y}_\beta)}>0
\end{align*} 
holds by minimality of $\phi_{\beta}$ and Lemma \ref{lmm:max}. Because $\phi_{\beta}(x_\beta,\b{y}_\beta)<0$, so $\gamma''(u(x_\beta,\b{y}_\beta))<0$. However, since $\lim\limits_{x\to\pm\infty}u(x,\b{y})=\pm 1$ uniformly in $\b{y}$ and $\gamma''(\pm 1)>0$, so there exists a constant $X_0>0$ such that if $|x|\geq X_0$, then $\gamma''(u(x,\b y))\geq 0$. Because $\gamma''(u(x_{\beta},\b y_{\beta}))<0$, so $|x_\beta|<X_0$.

Therefore, we know that $\{\b{\xi}_{\beta}\}_{\beta\in D_1}$ is a compact set in $\Omega_d$. So there exists a subsequence of $\beta$ in $D_1$ such that $\beta\to\overline{\beta}$, i.e. the supremum of set $D_1$, and
\begin{align*}
	\b{\xi}_{\beta}\to\overline{\b{\xi}}_{\beta}
\end{align*}
for some $\overline{\b{\xi}}_{\beta}\in\Omega_d$
. Because $\phi_{\beta}(\b{\xi}_\beta)<0$, so
\begin{align*}
	\phi_{\overline{\beta}}(\overline{\b{\xi}}_\beta)\leq 0
\end{align*}
by passing the limit $\beta\to\overline{\beta}$ and continuity of $g$ and $u_x$. However, by the definition of $\overline{\beta}$, we have $\phi_{\overline{\beta}}(\b{\xi})\geq 0$ for any $\b{\xi}\in\Omega_d$ otherwise $\overline{\beta}$ should not be the supremum of $D_1$. Thus $\phi_{\overline{\beta}}(\overline{\b{\xi}}_\beta)=0$.

This ensures $\phi_{\overline{\beta}}\equiv 0$. Because $\phi_{\overline{\beta}}\geq 0$, so $\phi_{\overline{\beta}}$ attains minimum at $\overline{\b{\xi}}_\beta$. However, since $\phi_{\overline{\beta}}$ is also in the kernel of $L$, we have
\begin{align*}
	\mathcal{L}\phi_{\overline{\beta}}|_{\overline{\b{\xi}}_{\beta}}=-\gamma''(u(\overline{\b{\xi}}_{\beta}))\phi_{\overline\beta}(\overline{\b{\xi}}_\beta)=0.
\end{align*}
Then by Lemma \ref{lmm:max} and minimality of $\overline{\b{\xi}}_\beta$, we have $\phi_{\overline{\beta}}\equiv 0$. Thus
\begin{align*}
	u_x+\overline{\beta}g=0,
\end{align*}
i.e. $g$ and $u_x$ are linearly dependent. Thus the kernel of $L$ is only 1 dimension. Notice that every partial derivative of $u$ belongs to kernel of $L$, so there exist constants $c_i\ (i=1,2,...,d-1)$ such that 
\begin{align*}
	u_{y_i}+c_iu_x=0,\ i=1,2,...,d-1
\end{align*}
for any $y_i\in\T^{d-1}$.

To close the proof, we prove that in fact $c_i=0, i=1,2,...,d-1$. Otherwise, we assume $c_i>0$ without loss of generality. For any given $(x,\b{y})\in\Omega_d$, by periodicity and the far end limit assumption \eqref{assm:farendlimit}, we have
\begin{align*}
	u(x,\b{y})&=\lim\limits_{n\to+\infty}u(x+c_in,\b{y}-n\b{e}_i)=\lim\limits_{n\to+\infty}u(x+c_in,\b{y})=1,\\
	u(x,\b{y})&=\lim\limits_{n\to+\infty}u(x-c_in,\b{y}+n\b{e}_i)=\lim\limits_{n\to+\infty}u(x-c_in,\b{y})=-1.
\end{align*}
Here $n$ are positive integers and $\b{e}_i=(0,...,0,1,0,...0), i=1,2,...,d-1$ form the canonical orthogonal basis in $\R^{d-1}$ with 1 only at the $i$ th component, and 0 for others. This yields contradiction. So $c_i=0, i=1,2,...,d-1$, i.e. $\nabla_{\b{y}}u(x,\b{y})=0$ and $u$ is a 1D profile that only depends on $x$.
\end{proof}
\subsection{Proof of Theorem \ref{thm:uniqueness}: uniqueness up to translations}
To completely understand all layer solutions to  \eqref{eq:generalized1d} and minimizers of functional $F$ on set $\mathcal{A}$, we prove the following lemma:
\begin{lemma}\label{lmm:minimizersarelayersolutions}
	(minimizers are layer solutions) For any dimension $d\geq 1$, suppose that $\gamma\in C^{\infty}(\R)$ is a double-well {type} potential satisfying \eqref{condition:potential}. Consider functional energy $F$ in \eqref{def:energyfunctional}, set $\mathcal{A}$ in \eqref{def:funcsetA}, and set $\mathcal{A}_\ell$, $\mathcal{A}_m$ in \eqref{def:minimizersandlayersolutions}.
	Then 
	\begin{align*}
		\mathcal{A}_m\subset\mathcal{A}_\ell.
	\end{align*} 
\end{lemma} 
\begin{proof}
	Let $u^*\in\mathcal{A}_m$. First of all, $u^*$ is a weak solution to equation \eqref{eq:generalized1d}. Then as in the proof of Theorem \ref{thm:existence}, we know that it solves \eqref{eq:generalized1d} $L^2$ sense, i.e. $\mathcal{L}u^*=-\gamma'(u^*)\in L^2(\Omega_d)$ (see calculation \eqref{eq:L2solution}). Then by Lemma \ref{lmm:propertyL3}, we know that $u^*-\eta\in H^1(\Omega_d)$.
	
	Because $u^*$ is a minimizer, so $|u^*|\leq 1$. Otherwise 
	\begin{align}
	\tilde{u}=\max\{1,\min\{-1,u^*\}\}
	\end{align}
	is also in $\mathcal{A}$ and satisfies $F(\tilde u)<F(u^*)$ by definition of $F$ in \eqref{def:energyfunctional}. This contradicts with the minimality of $u^*$. So $u^*$ is bounded. Then by Lemma \ref{lmm:regularity}, $u^*-\eta\in H^n(\Omega_d)$ for any $n>0$. Therefore, we have
	\begin{align}\label{eq:limit}
		\lim\limits_{x\to\pm\infty}u^*(x,\b{y})=\pm 1.
	\end{align}
	
	Now it is left to prove the strict monotonicity of $u^*$. Again, this is realized by the energy decreasing rearrangement method (Lemma \ref{lmm:energydecreasing}). For any $\tau>0$, consider the translation of $u^*$, i.e.
	\begin{align*}
		u_\tau(x,\b y)=u^*(x+\tau,\b y).
	\end{align*}
	Define
	\begin{align*}
		m(\b{w}):=\min\{u_\tau(\b w),\ u^*(\b w)\},\ M(\b w):=\max\{u_\tau(\b w),\ u^*(\b w)\}.
	\end{align*}
	Then by Lemma \ref{lmm:energydecreasing}, we know that
	\begin{align*}
		F(m)+F(M)\leq F(u_\tau)+F(u^*).
	\end{align*}
	By translation-invariance (Lemma \ref{lmm:translationinvariant}), we know $F(u_\tau)=F(u^*)$. Thus both $u^*$ and $u_\tau$ are minimizers. So by minimality of $u^*$ and $u_\tau$, we have
	\begin{align*}
		F(m)+F(M)=F(u_\tau)+F(u^*).
	\end{align*}
	Again, by Lemma \ref{lmm:energydecreasing}, this equality holds if and only if either $u_\tau(\b{w})\geq u^*(\b{w})$ or $u_\tau(\b{w})\leq u^*(\b{w})$. Then by the limit condition \eqref{eq:limit}, we know $u_\tau(\b{w})\geq u^*(\b{w})$. Thus $u^*$ is non-decreasing. 
	
	Finally, as in the proof of Theorem \ref{thm:existence}, the fact that $u$ is non-decreasing implies strict monotonicity. Suppose that $\dfrac{\p u^*(x_0,\b{y}_0)}{\p x}=0$ for some $(x_0,\b{y}_0)\in\Omega_d$, then taking derivative on both sides of \eqref{eq:generalized1d} yields 
	\begin{align*}
	\mathcal{L}\dfrac{\p u^*(x_0,\b{y}_0)}{\p x} = -\gamma''(u^*)\dfrac{\p u^*(x_0,\b{y}_0)}{\p x}=0.
	\end{align*} 
	Thus $\mathcal{L}\dfrac{\p u^*}{\p x}=0$ at $(x_0,\b{y}_0)$. However, since $\dfrac{\p u^*}{\p x}\geq 0$, we know that $\dfrac{\p u^*}{\p x}$ attains minimum at $(x_0,\b{y}_0)$. Then by Lemma \ref{lmm:max}, we know that $\dfrac{\p u^*}{\p x}=0$, i.e. $u^*$ is a constant. This contradicts with the far field limit of $u^*$. So $\dfrac{\p u^*(x,\b{y})}{\p x}>0$ holds for any $(x,\b{y})\in\Omega_d$. Thus $u^*$ is a layer solution. 
\end{proof} 
Therefore, all minimizers of $F$ on set $\mathcal{A}$ are layer solutions. Recall that Theorem \ref{thm:degiorgi} claims that all layer solutions with $H^1$ regularity have one-dimensional symmetry if the double-well potential $\gamma$ is smooth, so all these minimizers are also exactly 1D profiles. 

Moreover, these 1D profiles are unique up to translations. According to \cite{cabre2005layer}, if $\gamma\in C^{2,\alpha}(\R)$ is a double-well potential, then layer solutions to 
\begin{align}
	(-\p_{xx})^{1/2}u(x)+\gamma'(u(x))=0,\ x\in\R.
\end{align}
is unique up to translations (see Theorem 1.2 in \cite{cabre2005layer}). Remember that Lemma \ref{lmm:propertyofL2} ensures that $\mathcal{L}f=c_{\mathcal{L}}(-\p_{xx})^{1/2}f$ if $f(x,\b{y})=f(x)$ is a 1D profile, therefore, both layer solutions and minimizers are unique up to translations.
\begin{proof}[Proof of Theorem \ref{thm:uniqueness}]
	 By Theorem \ref{thm:degiorgi}, we know that for any $u\in\mathcal{A}_\ell$, u is a 1D profile to solution \eqref{eq:generalized1d}, i.e.
	\begin{align*}
		\mathcal{L}u(x,\b{y})+\gamma'(u(x,\b{y}))=0.
	\end{align*}
	By Lemma \ref{lmm:propertyofL2}, for any $(x,\b y)\in\Omega_d$, we have
	\begin{align*}
		\mathcal{L}u(x,\b{y})=c_{\mathcal{L}}(-\p_{xx})^{1/2}u(x).
	\end{align*}
	Thus viewed as a 1D profile $u(x)$, a layer solution $u(x,\b y)$ satisfies
	\begin{align}
		c_{\mathcal{L}}(-\p_{xx})^{1/2}u(x)+\gamma'(u(x))=0.
	\end{align}
	Then by Theorem 1.2 in \cite{cabre2005layer}, we know
	\begin{align}
		\mathcal{A}_\ell=\{u:\ u(x,\b{y})=u^*(x+x_0)\ \mathrm{for\ some}\ x_0\in\R\}.
	\end{align}
	Here $u^*$ is the unique solution to \eqref{eq:essential1d}. 
	
	By Theorem \ref{thm:existence} and Lemma \ref{lmm:minimizersarelayersolutions}, we know that $\mathcal{A}_m$ is non-empty and $\mathcal{A}_m\subset\mathcal{A}_\ell$. Moreover, by Lemma \ref{lmm:translationinvariant}, i.e. the translation-invariant property, we know that $u(x)\in\mathcal{A}_m$ if and only if $u(x+x_0)\in\mathcal{A}_m$. Notice that $\mathcal{A}_\ell$ itself is also unique up to translations, so we have \eqref{eq:uniqueness}, i.e.
	\begin{align*}
	\mathcal{A}_m=\mathcal{A}_\ell=\{u:\ u(x,\b{y})=u^*(x+x_0)\ \mathrm{for\ some}\ x_0\in\R\}.
	\end{align*}
	This concludes the uniqueness (up to translations) of layer solutions to equation \eqref{eq:generalized1d} and minimizers of $F$ on set $\mathcal{A}$.
\end{proof}
\subsection{Proof of Theorem \ref{thm:PN}: implication on the PN model}
As a direct application of previous results on the existence and rigidity, now we can prove Theorem \ref{thm:PN}.
\begin{proof}[Proof of Theorem \ref{thm:PN}]
	As a minimizer of $\tilde{E}$ in \eqref{energy} in the perturbed sense, we know that $\b u$ is a weak solution to \eqref{eq:EL} by Lemma \ref{Lem2.2}. A calculation (see  \cite{Caffarelli2006An}) involving the Dirichlet and Neumann map implies that if $\b u$ satisfies \eqref{eq:EL}, the elastic energy in the bulk can be expressed by $u_1^+(x,z)$ which is defined on the slip plane:
	\begin{align}
		E_{\mathrm{els}}(\b u)=\int_{\Gamma'}\mathcal{L}u_1^+(\b w)u_1^+(\b w)\mathrm{d}\b w=\dfrac{1}{2}\int_{\Omega_2}\int_{\Omega_2}|u_1^+(\b w)-u_1^+(\b w')|^2K(\b w-\b w')\mathrm{d}\b w\mathrm{d}\b w'.
	\end{align}  
	Here $\mathcal{L}$ is the linear operator defined in \eqref{eq:reduced1d} and $K$ is the corresponding convolution kernel which satisfies Assumption \eqref{A}-\eqref{D}. Therefore, $u_1^+$ is the minimizer of $F$ defined in \eqref{def:energyfunctional}. Then by Theorem \ref{thm:degiorgi} and Theorem \ref{thm:uniqueness}, we know that statement (ii) hold. Therefore, $u_1, u_2$ only depend on $x$ and $y$, satisfying the following reduced system of \eqref{eq:EL} in two dimensions:
	\begin{align}\label{eq:EL2}
	\begin{cases}
	\Delta\b{u}+\dfrac{1}{1-2\nu}\nabla(\nabla\cdot\b{u})=0,&\ \mathrm{in\ }\mathbb{R}^2\setminus\Gamma_1, \\
	\sigma^{+}_{12}+\sigma^{-}_{12}=\dfrac{\p \gamma}{\p u_1}(u_1^+),&\ \mathrm{on\ }\Gamma_1,\\
	\sigma^{+}_{22}=\sigma^{-}_{22},&\ \mathrm{on\ }\Gamma_1.
	\end{cases}
	\end{align}
	Here $\Gamma_1=\{(x,y)\in\R^2:\ y=0\}$.
	Thus smoothness of $u_1^+$ implies that $\b u$ is smooth in $\R^2\times\T\setminus\Gamma'$, so (i) holds. Finally, by Lemma 2.3 in \cite{gao2019mathematical}, we know that (iii), (iv) and (v) are true and hold point-wisely in $\R^2\times\T\setminus\Gamma'$ by smoothness.
\end{proof} 
\section{Spectral analysis of $L$}\label{sec:spectrum}
In Theorem \ref{thm:degiorgi}, we prove that if $u$ is a layer solution to \eqref{eq:generalized1d}, then the operator in \eqref{def:operatorL}, i.e.
\begin{align*}
L:H^1(\Omega)\subset L^2(\Omega)\to L^2(\Omega),\ L\phi=\mathcal{L}\phi+\gamma''(u)\phi
\end{align*} 
has one dimensional kernel which is exactly $\mathrm{span}\{u_x\}$. In this section, we proceed to prove that $L$ is positively semi-definite and $0$ is an isolated point spectrum. Denote the spectrum, the point spectrum, the residual spectrum and the continuous spectrum of a linear operator $L$ as $\sigma(L), \sigma_p(L), \sigma_r(L)$ and $\sigma_c(L)$ respectively. 

First of all, according to \cite{yosida2012functional}, since $L$ is self-adjoint (see Lemma \ref{lmm:selfadjoint}), we know $\sigma_r(L)=\emptyset$. Meanwhile, since 
\begin{align*}
	\lim\limits_{x\to\pm\infty}u(x,\b{y})=\pm 1
\end{align*}
holds uniformly in the $\b{y}$ direction and $\gamma''(\pm 1)>0$, we know that $\gamma''(u)$ is lower bounded and can only be negative on a compact set in $\Omega_d$. Therefore, there exists a finite lower bound of the spectrum of $L$, i.e.
\begin{restatable}{lemma}{spct}
	\label{lmm:spct}
		$\sigma(L)=\sigma_p(L)\cup\sigma_c(L)\subset\left[-\lambda_1,\infty\right)$. Here $\lambda_1>0$ is a constant.
\end{restatable} 
Employing the perturbation theory of self-adjoint operators \cite{kato2013perturbation}, we can characterize the essential spectrum of $L$ by viewing $L$ as a self-adjoint perturbation of $\mathcal{L}$. Remember that the continuous spectrum is a subset of the essential spectrum, we have the following lemma:
\begin{restatable}{lemma}{contspct}
	\label{lmm:contspct}
	$\sigma_c(L)\subset\left[\lambda_2,\infty\right)$. Here $\lambda_2>0$ is a constant.
\end{restatable}
Therefore, the spectrum of $L$ that belongs to $(-\lambda_1,\lambda_2)$ is a subset of $\sigma_p(L)$ with finite dimensional eigenspaces. Moreover, they are isolated points in $\sigma(L)$. To finish the spectral analysis of $L_2$, we finally prove the positive semi-definiteness of $L$.
\begin{lemma}\label{lmm:ptspct} $\sigma_p(L)\subset[0,\infty)$.
\end{lemma}

We will only prove Lemma \ref{lmm:ptspct} in this section and the proof of the fact that $L$ is self-adjoint, Lemma \ref{lmm:spct} and Lemma \ref{lmm:contspct} is attached in Appendix \ref{sec:funana}. Similar to the proof of Theorem \ref{thm:degiorgi}, the proof of Lemma \ref{lmm:ptspct} adopts an argument of contradiction and relies on the maximal principle of $\mathcal{L}$ in Lemma \ref{lmm:max}. 
\begin{proof} [Proof of Lemma \ref{lmm:ptspct}]
	
	We will prove $\sigma_{p}(L)\subset[0,\infty)$ by contradiction. Suppose that there exist $\lambda<0$ and non-zero $g\in H^1(\Omega_d)$ s.t.  
	\begin{align*}
		Lg=\lambda g.
	\end{align*}
	similar to the proof of Lemma \ref{lmm:kerneldecay}, we can prove that $g\in H^n(\Omega_d)$ for any $n>0$ and 
	\begin{align*}
		\lim\limits_{|x|\to+\infty}g(x,\b{y})=0
	\end{align*}
	holds uniformly in $\b{y}$ direction. 
	
    Consider $L|g|$. By Assumption \ref{B}, i.e. positivity of kernel $K$, and the fact that for any $\b{w},\b{w}'\in\Omega_d$,
    \begin{align*}
    	|g(\b{w}')|\geq\mathrm{sgn}(g)(\b{w})g(\b{w}'),
    \end{align*}
    we have
    \begin{align*}
	L|g|(x,\b{y})&=\mathcal{L}|g|(x,\b y)+\gamma''(u)|g|(x,\b y)\\
	&=\int_{\Omega_d}(|g(x,\b{y})|-|g(x',\b{y}')|)K(x-x,\b{y}-\b{y}')\mathrm{d}x'\mathrm{d}\b{y}'+\gamma''(u)|g|(x,\b y)\\
	&=\int_{\Omega_d}(\mathrm{sgn}(g)g(x,\b{y})-|g(x',\b{y}')|)K(x-x,\b{y}-\b{y}')\mathrm{d}x'\mathrm{d}\b{y}'+\mathrm{sgn}(g)(x,\b y)\gamma''(u)g(x,\b y)\\
	&\leq\mathrm{sgn}(g)(x,\b y)\left[\int_{\Omega_d}(g(x,\b{y})-g(x',\b{y}'))K(x-x,\b{y}-\b{y}')\mathrm{d}x'\mathrm{d}\b{y}'+\gamma''(u)g(x,\b y)\right]\\
	&\leq \mathrm{sgn}(g)(x,\b y)\cdot Lg(x,\b y) \\
	&=\lambda|g|(x,\b y)
	\leq 0.
	\end{align*} 
	Thus $|g|$ satisfies 
	\begin{align}\label{ineq:leq0}
		L|g|\leq \lambda|g|\leq 0.
	\end{align}
	
	Define $\phi_{\beta}=u_x+\beta|g|$ for real number $\beta$. Consider the following set of $\beta$:
	\begin{align*}
	D:=\{\beta<0\ |\ \phi_\beta(\b{\xi})<0\ \mathrm{for\ some\ }\b{\xi}\in\Omega_d\}.
	\end{align*}
	$D$ is nonempty because
	\begin{align*}
		\beta_1=\dfrac{-2u_x(x_0,\b{y}_0)}{|g|(x_0,\b{y}_0)}\in D
	\end{align*} for $(x_0,\b{y}_0)$ satisfying $|g|(x_0,\b{y}_0)>0$. Therefore 
	\begin{align*}
		\overline\beta:=\mathrm{sup}D
	\end{align*}
	is a well-defined finite number that lies in $[\beta_1,0]$. 
	
	Now for any $\beta\in D$, we will prove that there exists $(x_\beta,\b{y}_\beta)\in\Omega_d$ such that $\phi_{\beta}(x_\beta,\b{y}_\beta)$ attains a negative minimum at $\b{\xi}_\beta=(x_{\beta},\b{y}_\beta)$. First of all, by the definition of $D$, we know that $\phi_\beta$ is non-zero and attains a negative infimum. Remember that 
	\begin{align*}
		\lim\limits_{|x|\to\infty}g(x,\b{y})=0,\quad  \lim\limits_{|x|\to\infty}u_x(x,\b{y})=0
	\end{align*} holds uniformly in $\b{y}$, so
	\begin{align*}
		\lim\limits_{|x|\to\infty}\phi_{\beta}(x,\b y)=0
	\end{align*}
	holds uniformly in $\b{y}$. Recall that $\phi_\beta$ attains a negative infimum, so continuity of $\phi_\beta$ implies that this infimum is indeed a minimum that is attained for some $(x_\beta,\b{y}_\beta)$.
	
	Moreover, $\{\b{\xi}_\beta\}_{\beta}$ is bounded in $\Omega_d$. Notice that by \eqref{ineq:leq0} and $\beta<0$, 
	\begin{align*}
	L\phi_\beta &= Lu_x+\beta L|g|\geq \beta\lambda|g|\geq 0.
	\end{align*}
	Thus $L\phi_\beta|_{(x,\b{y})=\b{\xi}_\beta}\geq 0$. By maximal principle (Lemma \ref{lmm:max}), we know that
	\begin{align*}
		\mathcal{L}\phi_\beta|_{(x,\b{y})=\b{\xi}_\beta}\leq 0
	\end{align*}
	since $\phi_{\beta}$ attains minimum at $\b{\xi}_\beta$. Therefore, we have
	\begin{align*}
	\gamma''(u(x_\beta,\b{y}_\beta))\phi_\beta(x_\beta,\b{y}_\beta) = L\phi_\beta|_{(x,\b{y})=\b{\xi}_\beta} - \mathcal{L}\phi_\beta|_{(x,\b{y})=\b{\xi}_\beta}\geq 0.
	\end{align*}
	Because $\phi_\beta(x_\beta, \b{y}_\beta)\leq 0$ by definition of $(x_\beta,\b{y}_\beta)$, so $\gamma''(u(x_\beta,\b{y}_\beta))\geq 0$. So there exists $X>0$ that only depends on $u$ and $\gamma$ such that $|x_\beta|\leq X$. Since $g$ is periodic in $\b{y}$, we know that $\{\b{\xi}_\beta\}_{\beta}$ is bounded in $\Omega_d$.
	
	Given the boundedness of sequence $\{\b{\xi}_\beta\}_{\beta}$, we can now take a subsequence of $\beta$ (still denoted as $\beta$) such that $\beta\to\overline{\beta}$, the supremum of $D$, and $\b{\xi}_\beta\to\xi^*$ as $\beta\to\overline{\beta}$. As the supremum of $D$, $\overline{\beta}$ satisfies that  $\phi_{\overline{\beta}}(x,\b{y})\geq 0$. However, since $\phi_\beta(\b{\xi}_\beta)\leq 0$, passing the limit in $\beta$ gives that 
	\begin{align*}
	\phi_{\overline{\beta}}(\b{\xi}^*)=\lim\limits_{\beta\to\overline{\beta}}\phi_\beta(\b{\xi}_\beta)\leq 0.
	\end{align*}
	So $\phi_{\overline{\beta}}(\b{\xi}^*)=0$, which means that $\phi_{\overline{\beta}}$ attains minimum 0 at $\b{\xi}^*$. 
	
	This in fact ensures that $\phi_{\overline{\beta}}\equiv 0.$ By Lemma \ref{lmm:max}, we know 
	\begin{align*}
		\mathcal{L}\phi_{\overline{\beta}}|_{(x,\b{y})=\b{\xi}^*}\leq 0.
	\end{align*} 
	However, we also have 
	\begin{align*}
	L\phi_{\overline{\beta}} =\overline{\beta} L|g|\geq \overline{\beta}\lambda|g|\geq 0.
	\end{align*}
	Remember $\phi_{\overline{\beta}}(\b{\xi}^*)=0$, so 
	\begin{align*}
	0 &\leq L\phi_{\overline{\beta}}(\b{\xi}^*)\\
	&=\mathcal{L}\phi_{\overline{\beta}}|_{(x,\b{y})=\b{\xi}^*}+\gamma''(u(x,\b{y}))\phi_{\overline{\beta}}(x,\b{y})|_{(x,\b{y})=\b{\xi}^*}\\
	&=\mathcal{L}\phi_{\overline{\beta}}|_{(x,\b{y})=\b{\xi}^*}\\
	&\leq 0.
	\end{align*}
	So all these inequalities are in fact equalities, i.e. $\mathcal{L}\phi_{\overline{\beta}}|_{(x,\b{y})=\b{\xi}^*}=0$. By Lemma \ref{lmm:max}, we know that $\phi_{\overline{\beta}}\equiv 0$, which gives $L|g|=0$ and $0\leq L|g|\leq\lambda |g|\leq 0,$ hence $|g|=0$, contradiction! Thus $L$ has no negative point spectrum. 
\end{proof}

\section*{Acknowledgement}
Jian-Guo Liu was supported in part by the National Science Foundation (NSF) under award DMS-1812573 and DMS-2106988.

\bibliographystyle{plain}
\bibliography{rigidity}

\appendix

\section{Review of functional analysis}\label{sec:funana}
For the sake of completeness, we prove the fact that operator $L$ defined in \eqref{def:operatorL} and $\mathcal{L}$ in \eqref{def:convolutionL} are both self-adjoint in Lemma \ref{lmm:spct} and Lemma \ref{lmm:contspct} which address the spectrum of $L$. Let us recall that linear operator $L$ is given by 
\begin{align*}
	L:H^1(\Omega)\subset L^2(\Omega)\to L^2(\Omega),\ L\phi=\mathcal{L}\phi+\gamma''(u)\phi.
\end{align*}
Here $u$ is a layer solution to equation \eqref{eq:reduced1d}. For theorems and definitions in functional analysis, one can refer to \cite{yosida2012functional}. For perturbation theory of self-adjoint operators, one can refer to \cite{kato2013perturbation}.

In fact, the fact that $L$ is self-adjoint is a corollary of Kato-Rellich's theorem (see \cite{kato2013perturbation}). We still repeat the proof for readers' convenience. The proof here needs an equivalent criterion for self-adjoint operators:
\begin{lemma}\label{lmm:selfadjointcri}
	Suppose that $H$ is a complex Hilbert space with inner product $\langle\cdot,\cdot\rangle_H$ and $A:H\to H$ is a symmetry operator on $H$. Then $A$ is self-adjoint if and only if 
	\begin{align}
		\mathrm{Ran}(A\pm i)=H.
	\end{align}
\end{lemma}
\begin{proof}
	$\implies$: Suppose that $A$ is self-adjoint, we prove that $\mathrm{Ran}(A\pm i)=H$. To prove this, notice that for any $w\in H$, we have
	\begin{align*}
	\|(A\pm i)w\|^2=\|Aw\|^2+\|w\|^2\geq \|w\|^2,
	\end{align*}
	so by the closed image theorem, we know that $\mathrm{Ran}(A\pm i)$ is closed and $\mathrm{Ker}(A\pm i)=\{0\}$. Also notice that
	\begin{align*}
	\mathrm{Ran}(A\pm i)^{\perp}=\mathrm{Ker}(A\mp i)=\{0\},
	\end{align*}
	so $\mathrm{Ran}(A\pm i)$ are dense in $H$ by the Hahn-Banach theorem. Remember that they are also closed, so $\mathrm{Ran}(A\pm i)=H$.
	
	$\impliedby$: Suppose that $\mathrm{Ran}(A\pm i)=H$. We prove that $A$ is self-adjoint. Because $A$ is symmetry, so we only need to prove that $\mathrm{dom}(A^*)\subset\mathrm{dom}(A)$ since $\mathrm{dom}(A)\subset\mathrm{dom}(A^*)$ holds for any symmetry operator. Notice that
	\begin{align*}
	\mathrm{Ker}(A^*\mp i)=\mathrm{Ran}(A\pm i)^{\perp}=\{0\},
	\end{align*}
	so $\mathrm{Ker}(A^*\mp i)=\{0\}$. Remember that $\mathrm{Ran}(A\pm i)=H$, so for any $x\in\mathrm{dom}(A^*)$, there exists $z\in\mathrm{dom}(A)$ such that
	\begin{align*}
	(A^*\pm i)x=(A\pm i)z.
	\end{align*}
	So $(A^*\pm i)(x-z)=0$. Here we used that for $z\in\mathrm{dom}(A)$, $Az=A^*z$. Thus $x-z\in\mathrm{Ker}(A^*\pm i)=\{0\}$. So $x=z$ and $A=A^*$. So $A$ is self-adjoint. 
\end{proof}
\begin{lemma}\label{lmm:selfadjoint}
 Operator $L$ defined in \eqref{def:operatorL} and operator $\mathcal{L}$ defined in \eqref{def:convolutionL} are self-adjoint. 
\end{lemma}
\begin{proof}
First of all, both $L$ and $\mathcal{L}$ are symmetric by Assumption \ref{A}. For any $w,v\in H^1(\Omega_d)$, we have
\begin{align*}
\langle w, \mathcal{L}v \rangle_{L^2(\Omega_d)} &= \langle \hat{w}, \sigma_L(\b{\nu})\hat{v} \rangle_{L^2(\Omega_d')} = \langle \sigma_L(\b{\nu})\hat{w}, \hat{v} \rangle_{L^2(\Omega_d')} = \langle \mathcal{L}w, v \rangle_{L^2(\Omega_d)}\\
\langle w, Lv \rangle_{L^2(\Omega_d)} &= \langle w, \mathcal{L}v+\gamma''(u)v\rangle_{L^2(\Omega_d)}= \langle \mathcal{L}w+\gamma''(u)w, v\rangle_{L^2(\Omega_d)}= \langle Lw, v \rangle_{L^2(\Omega_d)}.
\end{align*} 
So they are all symmetric.

Then we prove that $\mathcal{L}$ in \eqref{def:convolutionL} is self-adjoint. By Lemma \ref{lmm:selfadjointcri}, we only need to prove that $\mathrm{Ran}(\mathcal{L}\pm i)=L^2(\Omega_d)$. We prove for only $\mathcal{L}+i$, the other side direction is just the same. 

To prove this, we only need to prove that for any $v\in L^2(\Omega_d)$, there exists $u\in H^1(\Omega_d)$ such that
\begin{align*}
	(\mathcal{L}+i)u=v.
\end{align*}
One can rewrite this equality on the Fourier side as
\begin{align*}
	(\sigma_L(\b{\nu})+i)\hat{u}(\b{\nu})=\hat{v}(\b{\nu}).
\end{align*}
Thus
\begin{align}\label{def:u}
	\hat{u}(\b{\nu})=\dfrac{\hat{v}(\b{\nu})}{\sigma_L(\b{\nu})+i}.
\end{align}
So we only need to prove that for any $v\in L^2(\Omega_d)$, $u$ in \eqref{def:u} is in $H^1(\Omega_d)$. This is true by Assumption \ref{A} which assumes that $\sigma(\b{\nu})$ is real and with same order as $|\b{\nu}|$:
\begin{align*}
	\|u\|_{H^1(\Omega_d)}^2 &= \langle \hat{u}(\b{\nu}),\hat{u}(\b{\nu})\rangle_{L^2(\Omega_d')}+\langle |\b{\nu}|\hat{u}(\b{\nu}),|\b{\nu}|\hat{u}(\b{\nu})\rangle_{L^2(\Omega_d')}\\
	&=\left\langle \hat{v}(\b{\nu}),\dfrac{|\b{\nu}|^2+1}{\sigma_L^2(\b{\nu})+1}\hat{v}(\b{\nu})\right\rangle_{L^2(\Omega_d')}\\
	&\leq \dfrac{1}{c^2}\|v\|_{L^2(\Omega)}^2.
\end{align*}
Here $c>0$ is the constant in Assumption \ref{A}. So by Lemma \ref{lmm:selfadjointcri}, we know that $\mathcal{L}$ is self-adjoint.

Finally, we prove that $L$ in \eqref{def:operatorL} is self-adjoint. Denote $A=\mathcal{L}$ and $B=\gamma''(u)$ who is understood as a multiplier, then $L=A+B$. First, because $A$ is self-adjoint, so by Lemma \ref{lmm:selfadjointcri}, $\mathrm{Ran}(A\pm\mu i)=L^2(\Omega_d)$ for any real number $\mu>0$. 

Moreover, there also exists $\mu>0$ such that $\mathrm{Ran}(A+B\pm\mu i)=L^2(\Omega_d)$. To prove this,
notice that for any $y\in H^1(\Omega_d)$, we have
\begin{align}
	\|(A\pm\mu i)y\|^2=\|Ay\|^2+\mu^2\|y\|^2.
\end{align}
Then take $y=(A\pm \mu i)^{-1}x$ for any $x\in L^2(\Omega_d)$, we have
\begin{align*}
	\|A(A\pm\mu i)^{-1}x\|^2&= \|(A\pm\mu i)(A\pm\mu i)^{-1}x\|^2-\mu^2\|(A\pm\mu i)^{-1}x\|^2\leq \|x\|^2,\\
	\mu^2\|(A\pm\mu i)^{-1}x\|^2&= \|(A\pm\mu i)(A\pm\mu i)^{-1}x\|^2-\|A(A\pm\mu i)^{-1}x\|^2\leq \|x\|^2.	
\end{align*}
so $\|A(A\pm i)^{-1}\|\leq 1$ and $\|(A\pm\mu i)^{-1}\|\leq\dfrac{1}{\mu}$. Notice that for sufficiently large $\mu$, we have $\|B(A\pm\mu i)^{-1}\|< 1$ since $B$ is a bounded linear operator and 
\begin{align*}
	\|B(A\pm\mu i)^{-1}x\|\leq {b}\|(A\pm\mu i)^{-1}x\|\leq\dfrac{b}{\mu}\|x\|.
\end{align*}
So by choosing sufficiently large $\mu$, we have $\|B(A\pm\mu i)^{-1}\|< 1$. This gives that $B(A\pm\mu i)^{-1}+I$ are invertible. Notice that 
\begin{align*}
A+B\pm\mu i=[B(A\pm\mu i)^{-1}+I](A\pm\mu i)
\end{align*}
so $\mathrm{Ran}(A+B\pm\mu i)=L^2(\Omega_d)$ since $\mathrm{Ran}(A\pm\mu i)=L^2(\Omega_d)$ and $B(A\pm\mu i)^{-1}+I$ is invertible. Then by Lemma \ref{lmm:selfadjointcri}, $L=A+B$ is self-adjoint. 
\end{proof}
Now we prove Lemma \ref{lmm:spct}.
\spct*
\begin{proof}
	Notice that $L$ is self-adjoint, so $\sigma(L)=\sigma_p(L)\cup\sigma_c(L)$. Because $\lim\limits_{x\to\pm\infty}u(x,\b{y})=\pm 1$ holds uniformly in $\b{y}$ and $\gamma''(\pm 1)>0$, so there exists $\lambda_1>0$ such that $\gamma''(u(x,y))>-\lambda_1$ holds by continuity of $u$ and $\gamma$. Now we prove that for any $\lambda\in\mathbb{C}\setminus[-\lambda_1,+\infty)$, $\lambda I-L$ has a bounded inverse. This directly shows $\sigma(L)\subset[-\lambda_1,+\infty)$.
	
	First, $\mathrm{Ran}(\lambda I-L)$ is closed. Let $\lambda=a+bi$. For any $w\in H^1(\Omega_d)$, if $b\neq 0$, we have
	\begin{align*}
		\|(\lambda I-L)w\|^2=(a^2+b^2)\|w\|^2+\|Lw\|^2-2a\langle w,Lw\rangle\geq b^2\|w\|^2.
	\end{align*}
	If $b=0$ but $a<-\lambda_1$, we have
	\begin{align*}
	\|(\lambda I-L)w\|^2=(a+\lambda_1)^2\|w\|^2+\|(L+\lambda_1)w\|^2-2(a+\lambda_1)\langle w,(L+\lambda_1)w\rangle\geq (a+\lambda_1)^2\|w\|^2.
	\end{align*}
	This is because
	\begin{align*}
		\langle w,(L+\lambda_1)w\rangle\geq 0
	\end{align*} 
	since $\gamma''(u)>-\lambda_1$ and $\mathcal{L}$ is positively semi-definite. Thus for any $\lambda\in\mathbb{C}\setminus[-\lambda_1,+\infty)$, there exists $c>0$ such that $\|(\lambda I-L)w\|\geq c\|w\|$ for any $w\in H^1(\Omega_d)$.
	
	Therefore, by the closed image theorem, $\mathrm{Ran}(\lambda I-L)$ is closed and $\mathrm{Ker}(\lambda I-L)=\{0\}$. So $\lambda I-L$ is injective. Moreover, we have
	\begin{align*}
		\mathrm{Ran}(\lambda I-L)^{\perp}=\mathrm{Ker}(\lambda^* I-L)=\{0\}
	\end{align*} 
	since $\lambda^*$ also belongs to $\mathbb{C}\setminus[-\lambda_1,+\infty)$. So by the Hahn-Banach theorem, $\overline{\mathrm{Ran}(\lambda I-L)}=L^2(\Omega_d)$. Remember that $\mathrm{Ran}(\lambda I-L)$ is closed, so $\mathrm{Ran}(\lambda I-L)=L^2(\Omega_d)$. Thus $\lambda I-L$ is a bijection. Because $L$ is self-adjoint, so $\lambda I-L$ is closed, so is $(\lambda I-L)^{-1}$. Thus by the closed graph theorem, $(\lambda I-L)^{-1}$ is bounded. Therefore, $\lambda$ is not in the spectrum of $L$.
\end{proof}

Finally, we prove Lemma \ref{lmm:contspct}. To prove this lemma, we need to employ Weyl's theorem on perturbation of self-adjoint operators. 
\begin{lemma}
	(Weyl's theorem \cite{kato2013perturbation}) Suppose that $H$ is a Hilbert space, $A$ is a self-adjoint operator on $H$ and $B$ is a symmetric operator on $H$. Then if $B$ is relatively compact with respect to $A$, then $\sigma_{ess}(A+B)=\sigma_{ess}(A)$.  
\end{lemma}
\contspct*
\begin{proof}
	Define function $f:\Omega_d\to\R$ as 
	\begin{align*}
	f(x,\b{y})=
	\begin{cases}
	\gamma''(1),\ &\mathrm{if}\ x>0,\\
	\gamma''(-1),\ &\mathrm{if}\ x\leq 0.
	\end{cases}
	\end{align*}
	Notice that $\gamma''(\pm 1)>0$ and $u(x,\b{y})\to\pm$ as $x\to\infty$ holds uniformly in $\b{y}$, so there exists $c>0$ such that $f>c$ and 
	\begin{align*}
		\lim\limits_{|x|\to\infty}\gamma''(u(x,\b{y}))-f(x,\b{y})= 0
	\end{align*}
	holds uniformly in $\b{y}$ direction. Now we rewrite operator $L$ as 
	\begin{align*}
		L=A+B,\ A=\mathcal{L}+f(x,y),\ B=\gamma''(u)-f(x,y).
	\end{align*}
	$B$ is understood as a multiplier. We will prove that $B$ is relatively compact with respect to $A$. Suppose that $\{u_j\} \subset L^2(\Omega_d)$ is bounded. We only need to prove that $\{B(A+i)^{-1}u_j\}_j$ is compact in $L^2(\Omega_d)$. 
	
	Denote $w_j=(A+i)^{-1}u_j$. We only need to prove that for any $\epsilon>0$, there exists a subsequence of $\{w_j\}$ such that $\|Bw_{j,n}-Bw_{j,m}\|\leq \epsilon$. First of all, because $w_j=(A+i)^{-1}u_j$, thus by Lemma \ref{lmm:propertyofL2}, we have
	\begin{align*}
		|\langle w_j,u_j\rangle| &= |\langle w_j, (A+i)w_j\rangle|\\
		&\geq |\langle w,\mathcal{L}w\rangle+\langle w_j, f(x,\b{y})w_j\rangle +i\|w_j\|^2|\\
		&\geq c\|w_j\|^2_{H^{1/2}(\Omega_d)}.
	\end{align*}
	Here $c>0$ is a constant that only depends on $\mathcal{L}$. Then by the Cauchy-Schwartz inequality, we know that 
	\begin{align*}
		\dfrac{1}{2c}\|u_j\|^2+\dfrac{c}{2}\|w_j\|^2&\geq |\langle w_j,u_j\rangle|\geq c\|w_j\|^2_{H^{1/2}(\Omega_d)}.
	\end{align*}
	Thus there exists $c'>0$ such that $\|w_j\|^2_{H^{1/2}(\Omega_d)}\leq c'\|u_j\|^2$, thus $\{w_j\}$ is bounded in $H^{1/2}(\Omega_d)$. Moreover, for any $\epsilon_1$ sufficiently small, there exists $R>0$ such that  
	\begin{align*}
		|\gamma''(u(x,\b{y}))-f(x,\b{y})|\leq \epsilon_1
	\end{align*}
	for $(x,\b{y})\in [-R,R]^c\times\T^{d-1}$. Therefore, 
	\begin{align*}
		\|Bw_j-Bw_k\|^2_{L^2([-R,R]^c\times\T^{d-1})}&=\int_{[-R, R]^c\times \T^{d-1}}|Bw_j-Bw_k|^2\mathrm{d}x\mathrm{d}\b{y}\\
		&\leq \epsilon_1^2\|w_j-w_k\|^2_{L^2(\Omega_d)}<\dfrac{\epsilon}{2}
	\end{align*} 
	by selecting $\epsilon_1$ sufficiently small. Then by compact embedding of $H^{1/2}([-R,R]\times\T^{d-1})\subset L^2([-R,R]\times\T^{d-1})$ and boundedness of $B$, we know that there exists a subsequent of ${w_j}$ (still denoted as $w_j$) such that $\|Bw_j-Bw_k\|^2_{L^2([-R,R]\times\T^{d-1})}\leq\dfrac{\epsilon}{2}$. Then for this subsequence, we have 
	\begin{align*}
		\|Bw_j-Bw_k\|^2_{L^2(\Omega_d)}=\|Bw_j-Bw_k\|^2_{L^2([-R,R]\times\T^{d-1})}+\|Bw_j-Bw_k\|^2_{L^2([-R,R]^c\times\T^{d-1})}\leq{\epsilon}.
	\end{align*}
	This proves that $\sigma_{ess}(L)=\sigma_{ess}(A)$. However, since $f(x,\b{y})>c>0$ is uniformly bounded from below, so $A=\mathcal{L}+f(x,\b{y})$ is positively definite and $\sigma(A)\subset[c,+\infty)$. Thus $\sigma_c(L)\subset\sigma_{ess}(L)\subset[c,+\infty)$. Taking $\lambda_2=c$ closes the proof. 
\end{proof}
\section{Proof of lemmas}\label{sec:proof}
\begin{proof}[Proof of Lemma \ref{Lem2.2}]
	From Definition \ref{minimizer} of minimizers, we calculate the variation of energy in terms of a perturbation with compact support in an arbitrary ball $B(R)\subset\R^3$ which is centered at $\b{0}$ with radius $R$. For any $ \b v\in C^\infty(B(R)\backslash \Gamma)$ such that $ \b v$ has compact support in $B(R)$ and satisfies \eqref{bcphi}, we consider the perturbation $\delta \b v$ where $\delta$ is a small real number. We denote  $\varepsilon:=\varepsilon(\b u)$, $\sigma:=\sigma(\b u)$ and $\varepsilon_1:=\varepsilon(\b v)$, $\sigma_1:=\sigma(\b v)$.
	Then we have that
	\begin{equation}\label{tem2.8}
	\begin{aligned}
	~&\lim_{\delta\to 0}\frac{1}{\delta} (E(\b u+\delta \b v)-E(\b u))\\
	=& \int_{B(R)\backslash \Gamma}\frac{1}{2}(\sigma_1:\varepsilon+ \sigma:\varepsilon_1)\ud x \ud y \ud z  +\int_{B(R)\cap\Gamma} \pt_{u_1}\gamma(u_1^+, u_3^+)v_1^++ \pt_{u_3} \gamma(u_1^+, u_3^+) v_3^+ \ud x\ud z\\
	=&\int_{B(R)\backslash \Gamma}\sigma:\varepsilon_1\ud x \ud y  \ud z +\int_{B(R)\cap\Gamma} \pt_{u_1}\gamma(u_1^+, u_3^+)v_1^++ \pt_{u_3} \gamma(u_1^+, u_3^+) v_3^+ \ud x\ud z\\
	=&\int_{B(R)\backslash \Gamma}\sigma:\nabla\b v\ud x \ud y \ud z  +\int_{B(R)\cap\Gamma} \pt_{u_1}\gamma(u_1^+, u_3^+)v_1^++ \pt_{u_3} \gamma(u_1^+, u_3^+) v_3^+ \ud x\ud z\\
	=&-\int_{B(R)\backslash \Gamma}\partial_j\sigma_{ij} v_i\ud x \ud y \ud z +\int_{B(R)\cap \{y=0^+\}}\sigma_{ij}^+ n_j^+ v_i^+ \ud x \ud z \\
	&\quad + \int_{B(R)\cap\{y=0^-\}}\sigma_{ij}^- n_j^- v_i^- \ud x \ud z
	+\int_{B(R)\cap\Gamma} \pt_{u_1}\gamma(u_1^+, u_3^+)v_1^++ \pt_{u_3} \gamma(u_1^+, u_3^+) v_3^+ \ud x\ud z\geq 0\\
	\end{aligned}
	\end{equation}
	where we used the property that $\sigma$  and $\nabla \cdot \sigma$ are locally integrable in $\{y>0\}\cup\{y<0\}$ when carrying out the integration by parts, and
	the outer normal vector of the boundary $\Gamma$ is
	$\mathbf n^+$ (resp. the $\mathbf n^-$) for the upper  {half-plane} (resp. lower half-plane). Similarly, taking perturbation as $-\b v$ and notice that that $\mathbf n^+=(0,-1,0)$ and $\mathbf n^-=(0,1,0)$, we have
	\begin{equation}
	\begin{aligned}
	&\int_{\{y=0^+\}}\sigma_{ij}^+ n_j^+ v_i^+ \ud x \ud z + \int_{\{y=0^-\}}\sigma_{ij}^- n_j^- v_i^- \ud x \ud z \\
	=& \int_{\{y=0^+\}}-\sigma_{22}^+  v_2^+ \ud x \ud z+ \int_{\{y=0^-\}}\sigma_{22}^-  v_2^- \ud x \ud z+ \int_{\{y=0^+\}}-\sigma_{12}^+  v_1^+ \ud x \ud z+ \int_{\{y=0^-\}}\sigma_{12}^-  v_1^- \ud x \ud z\\
	&+  \int_{\{y=0^+\}}-\sigma_{32}^+  v_3^+ \ud x \ud z+ \int_{\{y=0^-\}}\sigma_{32}^-  v_3^- \ud x \ud z
	\end{aligned}
	\end{equation}
	Since  $v_1^+(x,z)=-v_1^-(x,z)$, $v_3^+(x,z)=-v_3^-(x,z)$ and $v_2^+(x,z)=v_2^-(x,z)$. Hence due to the arbitrariness of $R$, we conclude that the minimizer $\b u$ must satisfy
	\begin{equation}\label{tem2.10}
	\begin{aligned}
	&\int_{\Gamma}\left[\sigma_{12}^+ + \sigma_{12}^- -\pt_{u_1}\gamma (u_1^+, u_3^+)\right] v_1^+ \ud x \ud z=0,\\
	&\int_{\Gamma}\left[\sigma_{32}^+ + \sigma_{32}^- -\pt_{u_3}\gamma (u_1^+, u_3^+)\right] v_3^+ \ud x \ud z=0,\\
	&\int_\Gamma \left(\sigma_{22}^+-\sigma_{22}^-\right) v_2^+ \ud x \ud z =0,\\
	&\int_{\mathbb{R}^3\backslash \Gamma} (\nabla\cdot \sigma) \cdot\b{v}~ \ud x\ud y \ud z =0
	\end{aligned}
	\end{equation}
	for any $\b v\in  C^\infty(B(R)\backslash\Gamma)$ and $ \b v$ has compact support in $B(R)$,
	which leads to the Euler--Lagrange equation \eqref{eq:EL}. Here we have written the equation $\nabla\cdot \sigma=0$ in $\mathbb{R}^3 \backslash \Gamma$ as the first equation of \eqref{eq:EL} in terms of the displacement $\b u$.
\end{proof}

\begin{proof}[Proof of Lemma \ref{lmm:rearrangeinequality}]
	If $a_1=a_2$ or $b_1=b_2$ holds, then the equality holds. So we will focus on cases where $a_1\neq a_2$ and $b_1\neq b_2$. By enumeration of all possible orders, we have:
	\begin{enumerate}[(i)]
		\item If $a_1>a_2$ and $b_1>b_2$, then $a=a_2, A=a_1, b=b_2$ and $B=b_1$. So 
		\begin{align*}
		ab+AB-a_1b_1-a_2b_2 = ab+AB-AB-ab=0.
		\end{align*}
		The equality in \eqref{ineq:rearrange} holds.
		\item If $a_1>a_2$ and $b_1<b_2$, then $a=a_2, A=a_1, b=b_1$ and $B=b_2$.
		So
		\begin{align*}
		ab+AB-a_1b_1-a_2b_2 = ab+AB-Ab-aB= (a-A)(b-B)> 0.
		\end{align*} 
		The '$>$' in \eqref{ineq:rearrange} holds.
		\item If $a_1<a_2$ and $b_1>b_2$, then $a=a_1, A=a_2, b=b_2$ and $B=b_1$.
		So
		\begin{align*}
		ab+AB-a_1b_1-a_2b_2 = ab+AB-aB-Ab= (a-A)(b-B)> 0.
		\end{align*} 
		The '$>$' in \eqref{ineq:rearrange} holds.
		\item If $a_1<a_2$ and $b_1<b_2$, then $a=a_1, A=a_2, b=b_1$ and $B=b_2$. So 
		\begin{align*}
		ab+AB-a_1b_1-a_2b_2 = ab+AB-ab-AB=0.
		\end{align*}
		The equality in \eqref{ineq:rearrange} holds.
	\end{enumerate} 
	Therefore, the inequality holds. The equality is attained if and only if $a_1=a_2$ or $b_1=b_2$ or the order is preserved, i.e. $a_1<a_2,b_1<b_2$ or $a_1>a_2,b_1>b_2$. These conditions are equivalent to the following clear inequality: $(a_1-a_2)(b_1-b_2)\geq 0$. This concludes the proof.
\end{proof}
\begin{proof}[Proof of Lemma \ref{lmm:translationinvariant}]
	In fact, by change of variables, we have
	\begin{align*}
	F(u(x+c_1,\b{y}+\b{c}_2)) &=\dfrac{1}{2}\int_{\Omega_d}\int_{\Omega_d}{|(u(x+c_1,\b{y}+\b{c}_2)-u(x'+c_1,\b{y}'+\b{c}_2)|^2}K(x-x',\b{y}-\b{y}')\\
	&-{|(\eta(x,\b{y})-\eta(x',\b{y}')|^2}K(x-x',\b{y}-\b{y}')\mathrm{d}x\mathrm{d}x'\mathrm{d}\b{y}\mathrm{d}\b{y}'\\
	&+\int_{\Omega_d}\gamma(u(x+c_1,\b{y}+\b{c}_2))\mathrm{d}x\mathrm{d}\b{y}\\
	&=\dfrac{1}{2}\int_{\Omega_d}\int_{\Omega_d}{|(u(x,\b{y})-u(x',\b{y}')|^2}K(x-x',\b{y}-\b{y}')\\
	&-{|(\eta(x-c_1,\b{y}-\b{c}_2)-\eta(x'-c_1,\b{y}'-\b{c}_2)|^2}K(x-x',\b{y}-\b{y}')\mathrm{d}x\mathrm{d}x'\mathrm{d}\b{y}\mathrm{d}\b{y}'\\
	&+\int_{\Omega_d}\gamma(u(x,\b{y}))\mathrm{d}x\mathrm{d}\b{y}.\\
	\end{align*}
 Thus by Lemma \ref{lmm:propertyofL2}, we have
	\begin{align*}
	F(u(x,y))-F(u(x+c_1,\b{y}+\b{c}_2)) &=\dfrac{1}{2}\int_{\Omega}\int_{\Omega}{|\eta(x+c_1,\b{y}+\b{c}_2)-\eta(x'+c_1,\b{y}'+\b{c}_2)|^2}K(x-x',\b{y}-\b{y}')\\
	&-{|(\eta(x,\b{y})-\eta(x',\b{y}')|^2}K(x-x',\b{y}-\b{y}')\mathrm{d}x\mathrm{d}x'\mathrm{d}\b{y}\mathrm{d}\b{y}'\\
	&=\dfrac{A}{2}\int_{\R}\dfrac{(\eta(x+c_1)-\eta(x'+c_1))^2}{(x-x')^2}- \dfrac{(\eta(x)-\eta(x'))^2}{(x-x')^2}\mathrm{d}x\mathrm{d}x'.
	\end{align*}
	Here $A$ is the constant in Lemma \ref{lmm:propertyofL2}.	So we only need to prove that 
	\begin{align*}
	\int_{\R}\int_{\R}\dfrac{(\eta(x+c)-\eta(x'+c))^2}{(x-x')^2}- \dfrac{(\eta(x)-\eta(x'))^2}{(x-x')^2}\mathrm{d}x'\mathrm{d}x=0
	\end{align*}
	for any $c\in\R$. Without loss of generality, we assume that $c>0$. Then for $x\geq 1$, we know that $\eta(x)=\eta(x+c)=1$ and for $x\leq -1-c$, we have $\eta(x)=\eta(x+c)=-1$. Denote $J=[-1-c,1]$, and we separate the integral into 3 different parts, i.e. integral on $J\times J$ (denoted as $I_1$), $J\times J^c$ (denoted as $I_2$) and $J^c\times J^c$ (denoted as $I_3$). Since $\eta(x)=\eta(x+c)$ on $J^c$, we know that
	\begin{align*}
	I_3 = \int_{J^c}\int_{J^c}\dfrac{(\eta(x+c)-\eta(x'+c))^2}{(x-x')^2}- \dfrac{(\eta(x)-\eta(x'))^2}{(x-x')^2}\mathrm{d}x'\mathrm{d}x=0.
	\end{align*}
	On $J\times J$, we have
	\begin{align*}
	I_1 &= \int_{-1-c}^1\int_{-1-c}^1\dfrac{(\eta(x+c)-\eta(x'+c))^2}{(x-x')^2}- \dfrac{(\eta(x)-\eta(x'))^2}{(x-x')^2}\mathrm{d}x'\mathrm{d}x\\
	&= \int_{-1}^{1+c}\int_{-1}^{1+c}\dfrac{(\eta(x)-\eta(x'))^2}{(x-x')^2}\mathrm{d}x'\mathrm{d}x-\int_{-1-c}^{1}\int_{-1-c}^{1}\dfrac{(\eta(x)-\eta(x'))^2}{(x-x')^2}\mathrm{d}x'\mathrm{d}x\\
	&=\int_1^{1+c}\int_1^{1+c}\dfrac{(\eta(x)-\eta(x'))^2}{(x-x')^2}\mathrm{d}x'\mathrm{d}x+2\int_{-1}^{1}\int_1^{1+c}\dfrac{(\eta(x)-\eta(x'))^2}{(x-x')^2}\mathrm{d}x'\mathrm{d}x\\
	&-\int_{-1-c}^{-1}\int_{-1-c}^{-1}\dfrac{(\eta(x)-\eta(x'))^2}{(x-x')^2}\mathrm{d}x'\mathrm{d}x-2\int_{-1}^{1}\int_{-1-c}^{-1}\dfrac{(\eta(x)-\eta(x'))^2}{(x-x')^2}\mathrm{d}x'\mathrm{d}x.
	\end{align*}
	Because $\eta(x)=\eta(x')$ if $x, x'\geq 1$ or $x, x'\leq -1$, so integral vanishes on $[1,1+c]\times[1,1+c]$ or $[-1-c,-1]\times[-1-c,-1]$. Thus 
	\begin{align}\label{eq:I1}
	I_1 &= 2\int_{-1}^{1}\int_1^{1+c}\dfrac{(\eta(x)-\eta(x'))^2}{(x-x')^2}\mathrm{d}x'\mathrm{d}x-2\int_{-1}^{1}\int_{-1-c}^{-1}\dfrac{(\eta(x)-\eta(x'))^2}{(x-x')^2}\mathrm{d}x'\mathrm{d}x\nonumber\\
	&= 2\int_{-1}^{1}\int_1^{1+c}\dfrac{(\eta(x)-1)^2}{(x-x')^2}\mathrm{d}x'\mathrm{d}x-2\int_{-1}^{1}\int_{-1-c}^{-1}\dfrac{(\eta(x)+1)^2}{(x-x')^2}\mathrm{d}x'\mathrm{d}x\nonumber\\
	&= 2\int_{-1}^1\dfrac{(\eta(x)-1)^2}{x-1-c}-\dfrac{(\eta(x)-1)^2}{x-1}\mathrm{d}x-2\int_{-1}^1\dfrac{(\eta(x)+1)^2}{x+1}-\dfrac{(\eta(x)+1)^2}{x+1+c}\mathrm{d}x.
	\end{align}
	On $J\times J^c$, we have 
	\begin{align*}
	I_2 &= \int_{-1-c}^1\int_{1}^{+\infty}\dfrac{(\eta(x+c)-\eta(x'+c))^2}{(x-x')^2}- \dfrac{(\eta(x)-\eta(x'))^2}{(x-x')^2}\mathrm{d}x'\mathrm{d}x\\
	&+ \int_{-1-c}^1\int_{-\infty}^{-1-c}\dfrac{(\eta(x+c)-\eta(x'+c))^2}{(x-x')^2}- \dfrac{(\eta(x)-\eta(x'))^2}{(x-x')^2}\mathrm{d}x'\mathrm{d}x\\
	&= \int_{-1-c}^1\int_{1}^{+\infty}\dfrac{(\eta(x+c)-1)^2}{(x-x')^2}- \dfrac{(\eta(x)-1)^2}{(x-x')^2}\mathrm{d}x'\mathrm{d}x\\
	&+ \int_{-1-c}^1\int_{-\infty}^{-1-c}\dfrac{(\eta(x+c)+1)^2}{(x-x')^2}- \dfrac{(\eta(x)+1)^2}{(x-x')^2}\mathrm{d}x'\mathrm{d}x\\
	&= \int_{-1-c}^1\dfrac{(\eta(x+c)-1)^2-(\eta(x)-1)^2}{1-x}\mathrm{d}x+\int_{-1-c}^1\dfrac{(\eta(x+c)+1)^2-(\eta(x)+1)^2}{x+1+c}\mathrm{d}x
	\end{align*}
	Notice that $\eta(x+c)=1$ for $x\in[1-c,1]$ and $\eta(x)=-1$ for $x\in[-1-c,-1]$, so we have
	\begin{align*} 
	&\ \ \ \ \int_{-1-c}^1\dfrac{(\eta(x+c)-1)^2-(\eta(x)-1)^2}{1-x}\mathrm{d}x\\
	&= \int_{-1-c}^{1-c}\dfrac{(\eta(x+c)-1)^2}{1-x}\mathrm{d}x-\int_{-1}^1\dfrac{(\eta(x)-1)^2}{1-x}\mathrm{d}x-\int_{-1-c}^{-1}\dfrac{4}{1-x}\mathrm{d}x\\
	&= \int_{-1}^1\dfrac{(\eta(x)-1)^2}{1+c-x}-\dfrac{(\eta(x)-1)^2}{1-x}\mathrm{d}x+4\ln 2-4\ln(2+c)
	\end{align*}
	and
	\begin{align*}
	\int_{-1-c}^1\dfrac{(\eta(x+c)+1)^2-(\eta(x)+1)^2}{x+1+c}\mathrm{d}x = \int_{-1}^1\dfrac{(\eta(x)+1)^2}{x+1}-\dfrac{(\eta(x)+1)^2}{x+1+c}\mathrm{d}x +4\ln(2+c)-4\ln 2.
	\end{align*}
	Then substituting these two formulas into $I_2$, we have 
	\begin{align}\label{eq:I2}
	I_2 &= \int_{-1-c}^1\dfrac{(\eta(x+c)-1)^2-(\eta(x)-1)^2}{1-x}\mathrm{d}x+\int_{-1-c}^1\dfrac{(\eta(x+c)+1)^2-(\eta(x)+1)^2}{x+1+c}\mathrm{d}x\nonumber\\
	&= \int_{-1}^1\dfrac{(\eta(x)-1)^2}{1+c-x}-\dfrac{(\eta(x)-1)^2}{1-x}\mathrm{d}x+4\ln 2-4\ln(2+c)\nonumber\\
	&+\int_{-1}^1\dfrac{(\eta(x)+1)^2}{x+1}-\dfrac{(\eta(x)+1)^2}{x+1+c}\mathrm{d}x +4\ln(2+c)-4\ln 2\nonumber\\
	&= \int_{-1}^1\dfrac{(\eta(x)-1)^2}{1+c-x}-\dfrac{(\eta(x)-1)^2}{1-x}\mathrm{d}x+\int_{-1}^1\dfrac{(\eta(x)+1)^2}{x+1}-\dfrac{(\eta(x)+1)^2}{x+1+c}\mathrm{d}x
	\end{align}
	A careful comparison of equation \eqref{eq:I1} and \eqref{eq:I2} shows that $I_1+2I_2=0$. Thus
	\begin{align*}
	\int_{\R}\int_{\R}\dfrac{(\eta(x+c)-\eta(x'+c))^2}{(x-x')^2}- \dfrac{(\eta(x)-\eta(x'))^2}{(x-x')^2}\mathrm{d}x'\mathrm{d}x=I_1+2I_2+I_3=0.
	\end{align*}
\end{proof}

\end{document}